\documentclass[10pt]{amsart}

\usepackage{ifpdf}
\ifpdf 
    \usepackage[pdftex]{graphicx}   
    \usepackage[pdftex,     
            plainpages=false,   
            breaklinks=true,    
            colorlinks=true,
            pdftitle=My Document
            pdfauthor=My Good Self
           ]{hyperref} 
\else 
    \usepackage{graphicx}       
    \usepackage{hyperref}       
\fi 
    
\usepackage{amsfonts,amsmath}	
\usepackage{amssymb}
\usepackage{verbatim}
\usepackage{amsopn}
\usepackage[english]{babel}
\usepackage{amsthm}
\usepackage{enumerate}
\usepackage{enumitem}
\usepackage{mathrsfs}	
\usepackage{mathtools}
\usepackage{color}
\usepackage{enumitem}
\usepackage{tikz}
\usepackage{soul}
\usepackage{url}

\theoremstyle{definition} \newtheorem{definition}{Definition}[section]
\theoremstyle{definition} \newtheorem{remark}[definition]{Remark}
\theoremstyle{plain} \newtheorem{lemma}[definition]{Lemma}
\theoremstyle{plain} \newtheorem{proposition}[definition]{Proposition}
\theoremstyle{plain} \newtheorem{theorem}[definition]{Theorem}
\theoremstyle{plain} \newtheorem{corollary}[definition]{Corollary}
\theoremstyle{definition} 
\theoremstyle{plain} 

\DeclareMathOperator*{\conv}{conv}
\DeclareMathOperator*{\conc}{conc}
\DeclareMathOperator{\card}{card}
\DeclareMathOperator{\sign}{sign}
\newcommand{\R}{\mathbb{R}}

\newcommand{\N}{\mathbb{N}}
\newcommand{\Z}{\mathbb{Z}}
\newcommand{\TV}{\text{\rm Tot.Var.}}

\DeclareMathOperator{\BV}{BV}

\newcommand{\I}{\textbf{I}}

\newcommand{\W}{\mathcal{W}}
\newcommand{\E}{\mathcal{E}}

\newcommand{\e}{\varepsilon}
\newcommand{\const}{\mathcal{O}(1)}
\newcommand{\fQ}{\mathfrak{Q}}

\newcommand{\Lip}{{\rm Lip}}
\newcommand{\feff}{\mathtt f^{\mathrm{eff}}}
\newcommand{\tcr}{\mathtt t^{\mathrm{cr}}}
\newcommand{\tcanc}{\mathtt t^{\mathrm{canc}}}
\newcommand{\tsp}{\mathtt t^{\mathrm{split}}}
\newcommand{\xsp}{\mathtt x^{\mathrm{split}}}
\newcommand{\xint}{\mathtt x^{\mathrm{int}}}
\newcommand{\tint}{\mathtt t^{\mathrm{int}}}
\newcommand{\sigmarh}{\sigma^{\mathrm{rh}}}
\newcommand{\sigmaent}{\sigma^{\mathrm{ent}}}

\newcommand{\Qtrans}{Q^\mathrm{trans}}
\newcommand{\Qcubic}{Q^\mathrm{cubic}}
\newcommand{\Qknown}{Q^{\mathrm{known}}}

\newcommand{\Atrans}{\mathtt {A}^{\rm{trans}}}
\newcommand{\Acubic}{\mathtt {A}^{\rm{cubic}}}
\newcommand{\Aquadr}{\mathtt {A}^{\rm{quadr}}}
\newcommand{\Acanc}{\mathtt {A}^{\rm{canc}}}
\newcommand{\Acr}{\mathtt {A}^{\rm{cr}}}

\numberwithin{equation}{section} 

\makeatletter
\def\grd@save@target#1{%
  \def\grd@target{#1}}
\def\grd@save@start#1{%
  \def\grd@start{#1}}
\tikzset{
  grid with coordinates/.style={
    to path={%
      \pgfextra{%
        \edef\grd@@target{(\tikztotarget)}%
        \tikz@scan@one@point\grd@save@target\grd@@target\relax
        \edef\grd@@start{(\tikztostart)}%
        \tikz@scan@one@point\grd@save@start\grd@@start\relax
        \draw[minor help lines] (\tikztostart) grid (\tikztotarget);
        \draw[major help lines] (\tikztostart) grid (\tikztotarget);
        \grd@start
        \pgfmathsetmacro{\grd@xa}{\the\pgf@x/1cm}
        \pgfmathsetmacro{\grd@ya}{\the\pgf@y/1cm}
        \grd@target
        \pgfmathsetmacro{\grd@xb}{\the\pgf@x/1cm}
        \pgfmathsetmacro{\grd@yb}{\the\pgf@y/1cm}
        \pgfmathsetmacro{\grd@xc}{\grd@xa + \pgfkeysvalueof{/tikz/grid with coordinates/major step}}
        \pgfmathsetmacro{\grd@yc}{\grd@ya + \pgfkeysvalueof{/tikz/grid with coordinates/major step}}
        \foreach \x in {\grd@xa,\grd@xc,...,\grd@xb}
        \node[anchor=north] at (\x,\grd@ya) {\pgfmathprintnumber{\x}};
        \foreach \y in {\grd@ya,\grd@yc,...,\grd@yb}
        \node[anchor=east] at (\grd@xa,\y) {\pgfmathprintnumber{\y}};
      }
    }
  },
  minor help lines/.style={
    help lines,
    step=\pgfkeysvalueof{/tikz/grid with coordinates/minor step}
  },
  major help lines/.style={
    help lines,
    line width=\pgfkeysvalueof{/tikz/grid with coordinates/major line width},
    step=\pgfkeysvalueof{/tikz/grid with coordinates/major step}
  },
  grid with coordinates/.cd,
  minor step/.initial=.2,           
  major step/.initial=1,            
  major line width/.initial=1pt,    
}
\makeatother

\title{Quadratic interaction functional for general systems of conservation laws}

\author{STEFANO BIANCHINI}
\address{Sissa, Via Bonomea, 265\\
Trieste, 34136, Italy}
\email{bianchin@sissa.it}

\author{STEFANO MODENA}
\address{Sissa, Via Bonomea, 265\\
Trieste, 34136, Italy} 
\email{smodena@sissa.it}

\thanks{This work is partially supported by the project PRIN-2012 ``Nonlinear Hyperbolic Partial Differential Equations, Dispersive and Transport Equations: theoretical and applicative aspects''.}
\thanks{Part of this paper has been written while the first author was visiting the Department of Mathematics, Penn State University at State College.}

\date{\today}

\begin{document}

\begin{abstract}
For the Glimm scheme approximation $u_\e$ to the solution of the system of conservation laws in one space dimension
\begin{equation*}
u_t + f(u)_x = 0, \qquad u(0,x) = u_0(x) \in \R^n,
\end{equation*}
with initial data $u_0$ with small total variation, we prove a quadratic (w.r.t. $\TV(u_0)$) interaction estimate, which has been used in the literature for stability and convergence results. No assumptions on the structure of the flux $f$ are made (apart smoothness), and this estimate is the natural extension of the Glimm type interaction estimate for genuinely nonlinear systems.

More precisely we obtain the following results:
\begin{itemize}
\item a new analysis of the interaction estimates of simple waves;
\item a Lagrangian representation of the derivative of the solution, i.e. a map $\mathtt x(t,w)$ which follows the trajectory of each wave $w$ from its creation to its cancellation;
\item the introduction of the characteristic interval and partition for couples of waves, representing the common history of the two waves;
\item a new functional $\mathfrak Q$ controlling the variation in speed of the waves w.r.t. time.
\end{itemize}
This last functional is the natural extension of the Glimm functional for genuinely nonlinear systems.

The main result is that the distribution $D_{tt} \mathtt x(t,w)$ is a measure with total mass $\leq \const \TV(u_0)^2$.
\end{abstract}

\keywords{Conservation laws, Interaction functionals}

\subjclass[2010]{35L65}

\maketitle

\centerline{Preprint SISSA 51/2014/MATE}

\tableofcontents

\section{Introduction}
\label{S_intro}

Consider a hyperbolic system of conservation laws
\begin{equation}
\label{cauchy}
\left\{
\begin{array}{l}
u_t + f(u)_x = 0, \\ [.5em]
u(0,x) = u_0(x),
\end{array}
\right.
\end{equation}
where $u_0 \in BV(\R,\R^n)$, $f: \R^n \to \R^n$ smooth (by \emph{smooth} we mean at least of class $C^3(\R^n,\R^n)$) and strictly hyperbolic. We are interested in the proof of an interaction estimate, quadratic w.r.t. the total variation of the initial data $u_0$, which has been considered in the literature and used to prove sharp convergence and stability results \cite{anc_mar_11_CMP, hua_jia_yan_10, hua_yang_10, anc_mar_11_DCDS}.

The quadratic estimate we are concerned with, can be easily explained in the case of a wavefront solution to \eqref{cauchy} \cite{bre_00}. Let $\{t_j\}_{j=1, \dots, J}$ be the times at which two wavefronts $w,w'$ meet; each $t_j$ can be an \emph{interaction} time if the wavefronts $w,w'$ belong to the same family and have the same sign; a \emph{cancellation} time, if $w,w'$ belong to the same family and have opposite sign; a \emph{transversal interaction} time if $w,w'$ belong to different families; a \emph{non-physical interaction} time if at least one among $w,w'$ is a non-physical wavefront (for a precise definition see Definition 3.2 in \cite{bia_mod_14}).

In a series of papers \cite{anc_mar_11_CMP, hua_jia_yan_10, hua_yang_10, anc_mar_11_DCDS} the following estimate has been discussed:
\begin{equation}
\label{E_quadrati_est}
\sum_{t_j \text{ interaction}} \frac{|\sigma(w) - \sigma(w')| |w| |w'|}{|w| + |w'|} \leq \mathcal O(1) \TV(u_0)^2.
\end{equation}
In the above formula $w,w'$ are the wavefronts which interact at time $t_j$, $\sigma(w)$ (resp. $\sigma(w')$) is the speed of the wavefront $w$ (resp. $w'$) and $|w|$ (resp. $|w'|$) is its strength. By $\const$ we denote a constant which depends only on the flux function $f$.

As it is shown in \cite{anc_mar_11_DCDS, bia_mod_13}, the proofs presented in the above papers contain glitches/missteps, which justified the publication of a new and different proof in \cite{bia_mod_13} and in \cite{bia_mod_14}.

In particular, the paper \cite{bia_mod_13} considers the simplest case at the level of \eqref{cauchy}, namely the scalar case $u \in \R$, and shows that even in this situation the analysis is already quite complicated: in fact, one has to follow the evolution of every elementary component of a wavefront, which we call \emph{wave} (see \cite[Sections 3.1 and 4.1]{bia_mod_13}, \cite[Section 2 and Definition 3.3]{bia_mod_14} or Section \ref{Ss_lagrang_def} below), an idea present also in \cite{anc_mar_11_CMP}. One of the conclusions of the analysis in \cite{bia_mod_13} is that the functional used to obtain the bound \eqref{E_quadrati_est} is non-local in time, a situation very different from the standard Glimm interaction analysis of hyperbolic systems of conservation laws.

In the second paper \cite{bia_mod_14} the authors study how the same estimate can be proved in the presence of waves of different families. For this aim, the most simple situation is considered, namely the Temple-class triangular system (see \cite{temple_83} for the definition of Temple class systems)
\begin{equation*}
\left\{ \begin{array}{l}
        u_t + \tilde f(u,v)_x = 0, \\ [.5em]
        v_t - v_x = 0,
        \end{array} \right.
\end{equation*}
with $\frac{\partial \tilde f}{\partial u} > - 1$, so that local uniform hyperbolicity is satisfied. The main novelties introduced in the proof of this case are:
\begin{enumerate}
\item the definition of an \emph{effective flux function}, which contains all the information about the ``convexity/concavity'' of each characteristic family, \cite[Section 3.4]{bia_mod_14};
\item a new choice of the weights for couple of waves, which take into account the presence of the wavefronts of the other family, \cite[Section 4.2]{bia_mod_14};
\item the construction of a \emph{tree} at each interaction: this tree describes the past history of all the waves involved in the interaction from the time of their last splitting, \cite[Section 4.3]{bia_mod_14}.
\end{enumerate}
The assumption on the triangular structure of the system allows to reduce the analysis to a non autonomous scalar PDE
\begin{equation*}
w_t + \bigg( \frac{\partial f}{\partial w}(w,v) \bigg) w_x = 0,
\end{equation*}
for some smooth ($C^3$-)function $f$ such that $\frac{\partial f}{\partial w} > -1$, thus granting several simplifications.

Aim of this paper is to complete the proof of \eqref{E_quadrati_est} in the general case. Even if a similar analysis can be done for wavefront tracking approximations, we choose to prove the quadratic estimate \eqref{E_quadrati_est} in  the case of the Glimm approximation scheme for two reasons.

First it is mathematically simpler. \\
On one hand, in fact, it requires the introduction of the \emph{quadratic amount of interaction} for two merging Riemann problems (Definition \ref{D_aquadr}), a new quantity not appeared before in the literature: we need this quantity because the Glimm restarting procedure merges Riemann problems at each time step, while wavefront tracking considers only binary interactions. The quadratic amount of interaction is interesting by itself. \\
On the other hand, once the analysis of two merging Riemann problems has been studied, the scheme proceeds flawlessly, because one does not need to study the various Riemann solver (accurate, approximate and brute, \cite{anc_mar_07_arma}).

Secondly, it is the original situation studied in \cite{anc_mar_11_CMP}, with the  aim of obtaining an explicit rate of convergence for the Glimm scheme. With the result proved here, Theorem 1 of \cite{anc_mar_11_CMP} is now correct.

After this brief introduction, we now present the main result of the paper.

\subsection{Main result}
\label{Ss_main_result}

Consider a general system of conservation laws
\begin{equation}
\label{E_gen_syst}
u_t + f(u)_x = 0,
\end{equation}
where $f: \Omega \subseteq \R^n \to \R^n$ is a smooth function (at least $C^3$), defined on a neighborhood $\Omega$ of $\check u \in \R^d$, satisfying the \emph{strict hyperbolicity assumption}, i.e. for any $u \in \Omega$, the Jacobian matrix $A(u) :=Df(u)$ has $n$ real distinct eigenvalues $\lambda_1(u) < \dots < \lambda_n(u)$. Together with \eqref{E_gen_syst} we consider the initial datum
\begin{equation}
\label{E_initial_datum}
u(0,x) := u_0(x) \in \BV(\R).
\end{equation}
As usual, denote by $\{r_k\}_k$ ($\{l_k\}_k$) the basis of right (left) eigenvalues, normalized by 
\begin{equation*}
|r_k(u)| = 1, \qquad \langle l_k(u), r_h(u) \rangle = 
\begin{cases}
1, & k= h, \\
0, & k \neq h.
\end{cases}
\end{equation*}
W.l.o.g. we assume $\check u = 0$ and $\{\lambda_k(u)\} \subseteq [0,1]$. Since we will consider only solutions with small total variation, taking values in a relatively compact neighborhood of the origin, we can also assume that all the derivatives of $f$ are bounded on $\Omega$ and that there exist constant $\hat \lambda_0, \dots, \hat \lambda_n$ such that 
\begin{equation}
\label{E_uniform_hyp}
0 < \hat \lambda_{k-1} < \lambda_k(u) < \hat \lambda_k < 1, \qquad \text{for any } u \in \Omega, \ k =1,\dots, n.
\end{equation}
When the initial datum has the particular form
\begin{equation*}
u(0,x) = 
\begin{cases}
u^L & x \geq 0, \\
u^R & x < 0,
\end{cases}
\end{equation*}
the Cauchy problem \eqref{E_gen_syst}-\eqref{E_initial_datum} is called the \emph{Riemann problem (RP) $(u^L, u^R)$}.


Let $(u^L, u^M)$, $(u^M, u^R)$ be two Riemann problems with a common state $u^M$, and consider the Riemann problem $(u^L, u^R)$. It is well known (see \cite{bre_00,daf_book_05}, or Section \ref{Ss_Rp}) that if $|u^M - u^L|, |u^R- u^M| \ll 1$, then one can solve the three Riemann problems as follows:
\begin{equation*}
u^M = T^{n}_{s'_n} \circ \dots \circ T^{1}_{s'_1} u^{L}, \qquad  
u^{R} = T^{n}_{s''_n} \circ \dots \circ T^{1}_{s''_1} u^{M}, \qquad 
u^{R} = T^{n}_{s_n} \circ \dots \circ T^{1}_{s_1} u^{L},
\end{equation*}
where for each $k = 1, \dots, n$, $s_k', s_k'',s_k \in \R$ and $(s,u) \mapsto T^k_s u$ is the map which at each left state $u$ associates the right state $T^k_s u$ such that the Riemann problem $(u, T^k_s u)$ has an entropy admissible solution made only by wavefronts with total strength $|s|$ belonging to the $k$-th family. This smooth curve is called \emph{$k$-th admissible curve} and it is parameterized such that $\frac{d}{ds} T^k_s u = r_k(u)$. We will also use the notation $T^k_s(u)$.

We are interested in studying how much the speed of the  wavefronts of the two incoming Riemann problems can change after the collision. More precisely, for each family $k$, writing for brevity (formula \eqref{E_I_def})
\begin{equation*}
\I(s) = \big[ \min\{s,0\},\max\{s,0\} \big] \setminus \{0\},
\end{equation*}
let us denote by
\begin{equation*}
\begin{array}{ll}
\sigma_k': \I(s_k) \to (\hat \lambda_{k-1},\hat \lambda_k) & \begin{array}{l} \text{the speed function of the wavefronts of the $k$-th family for} \\ \text{the Riemann problem $(u^L, u^M)$} , \end{array} \\ [1em]
\sigma_k'': s_k' + \I(s_k'') \to (\hat \lambda_{k-1},\hat \lambda_k) & \begin{array}{l} \text{the speed function of the wavefronts of the $k$-th family for} \\ \text{the Riemann problem $(u^M, u^R)$} , \end{array} \\ [1em]
\sigma_k: \I(s_k) \to (\hat \lambda_{k-1},\hat \lambda_k) & \begin{array}{l} \text{the speed function of the wavefronts of the $k$-th family for} \\ \text{the Riemann problem $(u^L, u^R)$.} \end{array}
\end{array}
\end{equation*}
and consider the $L^1$-norm of the speed difference between the waves of the Riemann problems $(u_L,u_M)$, $(u_M,u_R)$ and the outgoing waves of $(u_L,u_R)$:
\begin{equation*}
\Delta \sigma_k(u^L, u^M, u^R) := \left\{
\begin{array}{ll}
\big\| \big(\sigma_k' \cup \sigma_k''\big) - \sigma_k \big\|_{L^1(\I(s_k'+s_k'') \cap \I(s_k))} & \text{if } s_k' s_k'' \geq 0, \\ [1em]
\big\| \big(\sigma_k' \vartriangle  \sigma_k''\big) - \sigma_k \big\|_{L^1(\I(s_k'+s_k'') \cap \I(s_k))} & \text{if } s_k' s_k'' < 0,
\end{array} \right.
\end{equation*}
{where $\sigma_k' \cup \sigma_k''$ is the function obtained by piecing together $\sigma_k'$, $\sigma_k''$, while $\sigma_k' \vartriangle  \sigma_k''$ is the restriction of $\sigma_k'$ to $\I(s_k' + s_k'')$ if $|s_k'| \geq |s_k''|$ or $\sigma_k'' \llcorner_{\I(s_k' + s_k'')}$ in the other case, see formulas \eqref{E_f_cup_g}, \eqref{E_f_vartr_g}.}

Now consider a right continuous $\e$-approximate solution constructed by the Glimm scheme (see Section \ref{Ss_glimm_sol}); by simplicity, for any grid point $(i\e,m\e)$ denote by 
\begin{equation*}
\Delta \sigma_k(i\e,m\e) := \Delta \sigma_k (u^{i, m-1}, u^{i-1,m-1}, u^{i,m})
\end{equation*}
the change in speed of the $k$-th wavefronts at the grid point $(i\e,m\e)$ arriving from points $(i\e,(m-1)\e)$, $((i-1)\e,(m-1)\e)$, where $u^{j,r} := u(j\e,r\e)$. The main result of this paper is that the sum over all grid points of the change in speed is bounded by a quantity which depends only on the flux $f$ and the total variation of the initial datum and does not depend on $\e$. More precisely, the theorem we prove is the following.

\begin{theorem}
\label{T_main}
It holds
\begin{equation*}
\sum_{i = 1}^{+\infty} \sum_{m \in \Z} \Delta \sigma_k (i\e, m\e) \leq \const \TV(u_0; \R)^2,
\end{equation*}
where $\const$ is a quantity which depends only on the flux $f$.
\end{theorem}

The proof of Theorem \ref{T_main} follows a classical approach used in hyperbolic system of conservation laws in one space dimension.

We first prove a \emph{local} estimate. For the couple of Riemann problems $(u^L, u^M)$, $(u^M, u^R)$, we define the quantity
\begin{equation}
\label{E_amount_intro}
\begin{split}
\mathtt A(u^L, u^M, u^R) &:= \Atrans(u^L, u^M, u^R)  \\
&~ \quad + \sum_{h=1}^n \Big(\Aquadr_h(u^L,u^M, u^R) + \Acanc_h(u^L, u^M, u^R) + \Acubic_h(u^L, u^M, u^R) \Big),
\end{split}
\end{equation}
which we will call the \emph{global amount of interaction} of the two merging RPs $(u^L, u^M)$, $(u^M,u^R)$. Three of the terms in the r.h.s. of \eqref{E_amount_intro} have already been introduced in the literature, namely
\begin{align*}
& \Atrans(u^L, u^M, u^R) && \text{is the \emph{transversal amount of interaction} (see \cite{gli_65} and Definition \ref{D_atrans})}; \\
& \Acanc_h(u^L,u^M, u^R) && \text{is the \emph{amount of cancellation} of the $h$-th family (see Definition \ref{D_amount_canc})}; \\
& \Acubic_h(u^L, u^M, u^R) && \text{is the \emph{cubic amount of interaction} of the $h$-th family (see \cite{bia_03} and Definition \ref{D_acubic})}.
\end{align*}
The term $\Aquadr_h(u^L, u^M, u^R)$, which we will call the \emph{quadratic amount of interaction} of the $h$-th family (see Definition \ref{D_aquadr}), is introduced for the first time here. \\
The local estimate we will prove is the following: for all $k = 1, \dots, n$,
\begin{equation}
\label{E_delta_sigma_less_A}
\Delta \sigma_k(u^L, u^M, u^R) \leq \const \mathtt A(u^L, u^M, u^R).
\end{equation}
This is done in Section \ref{S_local}, Theorem \ref{T_general}.

Next we show a \emph{global} estimate, based on a new interaction potential. For any grid point $(i\e,m\e)$ define
\begin{equation*}
\mathtt A(i\e, m\e) : = \mathtt A(u^{i, m-1}, u^{i-1,m-1}, u^{i,m})
\end{equation*}
as the amount of interaction at the grid point $(i\e,m\e)$, and similarly let $\Atrans(i\e,m\e)$, $\Acanc_h(i\e,m\e)$, $\Acubic_h(i\e,m\e)$, $\Aquadr_h(i\e,m\e)$ be the transversal amount of interaction, amount of cancellation, cubic amount of interaction at the grid point $(i\e,m\e)$, respectively. \\
We will introduce a new interaction potential $\Upsilon$ with the following properties:
\begin{enumerate}
\item \label{Point_intro_1} it is uniformly bounded at time $t = 0$: in fact,
\begin{equation*}
\Upsilon(0) \leq \const \TV(u_0;\R)^2;
\end{equation*}
\item \label{Point_intro_2} it is constant on time intervals $[(i-1)\e, i\e)$;
\item \label{Point_intro_3} at any time $i\e$, it decreases at least of $ \frac{1}{2} \sum_{m \in \Z} \mathtt A(i\e,m\e)$.
\end{enumerate}
It is fairly easy to see that Points \eqref{Point_intro_1}, \eqref{Point_intro_2}, \eqref{Point_intro_3} above, together with inequality \eqref{E_delta_sigma_less_A}, imply Theorem \ref{T_main}.

The potential $\Upsilon$ is constructed as follows. We define a functional $t \mapsto \fQ(t)$, constant in the time intervals $[(i-1)\e, i\e)$ and bounded by $\const \TV(u_0;\R)^2$ at $t=0$, which satisfies the following inequality (see Theorem \ref{T_variation_fQ}):
\begin{equation}
\label{E_bound_on_fQ}
\begin{split}
\fQ(i\e) - \fQ(&(i-1)\e) \leq - \sum_{m \in \Z} \sum_{h=1}^n \Aquadr(i\e,m\e) + \const \TV(u_0; \R) \sum_{m \in \Z} \mathtt A(i\e, m\e) \\
&= - \big( 1 - \const \TV(u_0; \R) \big) \sum_{m \in \Z} \sum_{h=1}^n \Aquadr(i\e,m\e) \\
&~ \quad + \const \TV(u_0; \R) \sum_{m \in \Z} \Atrans(i\e,m\e) \\
&~ \quad + \const \TV(u_0; \R) \sum_{m \in \Z} \sum_{h=1}^n \bigg( \Acanc_h(i\e,m\e) + \Acubic_h(i\e,m\e) \bigg). 
\end{split}
\end{equation}
It is well known (see \cite{gli_65}, \cite{bia_03} and Section \ref{Ss_known_fcn}) that there exists a uniformly bounded, decreasing potential $\Qknown(t)$ such that at each time $i\e$ 
\begin{equation}
\label{E_bound_on_Qknown}
\sum_{m \in \Z}\bigg[\Atrans(i\e, m\e) + \sum_{h=1}^n \Big(\Acanc_h(i\e, m\e) + \Acubic_h(i\e, m\e)\Big)\bigg]  
\leq
\Qknown((i-1)\e) - \Qknown \big(i\e\big).
\end{equation}
Hence, it is straightforward to see from \eqref{E_bound_on_fQ}, \eqref{E_bound_on_Qknown} that we can find a constant $C$ big enough, such that the potential 
\begin{equation*}
\Upsilon(t) := \fQ(t) + C \Qknown(t)
\end{equation*}
satisfies Properties (1)-(3) above, provided that $\TV(u_0; \R) \ll 1$.

\subsection{Structure of the paper}
\label{Ss_structure_paper}

The paper is organized as follows.

In Section \ref{S_preliminaries} we recall some preliminary results, already present in the literature, which we will use thoroughly in the paper. \\
In Section \ref{Ss_Rp} we show how an entropic self-similar solution to the Riemann problem $(u^L, u^R)$ is constructed, focusing our attention especially on the proof of the existence of the elementary curves of a fixed family. Even if the main ideas are similar to the standard proof found in the literature (see for instance \cite{bia_bre_05}), we need to use a slightly different distance among elementary curves, see \eqref{E_distance_gamma}: this because we need sharper estimate on the variation of speed. \\
In Section \ref{Ss_def_aoi} we recall the definitions of some quantities which in some sense measure how strong is the interaction between two contiguous Riemann problems which are merging and we present some related results: these quantities are the \emph{transversal amount of interaction} (Definition \ref{D_atrans}), the \emph{cubic amount of interaction} (Definition \ref{D_acubic}), the \emph{amount of cancellation} (Definition \ref{D_amount_canc}) and the \emph{amount of creation} (Definition \ref{D_amount_canc}). \\
In Section \ref{Ss_glimm_sol} we review how a family of approximate solutions $\{u_\e(t,x)\}_{\e>0}$, to the Cauchy problem \eqref{E_gen_syst}-\eqref{E_initial_datum} is constructed by means of the Glimm scheme. \\
Finally in Section \ref{Ss_known_fcn} we recall the definitions of some already known functionals, which provide a uniform-in-time bound on the spatial total variation of the approximate solution $u_\e$.

Section \ref{S_local} is devoted to prove the \emph{local} part of the proof of Theorem \ref{T_main}, as explained in Section \ref{Ss_main_result}. In particular we will consider two contiguous Riemann problems $(u^L, u^M)$, $(u^M, u^R)$ which are merging, producing the Riemann problem $(u^L, u^R)$ and we will introduce a \emph{global amount of interaction $\mathtt A$}, which bounds
\begin{enumerate}
\item the $L^\infty$-distance between the $u$-component of the elementary curves before and after the interaction;
\item the $L^1$-distance between the speed of the wavefronts before and after the interaction, i.e. the $\sigma$-component of the elementary curves;
\item the $L^1$-distance between the second derivatives of the reduced fluxes, before and after the interaction.
\end{enumerate}
This is done in Theorem \ref{T_general}. \\
In Section \ref{Ss_def_quadr_amoun_int} we introduce the first novelty of the paper, i.e. the notion of \emph{quadratic amount of interaction $\Aquadr_k$} for waves belonging to the $k$-th characteristic family. Adding this new functional to the classical transversal amount of interaction, cubic amount of interaction and amount of cancellation we obtain the total amount of interaction, Definition \ref{D_quadr_amoun_inter}. \\
The important fact about the total amount of interaction is that it bounds the variation of the elementary curves when two Riemann problems are merged, Theorem \ref{T_general} of Section \ref{Ss_distance_curves}. The proof of Theorem \ref{T_general} is given in the next three subsections. In Section \ref{Sss_basic_est} we prove some basic estimates related to translations of the starting point of the curve which solves a Riemann problem and to changes of the length of such a curve; in Section \ref{Sss_elemnt_inter} we consider the situation in which each of the two incoming Riemann problems is solved by a wavefront of a single family $k$; in Section \ref{Sss_conclu_th_gen} we conclude the proof of Theorem \ref{T_general}, piecing together the analysis of the previous two cases.

In Section \ref{S_lagr_repr} we define the notion of \emph{Lagrangian representation} of an approximate solution $u_\e$ (Section \ref{Ss_lagrang_def}) obtained by the Glimm scheme to the Cauchy problem \eqref{E_gen_syst}-\eqref{E_initial_datum}, and we explicitly construct a Lagrangian representation satisfying some useful additional properties (Section \ref{Ss_explic_lagr}). In the final Section \ref{Ss_further_def} we introduce some notions related to the Lagrangian representation; in particular, the notion of \emph{effective flux $\feff_k(t)$ of the $k$-th family at time $t$} will have a major role in the next sections.

Starting with Section \ref{S_analysis_wave_coll} we enter in the heart of our construction. \\
In Section \ref{Ss_wave_pack} we define an equivalence relation on the waves of the $k$-th characteristic family, which allows to work only with finitely many equivalence classes instead of a continuum of waves. \\
In Section \ref{W_waves_collision} we introduce the notion of characteristic interval of waves $\mathcal I(\bar t,w,w')$ for any couple of waves $w,w'$: roughly speaking, the idea is that the waves outside this interval are not essential in computing the strength of the interaction of $w,w'$. In order to define this interval, we introduce the notion of \emph{pair of waves $(w,w')$ which have already interacted} and \emph{pair of waves $(w,w')$ which have never interacted} at time $\bar t$. \\
In Section \ref{Ss_partition} we give a partition $\mathcal P(\bar t,w,w')$ of this interval: the elements of this partition are waves with the same past history, from the moment in which one of them is created.

Now we have all the tools we need to define the functional $\fQ_k$ for $k =1,\dots,n$ and to prove that it satisfies the inequality \eqref{E_bound_on_fQ}, thus obtaining the \emph{global} part of the proof of Theorem \ref{T_main}. This is done in the final Section \ref{S_fQ}. \\
In Section \ref{Ss_def_fQ} we give the definition of $\fQ_k$, using the intervals $\mathcal I(\bar t,w,w')$ and their partitions $\mathcal P(\bar t,w,w')$. The idea is to adopt the form of the analogous functional $\fQ$ introduced first in \cite{bia_mod_13} and then further developed in \cite{bia_mod_14}: it is an integral among all couples of waves $w,w'$ of a weight $\mathfrak q_k(t,w,w')$ obtained (roughly speaking) by
\begin{equation*}
\mathfrak q_k(t,w,w') \approx \frac{\text{difference in speed of $w,w'$ for the Riemann problem in $\mathcal I(t,w,w')$ with the flux $\feff_k$}}{\text{length of the interval $\mathcal I(t,w,w')$}}.
\end{equation*}
The precise form is slight more complicated, in order to minimize the oscillations of $\mathfrak q_k$ in time. \\
In Section \ref{Ss_statement_main_thm} we state the main theorem of this last part of the paper, i.e. inequality \eqref{E_bound_on_fQ} and we give a sketch of its proof, which will be written down in all its details in the next subsections. \\
In Sections \ref{Ss_one_created}, we first study the increase of $\fQ$ for couples of waves one of which has been created at the time of interaction. The main result (Proposition \ref{P_prop_crea_fin}) is that this increase is controlled by the total amount of interaction times the total variation of the solution. \\
In Section \ref{Ss_both_conserved}, we study how the weight varies for couple of waves which are not interacting at the given time, and again in Theorem \ref{T_decreasing} we prove that it is controlled by the total amount of interaction times the total variation of the solution. \\
Finally in Section \ref{Ss_interacting}, Theorem \ref{T_decreasing} computes the decaying part of the functional, due to couple of waves which are involved in the same Riemann problem: it turns out that the decay is exactly $\Aquadr_k$ (plus the total amount of interaction times the total variation of the solution), and this concludes the proof of Theorem \ref{T_main}.

We conclude the paper with Appendix \ref{S_appendix}, where we collect some elementary results on convex envelopes and secant lines used thoroughly in the paper.

\subsection{Notations}
\label{Ss_notation}
\begin{itemize}
\item For any $s \in \R$, define
\begin{equation}
\label{E_I_def}
\I(s) := 
\begin{cases}
(0,s] & \text{if } s\geq 0, \\
[s, 0) & \text{if } s < 0.
\end{cases}
\end{equation}
\item Let $X$ be any set and let $f: \I(s') \to X$, $g: s'+\I(s'') \to X$;
    \begin{itemize}
    \item if $s's'' \geq 0$ and $f(s') = g(s')$, define
    \begin{equation}
    \label{E_f_cup_g}
    f \cup g: \I(s'+s'') \to X, \qquad
    \big(f \cup g \big) (x) :=
    \begin{cases}
    f(x) & \text{if } x \in \I(s'), \\
    g(x) & \text{if } x \in s'+\I(s'');
    \end{cases}
    \end{equation}
    \item if $s's'' <0$, define
    \begin{equation}
    \label{E_f_vartr_g}
    f \vartriangle g: \I(s'+s'') \to X, \qquad
    (f \vartriangle g)(x) := 
    \begin{cases}
    f(x) & \text{if } |s'| \geq |s''|,\ x \in \I(s'+s''), \\
    g(x-s') & \text{if } |s'| < |s''|,\ x \in \I(s'+s'').
    \end{cases}
    \end{equation}
    \end{itemize}
\item For a continuous real valued function $f$, we denote its convex envelope in the interval $[a,b]$ as $\displaystyle \conv_{[a,b]} f$.
\item Given a totally ordered set $(A, \preceq)$, we define a partial pre-ordering on $2^A$ setting, for any $I, J \subseteq A$,
\[
I \prec J  \text{ if and only if for any $a \in I, b \in J$ it holds } a \prec b.
\]
We will also write $I \preceq J$ if either $I \prec J$ or $I = J$, i.e. we add the diagonal to the relation, making it a partial ordering.
\item The $L^\infty$  norm of a map $g: [a,b] \to \R^n$ will be denoted either by $\|g\|_\infty$ or by $\|g\|_{L^\infty([a,b])}$, if we want to stress the domain of $g$; similar notation for the $L^1$-norm.
\item Given a $C^1$ map $g: \R \to \R$ and an interval $I \subseteq \R$, possibly made by a single point, let us define
\begin{equation*}
\sigmarh(g, I) := 
\begin{cases}
\dfrac{g(\sup I) - g(\inf I)}{\sup I - \inf I}, & \text{if $I$ is not a singleton}, \\
\dfrac{dg}{du}(I), & \text{if $I$ is a singleton}.
\end{cases}
\end{equation*}
\end{itemize}

\section{Preliminary results}
\label{S_preliminaries}

In this section we recall some preliminary results, already present in the literature, which we will use in the next sections.

In Section \ref{Ss_Rp} we show how an entropic self-similar solution to the Riemann problem $(u^L, u^R)$ is constructed, focusing our attention especially on the proof of the existence of the elementary curves of a fixed family. Even if the main ideas are similar to the standard proof found in the literature (see for instance \cite{bia_bre_05}), we need to use a slightly different distance among elementary curves, see \eqref{E_distance_gamma}: this because we need sharper estimate on the variation of speed. 

In Section \ref{Ss_def_aoi} we recall the definitions of some quantities which in some sense measure how strong is the interaction between two contiguous Riemann problems which are joining and we present some related results: these quantities are the \emph{transversal amount of interaction} (Definition \ref{D_atrans}), the \emph{cubic amount of interaction} (Definition \ref{D_acubic}), the \emph{amount of cancellation} (Definition \ref{D_amount_canc}) and the \emph{amount of creation} (Definition \ref{D_amount_canc}).

In Section \ref{Ss_glimm_sol} we review how a family of approximate solutions $\{u_\e(t,x)\}_{\e>0}$, to the Cauchy problem \eqref{E_gen_syst}-\eqref{E_initial_datum} is constructed by means of the Glimm scheme.

Finally in Section \ref{Ss_known_fcn} we recall the definitions of some already known functionals, which provide a uniform-in-time bound on the spatial total variation of the approximate solution $u_\e$.

\subsection{Entropic self-similar solution to the Riemann problem}
\label{Ss_Rp}

We describe here the method developed in \cite{bia_bre_05}, with some minimal variations, to construct a solution to the Riemann problem $(u^L, u^R)$, i.e. the system \eqref{E_gen_syst} together with the initial datum
\begin{equation}
\label{E_rp}
u(0,x) = 
\begin{cases}
u^L, & x\geq 0, \\
u^R, & x < 0,
\end{cases}
\end{equation}
provided that $|u^R - u^L|$ is small enough.
First we present the algorithm used to build the solution to the Riemann problem $(u^L, u^R)$ and then we focus our attention on the construction of the elementary curves of a fixed family.

\subsubsection{Algorithm for solving the Riemann problem}
\label{Sss_algor_RP}

The following proposition holds.

\begin{proposition}
\label{P_RIP_constr}
For all $\delta_2 > 0$ there exists $0 < \delta_1 < \delta_2$ such that for any $u^L, u^R \in B(0, \delta_1)$ the Riemann problem \eqref{E_gen_syst}, \eqref{E_rp} admits a unique, self-similar, right continuous, vanishing viscosity solution, taking values in $B(0, \delta_2)$.  
\end{proposition}

\begin{proof}[Sketch of the proof]
\smallskip
\noindent \textit{Step 1.} For any index $k \in \{1,\dots,n\}$, through a Center Manifold technique, one can find a neighborhood of the point $(0, 0, \lambda_k(0))$ of the form
\begin{equation*}
\mathcal{D}_k := \big\{(u,v_k,\sigma_k) \in \R^n \times \R \times \R \ \big| \ |u| \leq \rho, |v_k| \leq \rho, |\sigma_k - \lambda_k(0)| \leq \rho \big\}
\end{equation*}
for some $\rho>0$ (depending only on $f$) and a smooth vector field 
\[
\tilde r_k: \mathcal{D}_k \to \R^n, \qquad \tilde r_k = \tilde r_k(u,v_k,\sigma_k),
\]
satisfying 
\begin{equation}
\label{E_generalized_eigenvector}
\tilde r_k(u,0,\sigma_k) = \tilde r_k(u), \qquad
\bigg|\frac{\partial \tilde r_k}{\partial \sigma_k}(u,v_k,\sigma_k)\bigg| \leq \const \big|v_k\big|.
\end{equation}
We will call $\tilde r_k$ the \emph{$k$-generalized eigenvector}.
The characterization of $\tilde r_k$ is that
\begin{equation*}
\mathcal D_k \ni (u,v_k,\sigma_k) \mapsto \big(u,v_k \tilde r_k,\sigma_k \big) \in \R^n \times \R^n \times \R
\end{equation*}
is a parameterization of a center manifold near the equilibrium $(0,0,\lambda_k(0)) \in \mathcal D_k$ for the ODE of traveling waves
\begin{equation*}
\big( A(u) - \sigma \mathbb I \big) u_x = u_{xx} \qquad \Longleftrightarrow \qquad \left\{ \begin{aligned}
u_x &= v \\
v_x &= (A(u) - \sigma \mathbb I) v \\
\sigma_x &= 0
\end{aligned} \right.
\end{equation*}
where $A(u) = Df(u)$, the Jacobian matrix of the flux $f$, and $\mathbb I$ is the identity $n \times n$ matrix.

Associated to the generalized eigenvectors, we can define smooth functions $\tilde \lambda_k: \mathcal{D}_k \to \R$ by
\[
\tilde \lambda_k(u,v_k,\sigma_k) := \big \langle l_k(u), A(u) \tilde r_k(u,v_k,\sigma_k) \big \rangle.
\]
We will call $\tilde \lambda_k$ the \emph{$k$-generalized eigenvalue}. By \eqref{E_generalized_eigenvector} and the definition of $\tilde \lambda_k$, we can get
\begin{equation}
\label{E_delambdasudev}
\tilde \lambda_k(u,0,\sigma_k) = \lambda_k(u), \qquad
\bigg|\frac{\partial \tilde \lambda_k}{\partial \sigma_k}(u,v_k,\sigma_k)\bigg| \leq \const |v_k|.
\end{equation}
For the construction of the generalized eigenvectors and eigenvalues and the proof of \eqref{E_generalized_eigenvector}, \eqref{E_delambdasudev}, see Section 4 of \cite{bia_bre_05}.

\smallskip
\noindent \textit{Step 2.} By a fixed point technique one can now prove that there exist $\delta, \eta >0$ (depending only on $f$), such that for any
\begin{equation*}
k \in \{1,\dots,n\}, \qquad u^L \in B(0, \rho/2), \qquad 0 \leq s < \eta,
\end{equation*}
there is a curve
\begin{equation*}
\begin{array}{ccccc}
\gamma &:& [0,s] &\to& \mathcal{D}_k \\ [.5em]
&& \tau &\mapsto& \gamma(\tau) = (u(\tau), v_k(\tau),\sigma_k(\tau))
\end{array}
\end{equation*}
such that $u,v_k \in C^{1,1}([0,s])$, $\sigma_k \in C^{0,1}([0,s])$, it takes values in $B(u^L, \delta) \times B(0, \delta) \times B(\lambda_k(u^L), \delta)$
and it is the unique solution to the system
\begin{equation}
\label{E_fixed_pt}
\left\{
\begin{aligned}
u(\tau) &= u^L + \int_0^\tau \tilde r_k(\gamma(\varsigma)) d\varsigma \\
v_k(\tau) &= f_k (\gamma; \tau) - \conv_{[0,s]} f_k (\gamma; \tau) \\
\sigma_k(\tau) &= \frac{d}{d\tau} \conv_{[0,s]} f_k (\gamma; \tau)
\end{aligned}
\right.
\end{equation}
where
\begin{equation}
\label{E_reduced_flux}
f_k(\gamma;\tau) := \int_0^\tau \tilde \lambda_k(\gamma(\varsigma)) d\varsigma.
\end{equation}
In the case $s<0$ a completely similar result holds, replacing the convex envelope with the concave one.

\noindent If we want to stress the dependence of the curve $\gamma$ on $u^L$ and $s$ we will use the notation
\begin{equation*}
\gamma(u^L, s)(\tau) = \Big( u(u^L, s)(\tau), v_k(u^L, s)(\tau), \sigma_k(u^L, s)(\tau) \Big).
\end{equation*}

Even if the existence and uniqueness of such a curve is known, we give a proof in Section \ref{Ss_proof_step_2}, since we need to use a definition of distance among curves slightly different from the one in \cite{bia_bre_05}.
\smallskip

\noindent \textit{Step 3.} Once the curve $\gamma$ solving \eqref{E_fixed_pt} is found, one can prove the following lemma. 

\begin{lemma}
\label{P_rp_k_family}
Let $\gamma: [0,s] \to \mathcal{D}_k$, $\gamma(u^L,s)(\tau) = (u(\tau),v_k(\tau),\sigma_k(\tau))$, be the Lipschitz curve solving the system \eqref{E_fixed_pt} and define the right state $u^R := u(s)$. Then the unique, right continuous, vanishing viscosity solution of the Riemann problem $(u^L,u^R)$ is the function
\begin{equation*}
\omega(t,x) := 
\begin{cases}
u^L & \text{if } x/t \leq \sigma_k(0), \\
u(\tau) & \text{if } x/t = \max\{\xi \in [0,s] \ | \ x/t = \sigma_k(\xi)\}, \\
u^R & \text{if } x/t \geq \sigma_k(s).
\end{cases}
\end{equation*}
\end{lemma}

For the proof see Lemma 14.1 in \cite{bia_bre_05}. The case $s<0$ is completely similar.

\smallskip

\noindent \textit{Step 4.} By previous step, for any $k \in \{1,\dots,n\}$, $u^L \in B(0, \rho/2)$, there is a curve
\begin{equation*}
(-\eta,\eta) \ni s \quad \mapsto \quad T^k_s(u^L) := u(u^L,s)(s) \in B(u^L,\delta) \subseteq \R^n
\end{equation*}
such that the Riemann problem $(u^L, T^k_s(u^L))$ admits a self similar solution consisting only of $k$-waves. 
\begin{lemma}
The curve $s \mapsto T^k_s(u^L)$ is Lipschitz continuous and 
\begin{equation}
\label{E_T_lipschitz}
\underset{s \to 0}{\mathrm{ess\,lim}} \frac{d T^k_s(u^L)}{ds} = r_k(u^L).
\end{equation}
\end{lemma}
For the proof see Lemma 14.3 in \cite{bia_bre_05}.

\smallskip
\noindent \textit{Step 5}. Thanks to \eqref{E_T_lipschitz}, the solution to the general Riemann problem \eqref{E_gen_syst}, \eqref{E_rp} can be now constructed following a standard procedure (see for example \cite[Chapter 9]{daf_book_05}). One considers the composite map
\begin{equation*}
\begin{array}{ccccc}
T(u^L) &:& (-\eta, \eta)^n &\to& \R^n \\ [.5em]
&& (s_1, \dots, s_n) &\mapsto& T(u^L)(s_1,\dots,s_n) := T^n_{s_n} \circ \dots \circ  T^1_{s_1} (u^L)
\end{array}
\end{equation*}
By \eqref{E_T_lipschitz} and a version of the Implicit Function Theorem valid for Lipschitz continuous maps, $T(u^L)$ is a one-to-one mapping from a neighborhood of the origin in $\R^n$ onto a neighborhood of $u^L$. Hence, for all $u^R$ sufficiently close to $u^L$ (uniformly w.r.t. $u^L \in B(0, \rho/2)$), one can find unique values $s_1, \dots, s_n$ such that $T(u^L)(s_1,\dots,s_n) = u^R$. \\
In turn, this yields intermediate states $u_0 = u^L, u_1, \dots, u_n = u^R$ such that each Riemann problem with data $(u_{k-1}, u_k)$ admits a vanishing viscosity solution $\omega_k = \omega_k(t,x)$ consisting only of $k$-waves. By the assumption \eqref{E_uniform_hyp} we can define the solution to the general Riemann problem $(u^L, u^R)$ by 
\begin{equation*}
\omega(t,x) = \omega_k(t,x) \qquad \text{for } \hat \lambda_{k-1} < \frac{x}{t} < \hat \lambda_k.
\end{equation*}
Therefore we can choose $\delta_1, \delta_2 \ll 1$ such that if $u^L, u^R \in B(0, \delta_1)$, the Riemann problem $(u^L, u^R)$ can be solved as above and the solution takes values in $B(0, \delta_2)$, thus concluding the proof of the proposition.
\end{proof}

\subsubsection{Proof of Step 2}
\label{Ss_proof_step_2}

We now explicitly prove that the system \eqref{E_fixed_pt} admits a $C^{1,1} \times C^{1,1} \times C^{0,1}$-solution, i.e. we prove Step 2 of the previous algorithm, using the Contraction Mapping Principle. As we said, we need a proof slightly different from the one in \cite{bia_bre_05}: in fact, even if the general approach is the same, the distance used among curves is suited for the type of estimates we are interested in.

Fix an index $k = 1, \dots, n$ and consider the space
\[
X := C^0\big([0,s];\R^n\big) \times C^0\big([0,s]\big) \times L^1\big([0,s]\big)
\]
A generic element of $X$ will be denoted by $\gamma = (u, v_k, \sigma_k)$. The index $k$ is just to remember that we are solving a RP with wavefronts of the $k$-th family. Endow $X$ with the norm
\begin{equation}
\label{E_distance_gamma}
\|\gamma\|_\dagger = \big\| (u,v_k,\sigma_k) \big\|_{\dagger} := \|u\|_\infty + \|v_k\|_\infty + \|\sigma_k\|_1
\end{equation}
and consider the subset
\[
\begin{split}
\Gamma_k(u^L,s) := \bigg\{\gamma = (u,v_k,\sigma_k) \in X : &~ u,v_k \text{ are Lipschitz and } \Lip(u) + \Lip(v_k) \leq L,   \\ 
&~ u(0) = u^L, v_k(0) = 0,\\
&~ |u(\tau) - u^L| \leq \delta, |v_k(\tau)| \leq \delta \text{ for any } \tau \in [0,s], \\
&~ |\sigma_k(\tau) - \lambda_k(u^L)| \leq \delta \text{ for $\mathcal L^1$}-a.e.  \tau \in [0,s] \bigg\}
\end{split}
\]
for $u^L \in B(0, \rho)$ and $L, s,\delta >0$ which will be chosen later. Clearly $\Gamma_k(u^L,s)$ is a closed subset of the Banach space $X$
and thus it is a complete metric space. Denote by $D$ the distance induced by the norm $\| \cdot \|_\dagger$ on $X$.

Consider now the transformation
\begin{equation*}
\begin{array}{ccccc}
\mathcal T &:& \Gamma_k(u^L,s) &\to& L^\infty([0,s],\R^{n+2})  \\ [.5em]
&& \gamma &\mapsto& \hat \gamma := \mathcal T \gamma
\end{array}
\end{equation*}
defined by the formula
\begin{equation*}
\left\{
\begin{aligned}
\hat u(\tau) &:= u^L + \int_0^\tau \tilde r_k(\gamma(\varsigma)) d\varsigma, \\
\hat v_k(\tau) &:= f_k (\gamma; \tau) - \conv_{[0,s]} f_k (\gamma; \tau), \\
\hat \sigma_k(\tau) &:= \frac{d}{d\tau} \conv_{[0,s]} f_k (\gamma; \tau),
\end{aligned}
\right.
\end{equation*}
where $f_k$ has been defined in \eqref{E_reduced_flux}. Observe that, since $\tilde \lambda_k$ is uniformly bounded near $(u^L, 0, \lambda_k(u^L))$, it turns out that $f_k(\gamma)$ is a Lipschitz function for any $\gamma \in \Gamma_k(u^L, s)$, and thus by Theorem \ref{convex_fundamental_thm}, Point \eqref{convex_fundamental_thm_1}, ${\displaystyle \conv_{[0,s]}} f_k(\gamma) : [0,s] \to \R$ is Lipschitz and its derivative is in $L^\infty([0,s],\R)$.

\begin{lemma}
\label{L_contract_dagger}
There exist $L, \eta, \delta >0$ depending only on $f$ such that for all fixed $u^L \in B(0, \rho/2)$ it holds:
\begin{enumerate}
\item for any $|s| < \eta$, $\mathcal{T}$ is a contraction from $\Gamma_k(u^L, s)$ into itself, more precisely
\[
\big\|\mathcal{T}(\gamma) - \mathcal{T}(\gamma')\big\|_\dagger \leq \frac{1}{2} \big\| \gamma - \gamma' \big\|_\dagger;
\]
\item \label{E_point_w_contr_dag} if $\bar \gamma = (\bar u, \bar v_k, \bar \sigma_k)$ is the fixed point of $\mathcal T$, then $\bar u, \bar v_k \in C^{1,1}$ and $\bar \sigma_k \in C^{0,1}$.
\end{enumerate}
\end{lemma}

Clearly Point \eqref{E_point_w_contr_dag} above yields Step 2 of the proof of Proposition \ref{P_RIP_constr}. 

\begin{proof}
\textit{Step 1}. We first prove that if $\gamma \in \Gamma_k (u^L,s)$, then $\hat \gamma = (\hat u,\hat v_k, \hat \sigma_k) := \mathcal{T}(\gamma)\in \Gamma(u^L,s)$, provided $L \gg 1$, $\eta \ll 1$, while $\delta$ will be fixed in the next step.

Clearly $\hat u(0) = u^L$ and $\hat v_k(0) = 0$. Moreover $\hat u, \hat v_k$ are Lipschitz continuous and $\hat \sigma_k$ is in $L^\infty([0,s])$.

Let us prove the uniform estimate on the Lipschitz constants. First we have
\begin{equation*}
\big|\hat u(\tau_2) - \hat u(\tau_1) \big|\leq \int_{\tau_1}^{\tau_2} \big|\tilde r_k(\gamma(\varsigma))\big| d\varsigma \leq \|\tilde r_k\|_\infty |\tau_2 - \tau_1| \leq \frac{L}{2} \big|\tau_2 - \tau_1 \big|.
\end{equation*}
if the constant $L$ is big enough. For $v_k$ it holds
\begin{equation*}
\begin{split}
\big| \hat v_k(\tau_2) - \hat v_k(\tau_1)\big| &\leq \big|f_k(\gamma;\tau_2) - f_k(\gamma;\tau_1) \big| + \Big| \conv_{[0,s]}f_k(\gamma;\tau_2) - \conv_{[0,s]}f_k(\gamma;\tau_1) \Big| \\
\text{(by Theorem \ref{convex_fundamental_thm}, Point \eqref{convex_fundamental_thm_1})} &\leq 2 \ \Lip\big(f_k(\gamma)\big) \big| \tau_2 - \tau_1\big| 
\leq 2 \ \big\| \tilde \lambda_k \big\|_\infty \big| \tau_2 - \tau_1 \big| \\
&\leq \frac{L}{2} \big| \tau_2 - \tau_1 \big|,
\end{split}
\end{equation*}
if $L$ is big enough. 

Finally let us prove tat the curve $\gamma$ remains uniformly close to the point $(u^L, 0, \lambda_k(u^L))$. First we have
\begin{equation*}
\big| \hat u(\tau) - u^L| \leq \int_0^\tau \big| \tilde r_k(\gamma(\varsigma)) \big| d\varsigma \leq \|\tilde r_k\|_\infty |\tau| \leq \|\tilde r_k\|_\infty |\eta| \leq \delta,
\end{equation*}
if $\eta \ll 1$. Next it holds
\begin{equation*}
\begin{split}
\big| \hat v_k(\tau)\big| &\leq \big| \tilde f_k(\gamma; \tau)\big| + \big| \conv_{[0,s]} \tilde f_k(\gamma; \tau) \big| \\
&\leq \int_0^\tau \bigg| \frac{d \tilde f_k(\gamma; \varsigma)}{d\varsigma} \bigg| d\varsigma \ + \ \int_0^\tau \bigg| \frac{d}{d\varsigma}\conv_{[0,s]} \tilde f_k(\gamma; \varsigma)\bigg| d\varsigma\\ 
\text{(by Proposition \ref{P_estim_diff_conv})} &\leq 2 \int_0^\tau \big| (\tilde \lambda_k \circ \gamma )(\varsigma) \big| d \varsigma \\
&\leq \|\tilde \lambda_k\|_\infty \big|\eta\big| \leq \delta,
\end{split}
\end{equation*}
if $\eta \ll 1$. Finally, before making the computation for $\sigma_k$, let us observe that
\begin{equation*}
\begin{split}
\bigg| \frac{d \tilde f_k(\gamma;\tau)}{d\tau} - \lambda_k(\tau) \bigg| &\leq \Big| \tilde \lambda_k\big(u(\tau), v_k(\tau), \sigma_k(\tau)\big) - \tilde \lambda_k\big(u^L, 0, \sigma_k(\tau)\big)\Big| \\
&\leq \const \Big( \big| u(\tau) - u^L\big| + \big|v_k(\tau)\big| \Big) \\
&\leq \const \Big( \Lip(u) + \Lip(v_k) \Big) \eta \\
&\leq \const L \eta \leq \delta,
\end{split}
\end{equation*}
if $\eta \ll 1$.

We have thus proved that we can choose $L \gg 1$, $\eta \ll 1$ such that $\hat \gamma := \mathcal{T}(\gamma) \in \Gamma(u^L, s)$. Notice that the choice of $L,\eta$ depends only on $f$ and $\delta$ and not on $u^L \in B(0,\delta/2)$.

\smallskip

\noindent \textit{Step 2.} 
We now prove that the map $\mathcal{T}: \Gamma(u^L, s) \to \Gamma(u^L, s)$ is a contraction. Let $\gamma = (u,v_k,\sigma_k), \gamma'= (u',v_k',\sigma_k') \in \Gamma(u^L,s)$ and set
\begin{equation*}
\hat \gamma = (\hat u,\hat v_k, \hat \sigma_k) := \mathcal{T}(\gamma), \qquad \hat \gamma'= (\hat u', \hat v_k', \hat \sigma_k') := \mathcal{T}(\gamma').
\end{equation*}

It holds for the component $u$
\begin{equation}
\label{E_contraction_u}
\begin{split}
\big| \hat u(\tau) - \hat u'(\tau) \big| &\leq
\int_0^\tau \big| \tilde r_k(\gamma(\varsigma)) - \tilde r_k(\gamma'(\varsigma))\big| d\varsigma \\
&\leq
\int_0^\tau \Bigg( \bigg\|\frac{\partial \tilde r_k}{\partial u} \bigg\|_\infty \big| u(\varsigma) - u'(\varsigma) \big| + \bigg\| \frac{\partial \tilde r_k}{\partial v_k} \bigg\|_\infty \big| v_k(\varsigma) - v'_k(\varsigma) \big| + \bigg\| \frac{\partial \tilde r_k}{\partial \sigma_k} \bigg\|_\infty \big| \sigma_k(\varsigma) - \sigma_k'(\varsigma) \big| \Bigg) d\varsigma
\\
\text{(by \eqref{E_generalized_eigenvector})} &\leq \const \int_0^\tau \Big(\big| u(\varsigma) - u'(\varsigma)\big| + \big|v_k(\varsigma) - v'_k(\varsigma) \big| + \delta \big| \sigma_k(\varsigma) - \sigma_k'(\varsigma) \big| \Big) d\varsigma \\
&\leq \const D(\gamma,\gamma') ( \eta\ + \delta ) \\
&\leq \frac{1}{2} D(\gamma,\gamma'),
\end{split}
\end{equation}
if $\eta, \delta \ll 1$.

For the component $v_k$ we have
\begin{equation}
\label{E_contraction_v}
\begin{split}
\big| \hat v_k(\tau) - \hat v_k'(\tau) \big| &\leq \big| \tilde f_k(\gamma; \tau) - \tilde f_k(\gamma'; \tau) \big| + \big| \conv_{[0,s]} \tilde f_k(\gamma; \tau) - \conv_{[0,s]} \tilde f_k(\gamma'; \tau) \big| \\
\text{(by Proposition \ref{P_estim_diff_conv})} &\leq 2 \big\| \tilde f_k(\gamma) - \tilde f_k(\gamma') \big\|_\infty \\
&\leq 2 \bigg \| \frac{d \tilde f_k(\gamma)}{d\tau} - \frac{d \tilde f_k(\gamma')}{d\tau} \bigg\|_1 \\
&= 2 \int_0^s \Big| \tilde \lambda_k\big(\gamma(\tau)\big) - \tilde \lambda_k\big(\gamma'(\tau)\big) \Big| d\tau \\
\text{(using \eqref{E_delambdasudev} as in \eqref{E_contraction_u})} &\leq \frac{1}{2} D(\gamma, \gamma'),
\end{split}
\end{equation}
if $\eta, \delta \ll 1$.

Finally
\begin{equation*}
\big\| \hat \sigma_k - \hat \sigma'_k\big\|_1 \leq \int_0^s \Big| \tilde \lambda_k\big(\gamma(\varsigma)\big) - \tilde \lambda_k\big( \gamma'(\tau)\big) \Big| \leq \frac{1}{2} D(\gamma, \gamma'), 
\end{equation*}
if $\eta, \delta \ll 1$ using \eqref{E_delambdasudev} as in \eqref{E_contraction_v}.

Hence $\mathcal T$ is a contraction from $\Gamma(u^L, s)$ into itself, with contractive constant equal to $1/2$, provided $\eta, \delta \ll 1$. 

\smallskip

\noindent \textit{Step 3.} Let us now prove the second part of the lemma, concerning the regularity of the fixed point $\bar \gamma = (\bar u, \bar v_k, \bar \sigma_k)$.

Fix a big constant $M >0$ and let $A(M) \subseteq \Gamma(u^L,s)$ be the subset which contains all the curves $\gamma = (u,v_k,\sigma_k)$ such that $\sigma_k$ is Lipschitz with $\Lip(\sigma_k) \leq M$. Clearly $A(M)$ is non empty and closed in $X$. We claim that $\mathcal T(A(M)) \subseteq A(M)$ if $M$  is big enough and $\delta, \eta$ small enough. This will conclude the proof of the lemma.

Let $\gamma \in A(M)$ and, as before, $\hat \gamma = (\hat u, \hat v_k, \hat \sigma_k) := \mathcal T(\gamma)$. Let us first compute the Lipschitz constant of $\frac{d \tilde f_k(\gamma)}{d\tau}$:
\begin{equation*}
\begin{split}
\bigg| \frac{d \tilde f_k(\gamma; \tau_2)}{d \tau} - \frac{d \tilde f_k(\gamma; \tau_1)}{d \tau} \bigg| &= \Big| \tilde \lambda_k(\gamma(\tau_2)) - \tilde \lambda_k (\gamma (\tau_1))\Big| \\
&\leq \const \Big| \Big( \Lip(u) + \Lip(v_k) + \delta\Lip(\sigma_k) \Big) \Big| |\tau_2 - \tau_1| \\
&\leq \const (2L + \delta M) \big| \tau_2 - \tau_1\big| \\
\leq M \big| \tau_2 - \tau_1 \big|,
\end{split}
\end{equation*}
if $0 < \delta \ll 1$ and $M \gg 1$. Now observe that 
\begin{equation*}
\begin{split}
\big| \hat \sigma_k (\tau_2) - \hat \sigma_k(\tau_1) \big| &\leq \bigg| \frac{d}{d\tau} \conv_{[0,s]}\tilde f_k(\gamma; \tau_2) - \frac{d}{d\tau} \conv_{[0,s]} \tilde f_k(\gamma; \tau_1)\bigg| \\
&\leq \Lip\bigg(\frac{d}{d\tau}  \conv_{[0,s]} \tilde f_k(\gamma) \bigg) \big| \tau_2 - \tau_1\big| \\
\text{(by Theorem \ref{convex_fundamental_thm}, Point \eqref{convex_fundamental_thm_3})} &\leq \Lip\bigg(\frac{d\tilde f_k(\gamma)}{d\tau}\bigg) \big| \tau_2 - \tau_1\big| \\
&\leq M \big|\tau_2 - \tau_1 \big|.
\end{split}
\end{equation*}
Hence $\hat \sigma$ is Lipschitz and $\Lip(\hat \sigma) \leq M$, i.e. $\hat \gamma \in A(M)$, if $\delta \ll 1$ and $M \gg 1$.
\end{proof}

\subsection{Definition of amounts of interaction, cancellation and creation}
\label{Ss_def_aoi}

In this section we introduce some quantities, namely the \emph{transversal amount of interaction (Definition \ref{D_atrans}), the \emph{cubic amount of interaction} (Definition \ref{D_acubic}), the \emph{amount of cancellation} (Definition \ref{D_amount_canc}) and the \emph{amount of creation} (Definition \ref{D_amount_canc}), which} measure how strong is the interaction between two contiguous Riemann problems and we present some related results. All these quantities are already present in the literature. In Section \ref{S_local}, we will introduce one more quantity, which will be called the \emph{quadratic amount of interaction}, and (as far as we know) has never been defined before. 

Consider two contiguous Riemann problem, whose resolution in elementary waves is
\begin{equation}
\label{E_two_riemann}
u^M = T^{n}_{s'_n} \circ \dots \circ T^{1}_{s'_1} u^{L}, \qquad  
u^{R} = T^{n}_{s''_n} \circ \dots \circ T^{1}_{s''_1} u^{M},
\end{equation}
and the Riemann problem obtained by joining them, resolved by
\begin{equation*}
u^{R} = T^{n}_{s_n} \circ \dots \circ T^{1}_{s_1} u^{L}.
\end{equation*}

\noindent Let $f_k'$, $f_k''$ be the two reduced fluxes of the $k$-th waves $s_k'$, $s_k''$ 
for the Riemann problems $(u^L,u^M)$, $(u^M,u^R)$, respectively: more precisely, $\tilde f_k'$ ($\tilde f_k''$) is computed by  \eqref{E_reduced_flux} where, for $k=1,\dots,n$, $\gamma_k'$ ($\gamma_k''$) is the solutions of \eqref{E_fixed_pt} with length $s_k'$ ($s_k''$) and initial point 
\begin{equation*}
\begin{array}{rl}
u^L \ \text{for} \ k=1, \quad & T^{k-1}_{s'_{k-1}} \circ \dots \circ T^{1}_{s'_1} u^{L}, \ \text{for} \ k \geq 2, \\ [1em]
\Bigg( u^M \ \text{for} \ k=1, \quad & T^{k-1}_{s''_{k-1}} \circ \dots \circ T^{1}_{s''_1} u^{M}, \ \text{for} \ k \geq 2. \Bigg)
\end{array}
\end{equation*}
Since \eqref{E_fixed_pt} is invariant when we add a constant to $f_k$, having in mind to perform the merging operations \eqref{E_f_cup_g}, \eqref{E_f_vartr_g}, we can assume that $f''_k$ is defined in $s_k' + \I(s_k'')$ and satisfies $f_k''(s_k') = f_k'(s_k')$.

\begin{definition}
\label{D_atrans}
The quantity
\begin{equation*}
\mathtt {A}^{\textrm{trans}}(u^L, u^M, u^R) := \sum_{1 \leq h < k \leq n} |s_k'||s_h''|
\end{equation*}
is called the \emph{transversal amount of interaction} associated to the two Riemann problems \eqref{E_two_riemann}.
\end{definition}

\begin{definition}
\label{D_acubic}
For $s'_k > 0$, we define \emph{cubic amount of interaction of the $k$-th family} for the two Riemann problems $(u^L, u^M)$, $(u^M, u^R)$ as follows: 
\begin{enumerate}
\item if $s_k''\geq 0$, 
\begin{equation*}
\begin{split}
\Acubic_k(u^L, u^M, u^R) := \int_0^{s_k'} &\Big[\conv_{[0,s_k']} f_k'(\tau) - \conv_{[0,s_k'+s_k'']} \big(f_k' \cup f_k''\big) (\tau) \Big]
d\tau \\
&~ + \int_{s_k'}^{s_k'+s_k''} 
\Big[ \conv_{[s_k', s_k'']} f_k''(\tau) - \conv_{[0, s_k' + s_k'']} \big(f_k' \cup f_k''\big) (\tau)
 \Big]
d\tau;
\end{split}
\end{equation*}
\item if $-s_k' \leq s_k'' < 0$
\begin{equation*}
\begin{split}
\Acubic_k(u^L, u^M, u^R) := \int_0^{s_k'+s_k''} &\Big[\conv_{[0, s_k'+s_k'']} f_k'(\tau) - \conv_{[0,s_k']} f_k'(\tau) \Big] d\tau \\
&~ + \int_{s_k' + s_k''}^{s_k'} \Big[ \conc_{[s_k' + s_k'', s_k']} f_k'(\tau) - \conv_{[0,s_k']} f_k'(\tau) \Big] d\tau;
\end{split}
\end{equation*}
\item if $s_k'' < -s_k'$, 
\begin{equation*}
\begin{split}
\Acubic_k(u^L, u^M, u^R) := \int_{s_k' + s_k''}^{0}& \Big[ \conc_{[s_k'+s_k'', s_k']}f_k''(\tau) - \conc_{[s_k'+s_k'', 0]} f_k''(\tau) \Big] d\tau \\
&~ + \int_{0}^{s_k'} \Big[\conc_{[s_k'+s_k'', s_k']} f_k''(\tau) - \conv_{[0, s_k']} f_k''(\tau) \Big] d\tau.
\end{split}
\end{equation*}
\end{enumerate}
Similar definitions can be given if $s_k'<0$, replacing convex envelopes with concave.
\end{definition}

\begin{remark}
\label{R_def_cub_posit}
The previous definition is exactly Definition 3.5 in \cite{bia_03}, where it is also shown that the terms appearing in the above definition are non negative.
\end{remark}

The following definition is standard.

\begin{definition}
\label{D_amount_canc}
The \emph{amount of cancellation of the $k$-th family} is defined by 
\begin{equation*}
\Acanc_k(u^L, u^M, u^R) := 
\begin{cases}
0 & \text{if } s_k' s_k'' \geq 0, \\
\min\{|s_k'|, |s_k''|\} & \text{if } s_k' s_k'' < 0,
\end{cases}
\end{equation*}
while the \emph{amount of creation of the $k$-th family} is defined by
\begin{equation*}
\Acr_k(u^L, u^M, u^R) := \Big[|s_k| - |s_k' + s_k''| \Big]^+.
\end{equation*}
\end{definition}

The following theorem is proved in \cite{bia_03}.

\begin{theorem}
\label{T_bianchini}
It holds
\begin{equation*}
\sum_{k=1}^n \big| s_k - (s_k' + s_k'') \big| \leq 
\const \bigg[ \Atrans(u^L, u^M, u^R) + \sum_{k=1}^n \Acubic_k(u^L, u^M, u^R) \bigg].
\end{equation*}
\end{theorem}

\noindent As an immediate consequence, we obtain the following corollary.

\begin{corollary}
\label{C_aocr}
It holds
\begin{equation*}
\Acr_k(u^L, u^M, u^R) \leq \Atrans(u^L, u^M, u^R) + \sum_{h=1}^n \Acubic_h(u^L, u^M, u^R).
\end{equation*}
\end{corollary}

%
%
%
%

\subsection{Glimm scheme solution to a general system of conservation laws}
\label{Ss_glimm_sol}

Let us now briefly recall how an approximate solution to \eqref{E_gen_syst}is constructed by the Glimm scheme. Fix $\e >0$.

To construct an approximate solution $u_\e = u_\e(t,x)$ to the Cauchy problem \eqref{E_gen_syst}-\eqref{E_initial_datum}, we consider a grid in the $(t,x)$ plane having step size $\Delta t = \Delta x = \e$, with nodes in the points
\[
P_{i,m} = (t_i,x_m) := (i \e, m \e), \qquad i \in \N, \ m \in \Z. 
\]
Moreover we shall need a sequence of real numbers $\vartheta_1, \vartheta_2, \vartheta_3, \dots$, uniformly distributed over the interval $[0,1]$. This means that, for every $\lambda \in [0,1]$, the percentage of points $\vartheta_i, 1 \leq i \leq N$, which fall inside $[0,\lambda]$ should approach $\lambda$ as $N \rightarrow \infty$:
\begin{equation*}
\lim_{N \rightarrow \infty} \frac{\card\{i \ | \ 1 \leq i \leq N, \vartheta_i \in [0, \lambda]\}}{N} = \lambda \hspace{1cm} \text{for each } \lambda \in [0,1]. 
\end{equation*}

At time $t=0$, the Glimm algorithm starts by considering an approximation $u_\e(0, \cdot)$ of the initial datum $u(0, \cdot)$, which is constant on each interval of the form $[x_{m-1}, x_m)$ and such that its measure derivative has compact support. We shall take (remember that $u(0, \cdot)$ is right continuous)
\begin{equation*}
u_\e(0,x) = u(0,x_m) \hspace{1cm} \text{for all } x \in [x_m, x_{m+1}).
\end{equation*}
Notice that clearly
\begin{equation*}
\TV(u_\e(0); \R) \leq \TV(u(0); \R).
\end{equation*}

For times $t>0$ sufficiently small, the solution $u_\e = u_\e(t,x)$ is obtained by solving the Riemann problems corresponding to the jumps of the initial approximation $u_\e(0, \cdot)$ at the nodes $x_m$. By \eqref{E_uniform_hyp}, the solutions to the Riemann problems do not overlap on the time interval $[0, \e)$, and thus $u_\e(t)$ can be prolonged up to $t = \e$.

At time $t_1 = \e$ a restarting procedure is adopted: the function $u_\e(t_1-, \cdot)$ is approximate by a new function $u_\e(t_1,\cdot)$ which is piecewise constant, having jumps exactly at the nodes $x_m = m \e$. Our approximate solution $u_\e$ can now be constructed on the further time interval $[\e, 2\e)$, again by piecing together the solutions of the various Riemann problems determined by the jumps at the nodal points $x_m$. At time $t_2 = 2\e$, this solution is again approximated by a piecewise constant function, and the procedure goes on.

A key aspect of the construction is the restarting procedure. At each time $t_i = i \e$, we need to approximate $u_\e(t_i-, \cdot)$ with a piecewise constant function $u_\e(t_i, \cdot)$ having jumps precisely at the nodal points $x_m$. This is achieved by a random sampling technique. More precisely, we look at the number $\vartheta_i$ in the uniformly distributed sequence. On each interval $[x_{m-1},x_m)$, the old value of our solution at the intermediate point $\vartheta_i x_m + (1-\vartheta_i) x_{m-1}$ becomes the new value over the whole interval:
\[
u_\e(t_i, x) := u_\e \big( t_i-, (\vartheta_i x_m + (1-\vartheta_i) x_{m-1}) \big) \hspace{1cm} \text{for all } x \in [x_{m-1},x_m).
\]

One can prove that, if the initial datum $u(0,\cdot)$ has sufficiently small total variation, then an approximate solution can be constructed by the above algorithm for all times $t \geq 0$ and moreover
\begin{equation}
\label{E_bd_on_TV}
\TV(u_\e(t), \R) \leq \const \TV(u(0), \R).
\end{equation}
The bound \eqref{E_bd_on_TV} on the total variation of the function at times $t \geq 0$ can be obtained by standard arguments (see \cite{gli_65}, \cite{bre_00}), using Proposition \ref{P_V_equiv_TV} and Theorem \ref{T_cubic_potential} below.

For our purposes, it is convenient to redefine $u_\e$ inside the open strips $(i\e, (i+1)\e) \times \R$ as follows:
\begin{equation*}
u_\e(t,x):= 
\begin{cases}
u_\e((i+1)\e,m\e) & \text{if } m\e \leq x < m\e + t - i\e, \\
u_\e(i\e,m\e) & \text{if } m\e + t - i\e \leq x < (m+1)\e.
\end{cases}
\end{equation*}
In this way, $u_\e(t,\cdot)$ becomes a compactly supported, piecewise constant function for each time $t \geq 0$ with jumps along piecewise linear curves passing through the nodes $(i\e,m\e)$.

To conclude this section, let us introduce some notations which we will be used in the next. For any grid point $(i\e, m\e)$, $i \geq 0$, $m \in \Z$, set
\begin{equation*}
u^{i,m}:= u_\e(i\e, m\e),
\end{equation*}
and assume that the Riemann problem $(u^{i,m-1}, u^{i,m})$ is solved by
\begin{equation*}
u^{i,m} = T^{n}_{s_n^{i,m}} \circ \dots \circ T^{1}_{s_1^{i,m}} u^{i,m-1};
\end{equation*}
moreover denote by 
\begin{equation*}
\sigma_k^{i,m} :\I(s_k^{i,m}) \to \R, \qquad k=1, \dots, n,
\end{equation*}
the speed function of the $k$-th wavefront solving the Riemann problem $(u^{i,m-1}, u^{i,m})$.

Let us introduce also the following notation for the transversal and cubic amounts of interaction and for the amount of cancellation related to the two Riemann problems $(u_{i,m-1}, u_{i-1,m-1})$, $(u_{i-1,m-1}, u_{i,m})$ which interact at grid point $(i\e,(m-1)\e)$: 
\begin{equation*}
\Atrans(i\e,m\e) := \Atrans(u_{i,m-1}, u_{i-1,m-1}, u_{i,m}),
\end{equation*}
and for $k=1,\dots,n$,
\begin{equation*}
\begin{array}{c}
\Acubic_k(i\e,m\e) := \Acubic_k(u_{i,m-1}, u_{i-1,m-1}, u_{i,m}),  \\ [1em]
\Acanc_k(i\e, m\e) := \Acanc_k(u_{i,m-1}, u_{i-1,m-1}, u_{i,m}), \\ [1em]
\Acr_k(i\e, m\e) := \Acr_k(u_{i,m-1}, u_{i-1,m-1}, u_{i,m}).
\end{array}
\end{equation*}

\subsection{Known Lypapunov functionals}
\label{Ss_known_fcn}

In this section we recall the definitions and the basic properties of three already well known functionals, namely the \emph{total variation functional}, the functional introduced by Glimm in \cite{gli_65} which controls the transversal amounts of interaction and the functional introduced by the first author in \cite{bia_03}, which controls the cubic amounts of interaction.

\begin{definition}
Define the \emph{total variation along curves} as
\begin{equation*}
V(t) := \sum_{k=1}^n \sum_{m \in \Z} |s_k^{i,m}|, \qquad \text{ for any } t \in [i\e, (i+1)\e).
\end{equation*}
Define the \emph{transversal interaction functional} as
\begin{equation*}
\Qtrans(t):= 
\sum_{k=1}^n \sum_{h=1}^{k-1} \sum_{m > m'} |s_k^{i,m'}||s_h^{i,m}|, \qquad \text{ for any } t \in [i\e, (i+1)\e).
\end{equation*}
Define the \emph{cubic interaction functional} as
\begin{equation*}
\Qcubic(t):=  \sum_{k=1}^n \sum_{m,m' \in \Z} \int_{\I(s_k^{i,m})} \int_{\I(s_k^{i,m'})} \big| \sigma_k^{i,m}(\tau) - \sigma_k^{i,m'}(\tau') \big| d\tau' d\tau.
\end{equation*}
\end{definition}

\begin{remark}
\label{R_local_vs_global}
Notice that the three functionals $t \mapsto V(t), \Qtrans(t), \Qcubic(t)$ are local in time, i.e. their value at time $\bar t$ depends only on the solution $u_\e(\bar t)$ at time $\bar t$ and not on the solution at any other time $t \neq \bar t$. On the contrary, the functionals $\fQ_k, k =1, \dots, n$ we will introduce in Section \ref{S_fQ} to bound the difference in speed of the wavefronts before and after the interactions are non-local in time, i.e. their definition requires the knowledge of the whole solution in $\R^+ \times \R$.
\end{remark}

The following statements hold: for the proofs, see \cite{bre_00}, \cite{bia_03}.

\begin{proposition}
\label{P_V_equiv_TV}
There exists a constant $C>0$, depending only of the flux $f$, such that for any time $t \geq 0$
\begin{equation*}
\frac{1}{C} \TV(u(t)) \leq V(t) \leq C \TV(u(t)).
\end{equation*}
\end{proposition}

\begin{theorem}
\label{T_cubic_potential}
The following hold:
\begin{enumerate}
\item the functionals $t \mapsto V(t), \Qtrans(t), \Qcubic(t)$ are constant on each interval $[i\e, (i+1)\e)$;
\item they are bounded by powers of the $\TV(u(t))$ as follows:
\begin{equation*}
\begin{split}
V(t) &\leq \const \TV(u(t)), \\ \Qtrans(t) &\leq \const \TV(u(t))^2, \\ \Qcubic(t) &\leq \const \TV(u(t))^3;
\end{split}
\end{equation*}
\item there exist constants $c_1, c_2, c_3 > 0$, depending only on the flux $f$, such that for any $i \in \N$, defining
\begin{equation*}
\Qknown(t) := c_1 V(t) + c_2 \Qtrans(t) + c_3 \Qcubic(t),
\end{equation*}
it holds
\begin{equation*}
\begin{split}
\sum_{m \in \Z}\bigg[\Atrans(i\e, m\e) + \sum_{k=1}^n \Big(\Acanc_k(i\e, m\e) &+ \Acubic_k(i\e, m\e)\Big)\bigg]  
\leq
\Qknown((i-1)\e) - \Qknown(i\e).
\end{split}
\end{equation*}
\end{enumerate}
\end{theorem}

\section{Interaction between two merging Riemann problems}
\label{S_local}

This section is devoted to prove the \emph{local} part of the proof of Theorem \ref{T_main}, as explained in the introduction. In particular we will consider two contiguous Riemann problems $(u^L, u^M)$, $(u^M, u^R)$ which are merging, producing the Riemann problem $(u^L, u^R)$ and we will introduce a \emph{global amount of interaction $\mathtt A$}, which bounds
\begin{enumerate}
\item the $L^\infty$-distance between the $u$-component of the elementary curves before and after the interaction;
\item the $L^1$-distance between the speed of the wavefronts before and after the interaction, i.e. the $\sigma$-component of the elementary curves;
\item the $L^1$-distance between the second derivatives of the reduced fluxes, before and after the interaction.
\end{enumerate}
This is done in Theorem \ref{T_general}. 

Let us first introduce some notations, as in Section \ref{Ss_def_aoi}. Consider two contiguous Riemann problem 
\begin{equation}
\label{E_two_riemann_1}
u^M = T^{n}_{s'_n} \circ \dots \circ T^{1}_{s'_1} u^{L}, \qquad  
u^{R} = T^{n}_{s''_n} \circ \dots \circ T^{1}_{s''_1} u^{M},
\end{equation}
and the Riemann problem obtained joining them,
\begin{equation*}
u^{R} = T^{n}_{s_n} \circ \dots \circ T^{1}_{s_1} u^{L}.
\end{equation*}
In particular the incoming curves are 
\begin{equation}
\label{E_incoming_curves}
\begin{array}{ccc}
\gamma_1' = (u_1', v_1', \sigma_1') := \gamma_1(u^L, s_1'), \quad & \gamma_k' = (u_k', v_k', \sigma_k') := \gamma_k\big(u_{k-1}'(s_{k-1}'), s_k'\big) & \text{for } k = 2, \dots, n, \\ [1em]
\gamma_1'' =  (u_1'', v_1'', \sigma_1'') := \gamma_1(u^M, s_1''), \quad & \gamma_k'' = (u_k'', v_k'', \sigma_k'')  := \gamma_k\big(u_{k-1}''(s_{k-1}''), s_k''\big) & \text{for } k = 2, \dots, n,
\end{array}
\end{equation}
while the outcoming ones are
\begin{align}
\label{E_outcoming_curves}
\gamma_1 &= (u_1, v_1, \sigma_1) := \gamma_1(u^L, s_1), & \gamma_k &= (u_k, v_k, \sigma_k) := \gamma_k\big(u_{k-1}(s_{k-1}), s_k\big) \quad \text{for } k = 2, \dots, n.
\end{align}

\noindent We will denote by $f_k', f_k'', f_k$ the reduced fluxes associated by \eqref{E_reduced_flux} to $\gamma_k', \gamma_k'', \gamma_k$ respectively; for simplicity, we will assume that $\gamma_k''$ and $f_k''$ are defined on $s_k' + \I(s_k'')$, instead of $\I(s_k'')$ and $f_k''(s_k') = f_k'(s_k')$: indeed, it is clear that adding a constant to $\tilde f_k$ does not vary system \eqref{E_fixed_pt}. 

Fix an index $k \in \{1, \dots n\}$ and consider the points (Figure \ref{F_el_curves_after_trans})
\begin{equation*}
\begin{array}{lll}
u^L_1 := u^L, \quad & u^L_k := T^{k-1}_{s''_{k-1}} \circ T^{k-1}_{s'_{k-1}} \circ \dots \circ T^1_{s''_1} \circ T^1_{s'_1} u^L, \quad & k \geq 2 \\ [1em]
u^M_k:= T^k_{s'_k} u^L_k, \quad & u^R_k:= T^k_{s''_k} u^M_k, \quad & k = 1,\dots,n.
\end{array}
\end{equation*}
By definition, the Riemann problem between $u^L_k$ and $u^M_k$ is solved by a wavefront of the $k$-th family with strength $s_k'$ and the Riemann problem between $u^M_k$ and $u^R_k$ is solved by a wavefront of the $k$-th family with strength $s''_k$. Denote by $\tilde \gamma_k' = (\tilde u_k', \tilde v_k', \tilde \sigma_k')$ the curve which solves the Riemann problem $[u^L_k, u^M_k]$ and by $\tilde f_k'$ the associated reduced flux (see \eqref{E_reduced_flux}). \\
Similarly, let $\tilde \gamma_k'' = (\tilde u_k'', \tilde v_k'', \tilde \sigma_k'')$ be the curve solving the Riemann problem $[u^M_k, u^R_k]$ and let $\tilde f_k''$ be the associated reduced flux. Clearly, $\tilde \gamma_k'$, $\tilde f_k'$ are defined on $\I(s_k)$, while, since we are going to perform the patching \eqref{E_f_cup_g}, \eqref{E_f_vartr_g}), we will assume as above that $\tilde \gamma_k''$ and $\tilde f_k''$ are defined on $s_k' + \I(s_k'')$ (instead of $\I(s_k'')$) and that $\tilde f_k''(s_k') = \tilde f_k'(s_k')$..

\begin{figure}
\begin{tikzpicture}

\draw[fill] (1,1) circle [radius = 0.04];
\draw (1,1) to [out=45,in=195] (3,1.2);
\node[below] at (2,1.8) {$\gamma_1' = \tilde \gamma_1'$};
\node[below] at (3,1.2) {$u_1^M$};

\draw[fill] (3,1.2) circle [radius = 0.04];
\draw (3,1.2) to [out=105,in=340] (3.2,3);
\node[right] at (3.2,2) {$\gamma_2'$};

\draw[fill] (3.2,3) circle [radius = 0.04];
\draw (3.2,3) to [out=45,in=250] (5,4);
\node[above] at (4,3.4) {$\gamma_3'$};

\draw[fill] (5,4) circle [radius = 0.04];
\draw (5,4) to [out=45,in=195] (7,4.2);
\node[above] at (6.2,4.2) {$\gamma_1''$};

\draw[fill] (7,4.2) circle [radius = 0.04];
\draw (7,4.2) to [out=105,in=340] (7,5.3);
\node[right] at (7.1,4.6) {$\gamma_2''$};

\draw[fill] (7,5.3) circle [radius = 0.04];
\draw (7,5.3) to [out=45,in=250] (9,7);
\node[above left] at (8,6) {$\gamma_3''$};

\draw[fill] (9,7) circle [radius = 0.04];

\draw[red] (3,1.2) to [out=15,in=195] (5,1.4);
\node[below] at (3.8,1.3) {$\tilde \gamma_1''$};

\draw[fill] (5,1.4) circle [radius = 0.04];
\draw[red] (5,1.4) to [out=105,in=340] (5.4,3);
\node[right] at (5.4,2.2) {$\tilde \gamma_2'$};
\node[below] at (5,1.4) {$u_1^R = u_2^L$};

\draw[fill] (5.4,3) circle [radius = 0.04];
\draw[red] (5.4,3) to [out=160,in=340] (5.2,5);
\node[right] at (5.4,3.6) {$\tilde \gamma_2''$};
\node[right] at (5.5,3) {$u_2^M$};

\draw[fill] (5.2,5) circle [radius = 0.04];
\draw[red] (5.2,5) to [out=45,in=190] (6.6,6.2);
\node[above left] at (5.8,5.8) {$\tilde \gamma_3'$};
\node[left] at (5.2,5) {$u_2^R = u_3^L$};

\draw[fill] (6.6,6.2) circle [radius = 0.04];
\draw[red] (6.6,6.2) to [out=10,in=220] (8,7);
\node[above left] at (7.3,6.3) {$\tilde \gamma_3''$};
\draw[fill] (8,7) circle [radius = 0.04];
\node[above left] at (6.6,6.2) {$u_3^M$};
\node[above] at (8,7) {$u_3^R$};

%
%
%
%
%

\node[below] at (1,1.1) {$u^L_1 = u^L$};
\node[above left] at (5,4) {$u^M$};
\node[above right] at (9,7) {$u^R$};

\end{tikzpicture}
\caption{Elementary curves of two interacting Riemann problems before and after transversal interactions.}
\label{F_el_curves_after_trans}
\end{figure}
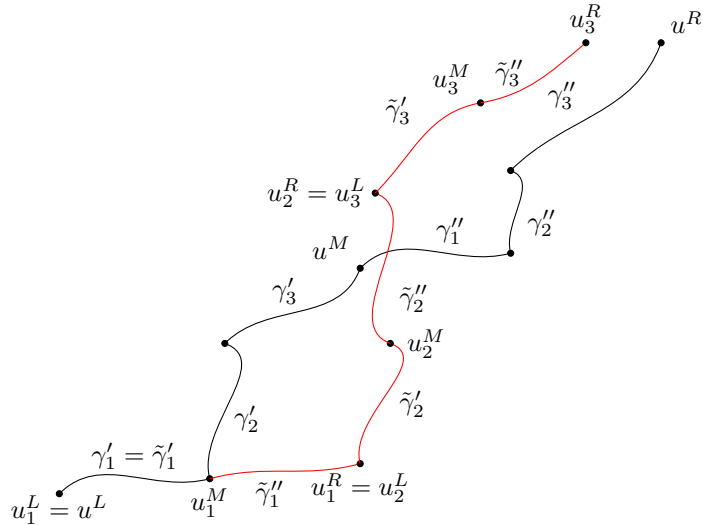

\subsection{Definition of the quadratic amount of interaction}
\label{Ss_def_quadr_amoun_int}

In this section we define a new quantity, namely the \emph{quadratic amount of interaction}, which will be used to bound the $L^1$-norm of the difference of speed between incoming and outgoing wavefronts. 

\begin{definition}
\label{D_aquadr}
If $s_k' s_k'' \geq 0$, we define the \emph{quadratic amount of interaction of the $k$-family} associated to the two Riemann problems \eqref{E_two_riemann_1} by
\begin{equation*}
\Aquadr_k(u^L, u^M, u^R) := 
\begin{cases}
\tilde f_k'(s_k') - \conv_{[0,s_k'+s_k'']} \big(\tilde  f_k' \cup \tilde f_k'' \big) (s_k') & \text{if } s_k' >0 , \ s_k'' >0, \\
\conc_{[s_k'+s_k'',0]} \big(\tilde f_k' \cup \tilde f_k'' \big)(s_k') - \tilde f_k'(s_k') & \text{if } s_k' < 0 ,\  s_k'' <0, \\
0 & \text{if } s_k' s_k'' \leq 0. 
\end{cases}
\end{equation*}
\end{definition}

\begin{definition}
\label{D_quadr_amoun_inter}
We define the \emph{total amount of interaction} associated to the two Riemann problems \eqref{E_two_riemann_1} as 
\begin{equation*}
\mathtt A(u^L, u^M, u^R) := \Atrans(u^L, u^M, u^R)  + \sum_{h=1}^n \Big(\Aquadr_h(u^L,u^M, u^R) + \Acanc_h(u^L, u^M, u^R) + \Acubic_h(u^L, u^M, u^R) \Big).
\end{equation*}
\end{definition}

\subsection{Distance between curves and between reduced fluxes}
\label{Ss_distance_curves}

The main result of this section is the following.
\begin{theorem}
\label{T_general}
For any $k = 1, \dots, n$,
\begin{itemize}
\item if $s_k' s_k'' \geq 0$, then
\begin{equation*}
\left.
\begin{array}{c}
\big\| (u_k' \cup u_k'') - u_k \big\|_{L^\infty(\I(s_k'+s_k'') \cap \I(s_k))} \\ [.7em]
\big\| (\sigma_k' \cup \sigma_k'') - \sigma_k \big\|_{L^1(\I(s_k'+s_k'') \cap \I(s_k))} \\ [.7em]
{\displaystyle \bigg\|\bigg(\frac{d^2f_k'}{d\tau^2} \cup \frac{d^2 f_k''}{d\tau^2} \bigg) - \frac{d^2 f_k}{d\tau^2} \bigg\|_{L^1(\I(s_k'+s_k'') \cap \I(s_k))}} \\
\end{array}
\right\}
\leq  \const \mathtt A(u^L, u^M, u^R);
\end{equation*}
\item if $s_k' s_k'' < 0$, then
\begin{equation*}
\left.
\begin{array}{c}
\big\|(u_k' \vartriangle  u_k'') - u_k\big\|_{L^\infty(\I(s_k'+s_k'') \cap \I(s_k))} \\ [.7em]
\big\|(\sigma_k' \vartriangle  \sigma_k'') - \sigma_k\big\|_{L^1(\I(s_k'+s_k'') \cap \I(s_k))} \\ [.7em]
{\displaystyle \bigg\|\bigg(\frac{d^2f_k'}{d\tau^2} \vartriangle  \frac{d^2 f_k''}{d\tau^2}\bigg) - \frac{d^2 f_k}{d\tau^2} \bigg\|_{L^1(\I(s_k'+s_k'') \cap \I(s_k))}} \\
\end{array}
\right\}
\leq  \const \mathtt A(u^L, u^M, u^R);
\end{equation*}
\end{itemize}
\end{theorem}

The proof of the theorem is a bit technical, but it is not difficult. It is essentialy based on the fact that we have used the $L^1$-norm of the difference in speed into the definition of the distance $D$ in Section \ref{Ss_proof_step_2} and on the computation in Lemma \ref{L_interaction_same_family}, which relates the difference in speed of the incoming wavefronts with the quadratic amount of interaction introduced in Definition \ref{D_aquadr}. 

The proof of Theorem \ref{T_general} is given in the following three subsections. In the first subsection we prove some basic estimates related to translations of the starting point of the curve which solves a Riemann problem and to changes of the length of such a curve; in the second subsection we consider the situation in which each of the two incoming Riemann problems is solved by a wavefront of a single family $k$; in the third subsection we conclude the proof of Theorem \ref{T_general}, piecing together the analysis of the previous two cases.

\subsubsection{Basic estimates}
\label{Sss_basic_est}

First of all, we prove the following three lemmas. The first is just an observation on the second derivative of the reduced flux. The second and the third provide estimates on the distance between curves and between reduced fluxes when varying the starting point of the curve or its length.

\begin{lemma}
\label{L_second_derivative_f}
Let $\gamma = (u, v_k, \sigma_k) := \gamma_k(u_0, s)$ the solution of the \eqref{E_fixed_pt}, i.e. the curve solving the Riemann problem with left state $u_0$ and length $s$. Denote by $f_k$ the reduced flux associated to $\gamma_k$ as in \eqref{E_reduced_flux}, i.e. 
\begin{equation*}
f_k(\tau) := \int_0^{\tau} \tilde \lambda_k(u(\varsigma), v_k(\varsigma), \sigma_k(\varsigma)) d\varsigma.
\end{equation*}
Then it holds
\begin{equation}
\label{E_second_derivative_f}
\frac{d^2 f_k}{d\tau^2} (\tau) = 
\frac{\partial \tilde \lambda_k}{\partial u}\big(\gamma_k(\tau)\big) \tilde r_k\big(\gamma_k(\tau)\big) 
+ \frac{\partial \tilde \lambda_k}{\partial v_k}\big(\gamma_k(\tau) \big) \bigg[\tilde \lambda_k\big(\gamma_k(\tau)\big) - \frac{d \conv_{[0,s]} f_k}{d\tau} (\tau) \bigg].
\end{equation}
\end{lemma}


\begin{proof}
Observe that 
\begin{equation*}
\frac{\partial \tilde \lambda_k}{\partial \sigma_k}\big(\gamma_k(\tau)\big) \frac{d \sigma_k}{d \tau}(\tau) = 0 \quad \text{ for $\mathcal L^1$-a.e. } \tau \in [0,s].
\end{equation*}
Namely, if $\frac{d \sigma_k(\tau)}{d\tau} \neq 0$ for some $\tau$, then $v_k(\tau) = 0$ and thus, by \eqref{E_delambdasudev},
\begin{equation*}
\bigg|\frac{\partial \tilde \lambda_k}{\partial \sigma_k}\big(\gamma_k(\tau)\big) \bigg| \leq \const \big|v_k(\tau)\big| = 0.
\end{equation*}
As a consequence, formula \eqref{E_second_derivative_f} holds $\mathcal L^1$-a.e., and being both sides continuous we can conclude.
\end{proof}

\begin{lemma}[Translation of the starting point]
\label{L_translation}
Let $\gamma = \gamma_k(u_0, s) = (u,v_k,\sigma_k)$ and $\gamma' = \gamma_k(u_0', s) = (u', v_k',\sigma_k')$. Denote by $f_k, f_k'$ the reduced flux associated to $\gamma, \gamma'$ respectively.

Then it holds
\begin{subequations}
\begin{equation*}
\|u - u'\|_\infty  \leq \const |u_0 - u_0'|, \qquad \quad \|v_k - v_k'\|_\infty \leq \const |s||u_0 - u_0'|, 
\end{equation*}
\begin{equation*}
\|\sigma_k - \sigma_k'\|_1 \leq \const |s| |u_0 - u_0'|, \qquad \bigg\| \dfrac{d^2 f_k}{d\tau^2} - \dfrac{d^2 f_k'}{d\tau^2}\bigg\|_1 \leq \const |s| |u_0 - u_0'|.
\end{equation*}
\end{subequations}
\end{lemma}
\begin{proof}
It holds
\begin{equation*}
\left\{
\begin{aligned}
u(\tau) &= u_0 + \int_0^\tau \tilde r_k(\gamma(\varsigma)) d\varsigma \\
v_k(\tau) &= f_k (\gamma; \tau) - \conv_{[0,s]} f_k (\gamma; \tau) \\
\sigma_k(\tau) &= \frac{d}{d\tau} \conv_{[0,s]} f_k (\gamma; \tau)
\end{aligned}
\right.
\qquad
\left\{
\begin{aligned}
u'(\tau) &= u'_0 + \int_0^\tau \tilde r_k(\gamma'(\varsigma)) d\varsigma \\
v'_k(\tau) &= f_k (\gamma'; \tau) - \conv_{[0,s]} f_k (\gamma'; \tau) \\
\sigma'_k(\tau) &= \frac{d}{d\tau} \conv_{[0,s]} f_k (\gamma'; \tau)
\end{aligned}
\right.
\end{equation*}
Consider the curve
$\widetilde{\gamma}(\tau) = \gamma(\tau) + (u_0' - u_0,0,0)$, i.e. the translation of $\gamma$ from the starting point $(u_0, 0, \sigma_k(0))$ to $(u_0', 0, \sigma_k(0))$ and set $\widetilde{\widetilde{\gamma}} := \mathcal{T}(\widetilde \gamma)$. The curves $\widetilde{\gamma}$ and $\widetilde{\widetilde{\gamma}}$ satisfy the following systems
\begin{equation*}
\left\{
\begin{aligned}
\widetilde{u}(\tau) &= u_0' + \int_0^\tau \tilde r_k(\gamma(\varsigma)) d\varsigma \\
\tilde v_k(\tau) &= f_k (\gamma; \tau) - \conv_{[0,s]} f_k (\gamma; \tau) \\
\widetilde{\sigma}_k(\tau) &= \frac{d}{d\tau} \conv_{[0,s]} f_k (\gamma; \tau)
\end{aligned}
\right.
\qquad
\left\{
\begin{aligned}
\widetilde{\widetilde{u}}(\tau) &= u_0' + \int_0^\tau \tilde r_k(\widetilde{\gamma}(\varsigma)) d\varsigma \\
\widetilde{\widetilde{v}}_k(\tau) &= f_k (\widetilde{\gamma}; \tau) - \conv_{[0,s]} f_k (\widetilde{\gamma}; \tau) \\
\widetilde{\widetilde{\sigma}}_k(\tau) &= \frac{d}{d\tau} \conv_{[0,s]} f_k (\widetilde{\gamma}; \tau)
\end{aligned}
\right.
\end{equation*}

Let us prove now the first inequality. Using the Contraction Mapping Principle, we get 
\begin{equation*}
\|u-u'\|_\infty \leq \|u-\tilde u\|_\infty + \|\tilde u - u'\|_\infty \leq |u_0 - u_0'| + D(\widetilde \gamma, \gamma') \leq |u_0 - u_0'| + 2 D(\widetilde \gamma, \mathcal T(\widetilde \gamma)),
\end{equation*}
being the map $\mathcal T$ a contraction with constant $1/2$. \\
Since $\widetilde{\gamma}$ is obtained from $\gamma$ by translation of the initial point of the $u$-component, we get
\begin{equation*}
\begin{split}
|\tilde{\tilde u}(\tau) - \tilde u(\tau)| &\leq \int_0^{\tau} \big| \tilde r_k(\tilde \gamma_k(\varsigma)) - \tilde r_k(\gamma_k(\varsigma)) \big| d \varsigma \\
&\leq \int_{0}^\tau \Bigg( \bigg\| \frac{\partial \tilde r_k}{\partial u} \bigg\|_\infty \big| \tilde u(\varsigma) - u(\varsigma) \big| +  
\bigg\| \frac{\partial \tilde r_k}{\partial v_k} \bigg\|_\infty \big| \tilde v_k(\varsigma) - v_k(\varsigma) \big|  + 
\bigg\|\frac{\partial \tilde r_k}{\partial \sigma_k} \bigg\|_\infty \big| \tilde \sigma_k(\varsigma) - \sigma_k(\varsigma) \big| \Bigg) d\varsigma \\
&= \int_0^\tau  \bigg\| \frac{\partial \tilde r_k}{\partial u} \bigg\|_\infty \big| \tilde u(\varsigma) - u(\varsigma) \big| d\varsigma \\
&\leq \const |u_0 - u_0'||s|.
\end{split}
\end{equation*}
\\
Similarly,
\[
\big\|\widetilde{\widetilde{v}}_k - \widetilde{v}_k \big\|_\infty \leq \const |s||u_0 - u_0'|,
\quad
\big\|\widetilde{\widetilde{\sigma}}_k - \widetilde{\sigma}_k \big\|_1 \leq \const |s||u_0 - u_0'|,
\]
and thus
\begin{equation}
\label{E_distance_gamma_translation}
D(\widetilde \gamma, \widetilde{\widetilde{\gamma}}) \leq \const |s||u_0 - u_0'|.
\end{equation}
Hence
\begin{equation*}
\|u-u'\|_\infty \leq \const |u_0 - u_0'|.
\end{equation*}

In a similar way
\[
\|v_k - v_k'\|_\infty \leq \|v_k - \widetilde{v}_k\|_\infty + \|\widetilde{v}_k - v_k'\|_\infty 
\leq \|\widetilde{v}_k - v_k'\|_\infty \leq D(\widetilde{\gamma}, \gamma') \leq 2 D\big(\widetilde{\gamma}, \mathcal{T}(\widetilde{\gamma})\big),
\]
and
\[
\|\sigma_k - \sigma_k'\|_1 \leq \|\sigma_k - \widetilde{\sigma}_k\|_1 + \|\widetilde{\sigma}_k - \sigma_k'\|_1 
\leq \|\widetilde{\sigma}_k - \sigma_k'\|_1 \leq D(\widetilde{\gamma}, \gamma') \leq 2 D\big(\widetilde{\gamma}, \mathcal{T}(\widetilde{\gamma})\big).
\]
A further application of \eqref{E_distance_gamma_translation} yields the estimates on $v_k, \sigma_k$. 

Finally, using the chain rule, Lemma \ref{L_second_derivative_f}, Proposition \ref{P_estim_diff_conv} and the first part of the proof, we get
\begin{equation}
\label{E_translation_f}
\begin{split}
\bigg\| \dfrac{d^2 f_k}{d\tau^2} - \dfrac{d^2 f_k'}{d\tau^2}\bigg\|_1 &\leq \const \bigg[\int_0^s \Big(|u(\tau) - u'(\tau)| + |v_k(\tau) - v_k'(\tau) | + |\sigma_k(\tau) - \sigma_k'(\tau)|\Big) d\tau \bigg]  \\
&\leq \const |s| |u_0 - u_0'|. 
\end{split}
\end{equation}
This concludes the proof.
\end{proof}

\begin{lemma}[Change of the length of the curve]
\label{L_change_length}
Let $s', s'' \in \R, u_0 \in \R^n$ and assume $s's''\geq0$. Let
\begin{equation*}
\gamma = \gamma_k(u_0, s'+s'') = (u,v_k,\sigma_k), \quad \gamma' = (u', v_k',\sigma_k') := \gamma_k(u_0, s'), \quad \gamma'' = (u'', v_k'', \sigma_k'') := \gamma_k(u(s'), s'').
\end{equation*}
As before, denote by $f_k, f_k', f_k''$ the reduced fluxes associated by \eqref{E_reduced_flux} to $\gamma,\gamma', \gamma''$ respectively; assume also that $\gamma'', f_k''$ are defined on $s' + \I(s'')$ instead of $\I(s'')$.

Then it holds
\begin{equation}
\label{E_change_length_1}
D\big(\gamma|_{\I(s')},\gamma'\big) \leq \const |s'||s''|, \qquad D\big(\gamma|_{s'+ \I(s'')},\gamma''\big) \leq \const |s'||s''|,
\end{equation}
and
\begin{subequations}
\label{E_change_length_2}
\begin{equation}
\label{E_change_lenght_21}
\bigg\| \dfrac{d^2 f_k}{d\tau^2} - \dfrac{d^2 f_k'}{d\tau^2}\bigg\|_{L^1(\I(s'))} \leq \const D\big(\gamma|_{\I(s')},\gamma'\big) \leq \const |s||s'|,
\end{equation}
\begin{equation}
\label{E_change_lenght_22}
\bigg\| \dfrac{d^2 f_k}{d\tau^2} - \dfrac{d^2 f_k''}{d\tau^2}\bigg\|_{L^1(s'+ \I(s''))} \leq \const D\big(\gamma|_{s'+\I(s'')},\gamma''\big) \leq \const |s||s'|.
\end{equation}
\end{subequations}
\end{lemma}

\begin{proof}
We prove only the first inequality in \eqref{E_change_length_1} and \eqref{E_change_length_2} and only in the case  $s',s''\geq 0$; all the other cases can be treated similarly.

By the Contraction Mapping Theorem, $D\big( \gamma|_{\I(s')}, \gamma' \big) \leq 2 D\big(\gamma|_{\I(s')}, \widetilde \gamma \big)$, where $\widetilde \gamma = (\tilde u, \tilde v_k, \tilde \sigma_k) := \mathcal{T} (\gamma|_{\I(s')})$: 
\begin{equation*}
\left\{
\begin{aligned}
\tilde u(\tau) &= u_0 + \int_0^\tau \tilde r_i(\gamma(\varsigma)) d\varsigma \\
\tilde v_k(\tau) &= f_k (\gamma; \tau) - \conv_{[0,s']} f_k (\gamma; \tau) \\
\tilde \sigma_k(\tau) &= \frac{d}{d\tau} \conv_{[0,s']} f_k (\gamma; \tau)
\end{aligned}
\right.
\end{equation*}
Hence $\tilde u(\tau) = u(\tau)$ for any $\tau \in [0,s]$. It also holds
\[
\begin{split}
\big|\tilde v_k(\tau) - v_k(\tau)\big| &= \Big|\conv_{[0,s']}f_k(\gamma;\tau) - \conv_{[0,s'+s'']}f_k(\gamma;\tau)\Big| \\
&= \bigg|\int_0^\tau  \frac{d}{d\tau} \conv_{[0,s']} f_k (\gamma; \tau) - \frac{d}{d\tau} \conv_{[0,s'+s'']} f_k (\gamma; \tau) d\tau\bigg| \\
&\leq \int_0^\tau \bigg|\frac{d}{d\tau} \conv_{[0,s']} f_k (\gamma; \tau) - \frac{d}{d\tau} \conv_{[0,s'+s'']} f_k (\gamma; \tau)  \bigg| \\
&\leq \bigg\|\frac{d}{d\tau} \conv_{[0,s']} f_k(\gamma) - \frac{d}{d\tau} \conv_{[0,s'+s'']} \tilde f_k(\gamma)|_{[0,s']}  \bigg\|_\infty |s'| \\
\text{(by Prop. \ref{P_diff_vel_proporzionale_canc})} &\leq \const |s'||s''|.
\end{split}
\]
Clearly 
\begin{equation*}
\begin{split}
\|\tilde \sigma_k - \sigma_k\|_1 &\leq \Big\|\frac{d}{d\tau} \conv_{[0,s']} f_k(\gamma) - \frac{d}{d\tau} \conv_{[0,s'+s'']}  f_k(\gamma)|_{[0,s']}  \Big\|_1 \\
&\leq \Big\|\frac{d}{d\tau} \conv_{[0,s']} f_k(\gamma) - \frac{d}{d\tau} \conv_{[0,s'+s'']}  f_k(\gamma)|_{[0,s']}  \Big\|_\infty |s'| \\
\text{(by Prop. \ref{P_diff_vel_proporzionale_canc})} &\leq \const |s'||s''|.
\end{split}
\end{equation*}
Hence we obtain
\begin{equation*}
D\big(\gamma|_{\I(s')},\gamma'\big) \leq 2 D\big(\gamma|_{\I(s')}, \widetilde \gamma \big) \leq \const |s'||s''|.
\end{equation*}

The first inequality in \eqref{E_change_length_2} is proved as in Lemma \ref{L_translation}, inequality \eqref{E_translation_f}. The second inequality in \eqref{E_change_length_1} and the second inequality in \eqref{E_change_length_2} can be treated similarly. 
\end{proof}

\subsubsection{Elementary interactions}
\label{Sss_elemnt_inter}

Now, using the previous estimates, we study how the curves $\gamma_k$ and the reduced fluxes $f_k$ vary when an elementary interaction between two Riemann problems occurs: by elementary interaction we means that two adjacent Riemann problem contains waves belonging to one family only. We consider the cases when the waves belong to two different families (\emph{transversal interaction}) or to the same family and have same sign (\emph{non-transversal interaction} or simply \emph{interaction}).

The case when the waves belong to the same family but have different sign (\emph{cancellation}) will be controlled by Lemma \ref{L_change_length}. 

\begin{lemma}[Transversal interaction]
\label{L_transversal_interaction}
Let $u^M = T^k_{s_k}(u^L)$, $u^R = T^h_{s_h}(u^M)$. Set $\hat u^M := T^h_{s_h}(u^L)$, $\hat u^R:= T^k_{s_k}(\hat u^M)$ and 
\begin{align*}
\gamma_k &= (u_k, v_k, \sigma_k) := \gamma_k(u^L, s_k), & 
\gamma_h = (u_h, v_h, \sigma_h) := \gamma_h(u^M, s_h),  \\
\widehat \gamma_h &= (\hat u_h, \hat v_h, \hat \sigma_h) := \gamma_h(u^L, s_h), & 
\widehat \gamma_k = (\hat u_k, \hat v_k, \hat \sigma_k) :=\gamma_k(\hat u^M, s_k).
\end{align*}
Denote by $f_k, \hat f_k, f_h, \hat f_h$ the reduced fluxes associated to the curves $\gamma_k, \hat \gamma_k, \gamma_h, \hat \gamma_h$ respectively. Then it holds
\begin{subequations}
\label{E_transv}
\begin{equation}
\label{E_trans_u}
\|u_h - \hat u_h\|_\infty \leq \const |s_k|, \qquad \qquad \qquad \quad \ \|u_k - \hat u_k\|_\infty \leq \const |s_h|,
\end{equation}
\begin{equation}
\label{E_trans_v}
\quad \ \, \|v_h - \hat v_h\|_\infty \leq \const |s_k||s_h|, \qquad \qquad \qquad \|v_k - \hat v_k\|_\infty \leq \const |s_k||s_h|,
\end{equation}
\begin{equation}
 \label{E_trans_sigma}
\quad \quad \|\sigma_h - \hat \sigma_h\|_1 \leq \const |s_k||s_h|, \qquad \qquad \qquad \ \|\sigma_k - \hat \sigma_k\|_1 \leq \const |s_k||s_h|,
\end{equation}
\begin{equation*}
\bigg\| \frac{d^2 f_k}{d\tau^2} - \frac{d^2 \hat f_k}{d\tau^2}\bigg\|_1 \leq \const |s_k||s_h|, \qquad \qquad \bigg\| \frac{d^2 f_h}{d\tau^2} - \frac{d^2 \hat f_h}{d\tau^2}\bigg\|_1 \leq \const |s_k||s_h|,
\end{equation*}
\end{subequations}
and
\begin{equation}
\label{E_dist_final_points_translation}
|u^R - \hat u^R|  \leq \const |s_k||s_h|.
\end{equation}
\end{lemma}

\begin{proof}
Inequalities \eqref{E_transv} are direct consequence of Lemma \ref{L_translation} and the fact that
\begin{align*}
\big| u^M - u^L \big| & \leq \const |s_k|, & \big| \hat u^M - u^L \big| \leq \const |s_h|.
\end{align*}
Let us now prove inequality \eqref{E_dist_final_points_translation}. We have 
\[
\begin{split}
u^R &= u^L + \int_0^{s_k} \tilde r_k(\gamma_k(\varsigma)) d\varsigma + \int_0^{s_h} \tilde r_h(\gamma_h(\varsigma)) d\varsigma, \\
\hat u^R &= u^L + \int_0^{s_h} \tilde r_h(\widehat \gamma_h(\varsigma)) d\varsigma + \int_0^{s_k} \tilde r_k(\widehat \gamma_k(\varsigma)) d\varsigma.
\end{split}
\]
Hence
\begin{equation*}
\begin{split}
|u^R - \hat u^R| &\leq \int_0^{s_k} \big| \tilde r_k(\gamma_k(\varsigma)) - \tilde r_k(\widehat \gamma_k(\varsigma))\big| d\varsigma 
+ \int_0^{s_h} \big| \tilde r_h(\gamma_h(\varsigma)) - \tilde r_h(\widehat \gamma_h(\varsigma))\big| d\varsigma \\
&\leq \const \bigg[ \int_0^{s_k}\Big( | u_k(\varsigma) - \hat u_k(\varsigma)| + |v_k(\varsigma) - \hat v_k(\varsigma)| + |\sigma_k(\varsigma) - \hat \sigma_k(\varsigma) \big|\Big) d\varsigma \\
& \qquad \qquad \qquad +  
\int_0^{s_h} \Big( |u_h(\varsigma) - \hat u_h(\varsigma)| + |v_h(\varsigma) - \hat v_h(\varsigma)| + |\sigma_h(\varsigma) - \hat \sigma_h(\varsigma) \big|\Big) d\varsigma \bigg] \\
&\leq \const \Big[  |s_k| \| u_k - \hat u_k\|_\infty + |s_k| \|v_k - \hat v_k\|_\infty + \|\sigma_k - \hat \sigma_k\|_1 \\
& \qquad \qquad \qquad |s_h| \| u_h - \hat u_h\|_\infty + |s_h| \|v_h - \hat v_h\|_\infty + \|\sigma_h - \hat \sigma_h\|_1 \Big] \\
\text{(by \eqref{E_trans_u}-\eqref{E_trans_sigma})} &\leq \const |s_k||s_h|,
\end{split}
\end{equation*}
thus concluding the proof of the lemma.
\end{proof}

Let $\bar u^1 \in B(0,\delta/2)$, $A \in \N$ and consider a family of $A$ consecutive curves, $\{\gamma_a\}_{a \in \{1, \dots, A\}}$: given
\begin{equation*}
\{1, \dots, A\} \ni a \mapsto (k(a),s^a) \in \{1, \dots, n\} \times \R,
\end{equation*}
define
\begin{equation*}
\gamma^a(\tau) = \big(u^a(\tau), v^a(\tau), \sigma^a(\tau)\big) = \gamma_{k(a)}\big(\bar u^a, s^a\big)(\tau), \qquad a \in 1,\dots,A,
\end{equation*}
where $\bar u^{a+1} = u^a(s^a)$, for $a = 1, \dots A-1$.

Let $\{\tilde \gamma_a = (\tilde u^a,\tilde v^a,\tilde \sigma^a)\}_{a \in \{1,\dots,A\}}$ be the curves obtained by switching the values of $k(a)$, $s(a)$ in two consecutive points $\bar a$, $\bar a+1$, $1 \leq \bar a < A$:
\begin{equation*}
\tilde k(a) = \begin{cases}
k(a) & a \not= \bar a, \bar a+1, \\
k(\bar a+1) & a = \bar a, \\
k(\bar a) & a = \bar a + 1,
\end{cases} \qquad \tilde s(a) = \begin{cases}
s(a) & a \not \bar a, \bar a+1, \\
s(\bar a+1) & a = \bar a, \\
s(\bar a) & a = \bar a + 1.
\end{cases}
\end{equation*}
We assume $\sum |s^a| \ll 1$ sufficiently small (depending only on $f$) so that the curves $\gamma_a$, $\tilde \gamma_a$ remain in $B(0,\delta)$.
%

Denote by $f^{a}, \tilde f^a$, $a =1, \dots, A$, the reduced fluxes associated to $\gamma^a, \widetilde \gamma^a$ respectively.

\begin{corollary}
\label{C_switching_curves}
It holds
\begin{align}
\label{E_switching_curves}
\begin{array}{cc}
\big\| u^a - \tilde u^a\big\|_{\infty} \leq \const |s^{\bar a}||s^{\bar a +1}|, \qquad &
\big\|v^a - \tilde v^a\big\|_{\infty} \leq \const |s^{\bar a}||s^{\bar a +1}|, \\ [1em]
\big\|\sigma^a - \tilde \sigma^a\big\|_{1} \leq \const |s^{\bar a}||s^{\bar a +1}|, \qquad &
\big\|f^a - \tilde f^a\big\|_{1} \leq \const |s^{\bar a}||s^{\bar a +1}|.
\end{array}
\end{align}
\end{corollary}

\begin{proof}
For $a =1, \dots, \bar a - 1$, the l.h.s. of each inequality in \eqref{E_switching_curves} is equal to zero. For $a = \bar a, \bar a +1$, the proof is an immediate consequence of Lemma \ref{L_transversal_interaction}.

Assume that \eqref{E_switching_curves} has been proved up to $\check a \in \bar a + 1,\dots,A-1$; in particular
\begin{equation*}
\big| u^{\check a}(s^{\check a}) - \tilde u^{\check a}(\tilde s^{\check a}) \big| \leq \const |s^{\bar a}| |s^{\bar a + 1}|.
\end{equation*}
Observe that at step $\check a+1$ we have to translate the initial point $\bar u^{\check a+1} =  u^{\check a}(s^{\check a})$ of the curve $\gamma^{\check a+1}$ to the point $\tilde u^{\check a}(\tilde s^{\check a})$.

Now use Lemma \ref{L_translation} to conclude that \eqref{E_switching_curves} holds for $\check a+1$.
%
\end{proof}

We now consider interactions among wavefronts of the same family and with same sign.
Let
\begin{equation*}
\gamma' = (u',v_k',\sigma_k'):= \gamma_k(u^L,s'), \qquad \gamma'' = (u'',v_k'',\sigma_k''):= \gamma_k(u^M,s''), 
\end{equation*}
where $u^M:= u'(s')$  and $s's''>0$. Let
\begin{equation*}
\gamma= (u,v_k, \sigma_k):= \gamma(u^L, s'+s'').
\end{equation*}
Denote by $f_k', f_k'', f_k$ the reduced flux associated to $\gamma', \gamma'', \gamma$ respectively. Assume as before that $\gamma''$ and $f_k''$ are defined on the interval $s' + \I(s'')$, instead of $\I(s'')$, and that $f_k'(s') = f_k''(s_k)$.

\begin{lemma}[Interaction of wavefronts of the same family and with same sign]
\label{L_interaction_same_family}
It holds 
\begin{equation*}
D(\gamma' \cup \gamma'', \gamma) \leq \const \Aquadr(u^L, u^M, u^R), \qquad 
\bigg\| \frac{d^2 (f_k' \cup f_k'')}{d\tau^2} - \frac{d^2 f_k}{d\tau^2}\bigg\|_1 \leq \const \Aquadr(u^L, u^M, u^R).
\end{equation*}
\end{lemma}

\begin{proof}
We prove the lemma only in the case $s', s'' >0$, the case $s', s'' <0$ being entirely similar. We have
\begin{equation*}
\left\{
\begin{aligned} 
u'(\tau) &= u^L + \int_0^\tau \tilde r_k(\gamma'(\varsigma)) d\varsigma, \\
v_k'(\tau) &= f_k' (\tau) - \conv_{[0,s']} f_k' (\tau), \\
\sigma_k'(\tau) &= \frac{d}{d\tau} \conv_{[0,s']} f_k' (\tau),
\end{aligned} 
\right.
\qquad \qquad
\left\{
\begin{aligned} 
u''(\tau) &= u^M + \int_{s'}^\tau \tilde r_k(\gamma''(\varsigma)) d\varsigma, \\
v''_k(\tau) &= f_k'' (\tau) - \conv_{[s',s'+s'']} f_k'' (\tau), \\
\sigma''_k(\tau) &= \frac{d}{d\tau} \conv_{[s',s'+s'']} f_k'' (\tau).
\end{aligned} 
\right.
\end{equation*}

Set $\hat \gamma = (\hat u, \hat v_k, \hat \sigma_k): = \mathcal T(\gamma' \cup \gamma'')$:
\begin{equation*}
\left\{
\begin{aligned} 
\hat u(\tau) &= u^L + \int_{0}^{\min\{\tau,s'\}} \tilde r_k(\gamma'(\varsigma)) d\varsigma + \int_{\min\{\tau,s'\}}^\tau \tilde r_k(\gamma''(\varsigma)) d\varsigma, \\
\hat v_k(\tau) &= \big(f_k' \cup f_k''\big)(\tau) - \conv_{[0,s'+s'']} \big(f_k' \cup f_k''\big)(\tau), \\
\hat \sigma_k(\tau) &= \frac{d}{d\tau} \conv_{[0,s'+s'']} \big(f_k' \cup f_k''\big)(\tau).
\end{aligned} 
\right.
\end{equation*}

Let us prove the first inequality of the statement. It holds
\[
\| ( u' \cup u'' ) - \hat u \|_\infty = 0, \qquad \| ( v_k' \cup v_k'' ) - \hat v_k \|_\infty = \Aquadr(u^L, u^M, u^R).
\]
Moreover,
\begin{equation*}
\begin{split}
\big\| \big(\sigma_k' \cup \sigma_k''\big) - \hat \sigma_k \big\|_1
&= \int_0^{s'} \bigg| \frac{d \conv_{[0,s']}f_k'}{d\tau}(\tau) - \frac{d \conv_{[0,s'+s'']} (f_k' \cup f_k'')}{d\tau}  \bigg| d\tau \\ 
& \quad +
\int_{s'}^{s'+s''} \bigg| \frac{d \conv_{[s',s'+s'']}f_k''}{d\tau} - \frac{d \conv_{[0,s'+s'']} (f_k' \cup f_k'')}{d\tau}  \bigg| d\tau  \\
&= \int_0^{s'} \bigg[ \frac{d \conv_{[0,s']}f_k'}{d\tau} - \frac{d \conv_{[0,s'+s'']} (f_k' \cup f_k'')}{d\tau}  \bigg] d\tau \\ 
& \quad +
\int_{s'}^{s'+s''} \bigg[ \frac{d \conv_{[0,s'+s'']} (f_k' \cup f_k'')}{d\tau} - \frac{d \conv_{[s',s'+s'']}f_k''}{d\tau} \bigg] d\tau  \\
&= \Big(f_k'(s') - \conv_{[0,s'+s'']}\big(f_k' \cup f_k'' \big)(s')\Big) 
+ \Big(f_k''(s') - \conv_{[0,s'+s'']}\big(f_k' \cup f_k'' \big)(s') \Big) \\
&= 2 \Big(f_k'(s') - \conv_{[0,s'+s'']}\big(f_k' \cup f_k'' \big)(s')\Big) \\
&= 2 \ \Aquadr(u^L, u^M, u^R).
\end{split}
\end{equation*}
Hence
\[
D\big(\gamma' \cup \gamma'', \widehat \gamma \big) = 3 \Aquadr(u^L, u^M, u^R),
\]
and thus, by the Contraction Mapping Theorem,
\[
D\big(\gamma' \cup \gamma'', \gamma \big)  \leq 2 D\big(\gamma' \cup \gamma'', \widehat \gamma \big) \leq 6 \Aquadr(u^L, u^M, u^R). 
\]
The second inequality is a consequence of Lemma \ref{L_second_derivative_f}, Proposition \ref{P_estim_diff_conv} and the first part of the statement. 
\end{proof}

\subsubsection{Conclusion of the proof of Theorem \ref{T_general}}
\label{Sss_conclu_th_gen}

To prove Theorem \ref{T_general}, we piece together all the previous estimates as follows. First of all we split the operation of ``merging the two Riemann problems'' into three steps:
\begin{enumerate}
\item first we pass from the collection of curves \eqref{E_incoming_curves}, i.e. the black ones in Figure \ref{F_el_curves_general},
to the collection of curves
\begin{align}
\label{E_after_trans_curve}
&\widetilde{\gamma}_1'  = (\tilde u_1', \tilde v_1', \tilde \sigma_1') := \gamma_1(u^L, s_1'),  && \tilde \gamma_1''  = (\tilde u_1'', \tilde v_1'', \tilde \sigma_1'') := \gamma_1(u_1'(s_1'), s_1''), \crcr
& \tilde \gamma_k'  = (\tilde u_k', \tilde v_k', \tilde \sigma_k') := \gamma_k \big(\tilde u_{k-1}''(s_{k-1}''), s_k'\big), && 
\tilde \gamma_k''  = (\tilde u_k'', \tilde v_k'', \tilde \sigma_k'') := \gamma_k \big(\tilde u_{k}'(s_{k}'), s_k''\big), 
\ k = 2, \dots, n,
\end{align}
i.e. the red curves in Figure \ref{F_el_curves_general}; this first step will be called \emph{transversal interactions} and it will be studied in Lemma \ref{L_general_transversal}; 
\item as a second step, we let the curves of the same family interact, passing from the collection of red curves \eqref{E_after_trans_curve} to the collection of curves (green in Figure \ref{F_el_curves_general})
\begin{equation}
\label{E_after_same_family_curves}
\begin{split}
& \tilde \gamma_1 = (\tilde u_1, \tilde v_1, \tilde \sigma_1) := \gamma_1(u^L, s_1' + s_1''), \\
& \tilde \gamma_k = (\tilde u_k, \tilde v_k, \tilde \sigma_k) := \gamma_k\big( \tilde u_{k-1}(s'_{k-1}+s_{k-1}'') ,s_k'+s_k''\big), \ k =2, \dots, n;
\end{split}
\end{equation}
this operation will be called \emph{collision among waves of the same family} and it will be studied in Lemma \ref{L_general_same_family};
\item finally we pass from the collection of green curves \eqref{E_after_same_family_curves} to the outcoming collection of curves \eqref{E_outcoming_curves}, blue in Figure \ref{F_el_curves_general}; this operation will be called \emph{perturbation of the total variation due to nonlinearity}, and it will be studied in Lemma \ref{L_general_cubic} and its Corollary \ref{C_general_cubic}.
\end{enumerate}
We will denote by $\tilde f_k', \tilde f_k'', \tilde f_k$ the reduced fluxes associated respectively to 
$\tilde{\gamma}_k', \tilde{\gamma}_k'', \tilde{\gamma}_k$
and, as before, we will assume that $\tilde \gamma_k'', \tilde f_k''$ are defined on $s_k' + \I(s_k'')$, instead of $\I(s_k'')$, and $\tilde f_k''(s_k') = \tilde f_k'(s_k')$. Let us begin with the analysis of transversal interactions. 

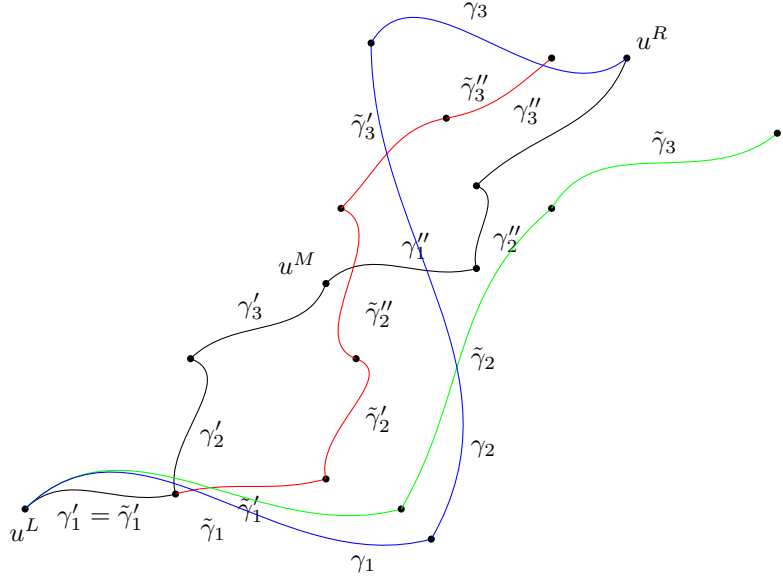
\begin{figure}
\begin{tikzpicture}

\draw[fill] (1,1) circle [radius = 0.04];
\draw (1,1) to [out=45,in=195] (3,1.2);
\node[below] at (2,1.2) {$\gamma_1' = \tilde \gamma_1'$};

\draw[fill] (3,1.2) circle [radius = 0.04];
\draw (3,1.2) to [out=105,in=340] (3.2,3);
\node[right] at (3.2,2) {$\gamma_2'$};

\draw[fill] (3.2,3) circle [radius = 0.04];
\draw (3.2,3) to [out=45,in=250] (5,4);
\node[above] at (4,3.4) {$\gamma_3'$};

\draw[fill] (5,4) circle [radius = 0.04];
\draw (5,4) to [out=45,in=195] (7,4.2);
\node[above] at (6.2,4.2) {$\gamma_1''$};

\draw[fill] (7,4.2) circle [radius = 0.04];
\draw (7,4.2) to [out=105,in=340] (7,5.3);
\node[right] at (7.1,4.6) {$\gamma_2''$};

\draw[fill] (7,5.3) circle [radius = 0.04];
\draw (7,5.3) to [out=45,in=250] (9,7);
\node[above left] at (8,6) {$\gamma_3''$};

\draw[fill] (9,7) circle [radius = 0.04];

\draw[red] (3,1.2) to [out=15,in=195] (5,1.4);
\node[below] at (4,1.3) {$\tilde \gamma_1'$};

\draw[fill] (5,1.4) circle [radius = 0.04];
\draw[red] (5,1.4) to [out=105,in=340] (5.4,3);
\node[right] at (5.4,2.2) {$\tilde \gamma_2'$};

\draw[fill] (5.4,3) circle [radius = 0.04];
\draw[red] (5.4,3) to [out=160,in=340] (5.2,5);
\node[right] at (5.4,3.6) {$\tilde \gamma_2''$};

\draw[fill] (5.2,5) circle [radius = 0.04];
\draw[red] (5.2,5) to [out=45,in=190] (6.6,6.2);
\node[above left] at (5.8,5.8) {$\tilde \gamma_3'$};

\draw[fill] (6.6,6.2) circle [radius = 0.04];
\draw[red] (6.6,6.2) to [out=10,in=220] (8,7);
\node[above left] at (7.3,6.3) {$\tilde \gamma_3''$};
\draw[fill] (8,7) circle [radius = 0.04];

\draw[green] (1,1) to [out=45,in=195] (6,1);
\node[below] at (3.5,1) {$\tilde \gamma_1$};

\draw[fill] (6,1) circle [radius = 0.04];
\draw[green] (6,1) to [out=60,in=220] (8,5);
\node[right] at (6.8,3) {$\tilde \gamma_2$};

\draw[fill] (8,5) circle [radius = 0.04];
\draw[green] (8,5) to [out=60,in=220] (11,6);
\node[above] at (9.5,5.6) {$\tilde \gamma_3$};
\draw[fill] (11,6) circle [radius = 0.04];

\draw[blue] (1,1) to [out=45,in=195] (6.4,0.6);
\node[below] at (5.5,0.5) {$\gamma_1$};

\draw[fill] (6.4,0.6) circle [radius = 0.04];
\draw[blue] (6.4,0.6) to [out=60,in=270] (5.6,7.2);
\node[right] at (6.8,1.8) {$\gamma_2$};

\draw[fill] (5.6,7.2) circle [radius = 0.04];
\draw[blue] (5.6,7.2) to [out=60,in=220] (9,7);
\node[above] at (7,7.4) {$\gamma_3$};

\node[below] at (1,1) {$u^L$};
\node[above left] at (5,4) {$u^M$};
\node[above right] at (9,7) {$u^R$};

\end{tikzpicture}
\caption{Elementary curves of two interacting Riemann problems before the interaction (black ones), after transversal interaction (red ones), after interaction/cancellation (collision) among wavefronts of the same family (green ones), after creation/cancellation (perturbation of the total variation) due to non-linearity (blue ones).}
\label{F_el_curves_general}
\end{figure}

\begin{lemma}
\label{L_general_transversal}
For any $k =1, \dots, n$, it holds
\begin{equation*}
\left.
\begin{aligned}
& \|u_k' - \tilde u_k'\|_{L^{\infty}(\I(s_k'))} \\
& \|u_k'' - \tilde u_k''\|_{L^{\infty}(s_k' + \I(s_k''))} \\
& \|\sigma_k' - \tilde \sigma_k'\|_{L^{1}(\I(s_k'))} \\
& \|\sigma_k'' - \tilde \sigma_k''\|_{L^{1}(s_k' + \I(s_k''))} \\
& {\displaystyle \bigg\|\frac{d^2 f_k'}{d\tau^2} - \frac{d^2 \tilde f_k'}{d\tau^2}\bigg\|_{L^1(\I(s_k'))}} \\
& {\displaystyle \bigg\|\frac{d^2 f_k''}{d\tau^2} - \frac{d^2 \tilde f_k''}{d\tau^2}\bigg\|_{L^1(s_k' + \I(s_k''))}}
\end{aligned} \right\} \leq \const \Atrans(u^L, u^M, u^R).
\end{equation*}
\end{lemma}
\begin{proof}
The proof is an easy consequence of Corollary \ref{C_switching_curves} and Definition \ref{D_atrans} of $\Atrans(u^L, u^M, u^R)$, just observing that we can pass from the collection of curves \eqref{E_incoming_curves} to the collection of curves \ref{E_after_trans_curve} first switching the curve $\gamma_1''$ with all the curves $\gamma_k'$, $k =2, \dots, n$, then switching the curve $\gamma_2''$ with all the curves $\gamma_k'$, $k = 3, \dots, n$, and so on up to curve $\gamma_{n-1}''$. 
\end{proof}

Let us now analyze the collision among waves of the same family.

\begin{lemma}
\label{L_general_same_family}
For any $k =1, \dots, n$, 
\begin{itemize}
\item if $s_k' s_k'' \geq 0$, then
\begin{equation*}
\left.
\begin{aligned}
&\|\big(\tilde u_k' \cup \tilde u_k''\big) - \tilde u_k\|_{L^{\infty}(\I(s_k'+s_k''))} \\
&\|\big(\tilde \sigma_k' \cup \tilde \sigma_k''\big) - \tilde \sigma_k\|_{L^{1}(\I(s_k' + s_k''))} \\
&\bigg\|\Big(\frac{d^2 \tilde f_k'}{d\tau^2} \cup \frac{d^2 \tilde f_k''}{d\tau^2}\Big) - \frac{d^2 \tilde f_k}{d\tau^2}\bigg\|_{L^1(\I(s_k' + s_k''))}
\end{aligned}
\right\}
\leq \const \bigg[ \sum_{h=1}^k \Aquadr_h(u^L, u^M, u^R) + \Acanc_h(u^L, u^M, u^R) \bigg].
\end{equation*}
\item if $s_k' s_k'' < 0$, then
\begin{equation*}
\left.
\begin{aligned}
& \|\big(\tilde u_k' \vartriangle  \tilde u_k''\big) - \tilde u_k\|_{L^{\infty}(\I(s_k'+s_k''))} \\
& \|\big(\tilde \sigma_k' \vartriangle  \tilde \sigma_k''\big) - \tilde \sigma_k\|_{L^{1}(\I(s_k' + s_k''))} \\
& \bigg\|\Big(\frac{d^2 \tilde f_k'}{d\tau^2} \vartriangle  \frac{d^2 \tilde f_k''}{d\tau^2}\Big) - \frac{d^2 \tilde f_k}{d\tau^2}\bigg\|_{L^1(\I(s_k' + s_k''))}
\end{aligned}
\right\}
\leq \const \bigg[ \sum_{h=1}^k \Aquadr_h(u^L, u^M, u^R) + \Acanc_h(u^L, u^M, u^R) \bigg].
\end{equation*}
\end{itemize}
\end{lemma}
\begin{proof}
\textit{Step 1.} First we prove that for each $k =1, \dots, n$, it holds
\begin{equation}
\label{E_general_same_family_1}
\big| \tilde u_k''(s_k' + s_k'') - \tilde u_k(s_k' + s_k'')\big| \leq \const \sum_{h=1}^k \Big[\Aquadr_h(u^L, u^M, u^R) + \Acanc_h(u^L, u^M, u^R) \Big].
\end{equation}
Recalling that $\tilde u_k''$ is defined on $s_k' + \I(s_k'')$, set
\begin{equation*}
\widehat \gamma_k = (\hat u_k, \hat v_k, \hat \sigma_k):= 
\begin{cases}
\gamma_k\big(\tilde u_k'(0), s_k'+s_k''\big) & \text{if } s_k's_k'' \geq 0 \text{ or } \Big(s_k' s_k'' < 0 \text{ and } |s_k'| > |s_k''| \Big), \\
\gamma_k\big(\tilde u_k''(0), s_k'+s_k''\big) & \text{if } s_k' s_k'' < 0 \text{ and } |s_k''| > |s_k'|.\\
\end{cases}
\end{equation*}
In order to prove \eqref{E_general_same_family_1}, distinguish three cases:
\begin{itemize}
\item first assume that $s_k' s_k'' \geq 0$; the following computation holds:
\begin{equation*}
\begin{split}
\big| \tilde u_k''(s_k' + s_k'') - \tilde u_k(s_k' + s_k'')\big| &\leq 
\big| \tilde u_k''(s_k' + s_k'') - \hat u_k(s_k' + s_k'')\big| + \big| \hat u_k(s_k' + s_k'') - \tilde u_k(s_k' + s_k'')\big| \\
\text{(by Lemma \ref{L_interaction_same_family})} 
&\leq  \const \Big[ \Aquadr_k(u^L, u^M, u^R) + \big| \hat u_k(s_k' + s_k'') - \tilde u_k(s_k' + s_k'')\big| \Big]\\
\text{(by Lemma \ref{L_translation})} &\leq \const \Big[ \Aquadr_k(u^L, u^M, u^R) 
+ \big| \hat u_k(0) - \tilde u_k(0)\big|  \Big]; 
\end{split}
\end{equation*}

\item now assume that $s_k' s_k'' <0$ and $|s_k'| \geq |s_k''|$; in this case it holds, applying again Lemma \ref{L_translation} and using the fact that $\tilde u_k''(s_k') = \tilde u_k'(s_k')$,
\begin{equation*}
\begin{split}
\big| \tilde u_k''(s_k' + s_k'') - \tilde u_k(s_k' + s_k'')\big| 
&\leq \big| \tilde u_k''(s_k' + s_k'') - \tilde u_k''(s_k')\big| + 
\big| \tilde u_k'(s_k') + \tilde u_k'(s_k'+s_k'') \big| \\
& \qquad \qquad + \big| \tilde u_k'(s_k' + s_k'') - \tilde u_k(s_k' + s_k'')\big| \\
&\leq \const \Big[ |s_k''| + \big| \tilde u_k'(0) - \tilde u_k(0)\big| \Big] \\
&= \const \Big[ \Acanc_k(u^L, u^M, u^R) +  \big| \tilde u_k'(0) - \tilde u_k(0)\big| \Big];
\end{split}
\end{equation*}

\item finally assume that $s_k' s_k'' <0$ and $|s_k'| < |s_k''|$ and perform the following computation:
\begin{equation*}
\begin{split}
\big| \tilde u_k''(s_k' + s_k'') - \tilde u_k(s_k' + s_k'')\big|
&\leq \big|\tilde u_k''(s_k'+s_k'') -  \hat u_k(s_k'+s_k'') \big| 
+ \big| \hat u_k(s_k' + s_k'') - \tilde u_k(s_k' + s_k'') \big| \\
\text{(by Lemmas \ref{L_change_length} and \ref{L_translation})} &\leq \const \Big[|s_k'| + \big| \hat u_k(0) - \tilde u_k(0) \big| \Big] \\
\text{(since $\hat u_k(0) = \tilde u''_k(0)$)} &\leq \const \Big[|s_k'| + \big| \tilde u_k''(0) - \tilde u_k(0) \big| \Big] \\
\text{(since $\tilde u_k'(s_k') = \tilde u''_k(s_k')$)} 
&\leq \const \Big[|s_k'| + \big| \tilde u_k''(0) - \tilde u_k''(s_k') \big| +
\big| \tilde u_k'(s_k') - \tilde u_k'(0) \big| 
+ \big| \tilde u_k'(0) - \tilde u_k(0) \big| \Big] \\
\text{(by Lemma \ref{L_change_length})} &\leq \const \Big[ |s_k'| +  \big| \tilde u_k'(0) - \tilde u_k(0) \big| \Big] \\
&\leq \const \Big[ \Acanc_k(u^L, u^M, u^R) + \big| \tilde u_k'(0) - \tilde u_k(0) \big| \Big];
\end{split}
\end{equation*}
\end{itemize}
Summarizing the three previous cases, we obtain
\begin{equation}
\label{E_general_same_family_2}
\big| \tilde u_k''(s_k' + s_k'') - \tilde u_k(s_k' + s_k'')\big|
\leq \const \Big[\Aquadr_k(u^L, u^M, u^R) + \Acanc_k(u^L, u^M, u^R) + \big| \tilde u_k'(0) - \tilde u_k(0) \big| \Big].
\end{equation}
If $k=1$, $\tilde u_1'(0)  = \tilde u_1(0) = u^L$, and thus \eqref{E_general_same_family_2} yields \eqref{E_general_same_family_1}. If $k \geq 2$, one observes that 
\begin{equation*}
\tilde u_k'(0) - \tilde u_k(0)
=  \tilde u_{k-1}''(s_{k-1}' + s_{k-1}'') - \tilde u_{k-1}(s_{k-1}' + s_{k-1}'')
\end{equation*}
and argues by induction to obtain \eqref{E_general_same_family_1}.

\textit{Step 2}. Using Step 1 we can now conclude the proof of the lemma. 
We will prove only the inequalities related to the $u$-component, the proof of the other ones being completely analogous. We again study  the three cases separately:
\begin{itemize}
\item if $s_k' s_k'' \geq 0$, it holds
\begin{equation*}
\begin{split}
\big\| (\tilde u_k' \cup \tilde u_k'' ) - \tilde u_k \big\|_{L^\infty(\I(s_k' + s_k''))} &\leq
\big\| (\tilde u_k' \cup \tilde u_k'' ) - \hat u_k\big\|_{L^\infty(\I(s_k' + s_k''))}
+ \big\| \hat u_k - \tilde u_k \big\|_{L^\infty(\I(s_k' + s_k''))} \\
\text{(by Lemmas \ref{L_interaction_same_family} and \ref{L_translation})}
&\leq \const \Big[\Aquadr_k(u^L, u^M, u^R) + \big|\tilde u_k'(0) - \tilde u_k(0) \big| \Big];
\end{split}
\end{equation*}

\item if $s_k' s_k'' <0$ and $|s_k'| \geq |s_k''|$, it holds
\begin{equation*}
\begin{split}
\big\| (\tilde u_k' \vartriangle  \tilde u_k'' ) - \tilde u_k \big\|_{L^\infty(\I(s_k' + s_k''))} &\leq
\big\| (\tilde u_k' \vartriangle  \tilde u_k'' ) - \hat u_k\big\|_{L^\infty(\I(s_k' + s_k''))} 
+ \big\| \hat u_k - \tilde u_k \big\|_{L^\infty(\I(s_k' + s_k''))} \\
\text{(by Lemmas \ref{L_change_length} and \ref{L_translation})}
&\leq \const \Big[\Acanc_k(u^L, u^M, u^R) + \big|\tilde u_k'(0) - \tilde u_k(0) \big| \Big];
\end{split}
\end{equation*}

\item if $s_k' s_k'' < 0$ and $|s_k'| < |s_k''|$, it holds
\begin{equation*}
\begin{split}
\big\| (\tilde u_k' \vartriangle  \tilde u_k'' ) - \tilde u_k \big\|_{L^\infty(\I(s_k' + s_k''))} &\leq
\big\| (\tilde u_k' \vartriangle  \tilde u_k'' ) - \hat u_k\big\|_{L^\infty(\I(s_k' + s_k''))} 
+ \big\| \hat u_k - \tilde u_k \big\|_{L^\infty(\I(s_k' + s_k''))} \\
\text{(by Lemmas \ref{L_change_length} and \ref{L_translation})} 
&\leq \const \Big[\Acanc_k(u^L, u^M, u^R) + \big|\tilde u_k''(0) - \tilde u_k(0) \big| \Big] \\
\text{(since $\tilde u_k'(s_k') = \tilde u''_k(s_k')$)} &\leq \const \Big[\Acanc_k(u^L, u^M, u^R) 
+ \big|\tilde u_k''(0) - \tilde u_k''(s_k')\big| \\
& \qquad \qquad \qquad
+ \big| \tilde u_k'(s_k') - \tilde u_k'(0) \big|
+ \big| \tilde u_k'(0) - \tilde u_k(0) \big| \Big] \\
&\leq \const \Big[\Acanc_k(u^L, u^M, u^R) 
+ \big| \tilde u_k'(0) - \tilde u_k(0) \big| \Big].
\end{split}
\end{equation*}
\end{itemize}
Summarizing,
\begin{equation}
\label{E_general_same_family_3}
\left.
\begin{aligned}
&\big\| (\tilde u_k' \cup  \tilde u_k'' ) - \tilde u_k \big\|_{L^\infty(\I(s_k' + s_k''))} \\
&\big\| (\tilde u_k' \vartriangle  \tilde u_k'' ) - \tilde u_k \big\|_{L^\infty(\I(s_k' + s_k''))}
\end{aligned}
\right\}
\leq
\const \Big[ \Aquadr_k(u^L, u^M, u^R) + \Acanc_k(u^L, u^M, u^R) +  \big| \tilde u_k'(0) - \tilde u_k(0) \big| \Big].
\end{equation}
If $k=1$, \eqref{E_general_same_family_3} together with the fact that $\tilde u_1'(0) = \tilde u_1(0) = u^L$ yields the thesis. If $k \geq 2$, one observes that $\tilde u_k'(0) = \tilde u_{k-1}''(s_k'+s_k'')$ and $\tilde u_k(0) = \tilde u_{k-1}(s_k'+s_k'')$; hence, using \eqref{E_general_same_family_3} and \eqref{E_general_same_family_1} of Step 1, one gets the statement.
\end{proof}

Finally, let us analyze the perturbation of the total variation due to nonlinearity.

\begin{lemma}
\label{L_general_cubic}
For any $k = 1, \dots, n$ it holds
\begin{equation*}
\left.
\begin{aligned}
& \|\tilde u_k - u_k\|_{L^\infty(\I(s_k'+s_k'') \cap \I(s_k))} \\
& \|\tilde \sigma_k - \sigma_k\|_{L^1(\I(s_k'+s_k'') \cap \I(s_k))} \\
& \bigg\|\frac{d^2 \tilde f_k}{d\tau^2} - \frac{d^2 f_k}{d\tau^2}\bigg\|_{L^1(\I(s_k'+s_k'') \cap \I(s_k))} \\
\end{aligned}
\right\} 
\leq \const \sum_{h=1}^k\big| s_h - (s_h' + s_h'')\big|.
\end{equation*}
\end{lemma}
\begin{proof}
We prove only the inequality related to the $u$ component, the other ones being completely similar. The proof is by induction on $k$. If $(s_k'+s_k'')s_k \leq 0$, there is nothing to prove. Hence, let us assume $(s_k'+s_k'')s_k >0$. Set $\overline \gamma_k = (\overline u_k, \overline v_k, \overline \sigma_k) := \gamma_k\big(\tilde u_k(0), s_k\big)$. It holds
\begin{equation}
\label{E_general_cubic_1}
\begin{split}
\big\|\tilde u_k - u_k\big\|_{L^\infty(I(s_k'+s_k'') \cap I(s_k))} &\leq
\big\|\tilde u_k - \overline u_k\big\|_{L^\infty(I(s_k'+s_k'') \cap I(s_k))}
+ \big\|\overline u_k - u_k\big\|_{L^\infty(I(s_k'+s_k'') \cap I(s_k))} \\
\text{(by Lemmas \ref{L_change_length} and \ref{L_translation})} 
&\leq \const \Big[ \big| |s_k| - |s_k' + s_k''| \big| + \big| \overline u_k(0) - u_k(0) \big|
 \Big].
\end{split}
\end{equation}
If $k=1$, \eqref{E_general_cubic_1} yields the thesis. If $k \geq 2$, observe that
\begin{equation*}
\begin{split}
\big|\overline u_k(0) - u_k(0) \big| &= \big| \tilde u_k(0) - u_k(0) \big| \\
&= \big| \tilde u_{k-1}(s_{k-1}' + s_{k-1}'') - u_{k-1}(s_{k-1}) \big|\\
&\leq \left\{
\begin{array}{ll}
\begin{split}
\big| \tilde u_{k-1}(s_{k-1}'+s_{k-1}'') - \tilde u_{k-1}(s_{k-1})& \big| \\+ \big| \tilde u_{k-1}(s_{k-1}) - u_{k-1}(s_{k-1}&) \big|
\end{split} & \text{if } |s_{k-1}' + s_{k-1}''| \geq |s_{k-1}|, \\ [1em]
\begin{split} \big| \tilde u_{k-1}(s_{k-1}'+s_{k-1}'') -  u_{k-1}(s_{k-1}'+s_{k-1}'')& \big| \\ + \big| u_{k-1}(s_{k-1}'+s_{k-1}'') - u_{k-1}(s_{k-1}&) \big| \end{split} & \text{if } |s_{k-1}' + s_{k-1}''| < |s_{k-1}|,
\end{array} \right. \\
&\leq \const \Big[ \big| s_{k-1} - (s_{k-1}' + s_{k-1}'') \big| 
+ \big\|\tilde u_{k-1} - u_{k-1}\big\|_{L^\infty(I(s_{k-1}'+s_{k-1}'') \cap I(s_{k-1}))} \Big] \\
\text{(by induction)} &\leq \sum_{h=1}^{k-1} \Big| s_h - (s_h'+s_h'') \Big|.
\end{split}
\end{equation*}
Hence, using \eqref{E_general_cubic_1}, we get
\begin{equation*}
\big\| \tilde u_k - u_k \big\|_{L^\infty(I(s_k'+s_k'') \cap I(s_k))} \leq 
\const \Big[ \big|s_k - (s_k' + s_k'')\big| + \big|  \overline u_k(0) - u_k(0) \big| \Big] \leq \const \sum_{h=1}^k \big| s_h - (s_h' + s_h'')\big|.
\end{equation*}
\end{proof}

Applying Theorem \ref{T_bianchini}, we immediately obtain the following corollary.

\begin{corollary}
\label{C_general_cubic}
For any $k = 1, \dots, n$ it holds
\begin{equation*}
\left.
\begin{aligned}
&\|\tilde u_k - u_k\|_{L^\infty(\I(s_k'+s_k'') \cap \I(s_k))} \\
&\|\tilde \sigma_k - \sigma_k\|_{L^1(\I(s_k'+s_k'') \cap \I(s_k))} \\
&\bigg\|\frac{d^2 \tilde f_k}{d\tau^2} - \frac{d^2 f_k}{d\tau^2}\bigg\|_{L^1(\I(s_k'+s_k'') \cap \I(s_k))} \\
\end{aligned}
\right\} 
\leq \const \bigg[ \Atrans(u^L, u^M, u^R) + \sum_{h=1}^k \Acubic_h(u^L, u^M, u^R) \bigg].
\end{equation*}
\end{corollary}

\noindent It is easy to see that Theorem \ref{T_general} follows from Lemmas \ref{L_general_transversal}, \ref{L_general_same_family} and Corollary \ref{C_general_cubic}.

\section{\texorpdfstring{Lagrangian representation for the Glimm approximate solution $u_\e$}{Lagrangian representation for the Glimm approximate solution}}
\label{S_lagr_repr}

In this section we define the notion of \emph{Lagrangian representation} of an approximate solution $u_\e$ obtained by the Glimm scheme to the Cauchy problem \eqref{E_gen_syst}-\eqref{E_initial_datum}, and we explicitly construct a Lagrangian representation satisfying some useful additional properties. At the end of the section we introduce some notions related to the Lagrangian representation; in particular, the notion of \emph{effective flux $\feff_k(t)$ of the $k$-th family at time $t$} will have a major role in the next sections.

\subsection{Definition of Lagrangian representation}
\label{Ss_lagrang_def}

Given a piecewise constant approximate solution $u_\e$ constructed by the Glimm scheme (see Section \ref{Ss_glimm_sol}), for any time $t \geq 0$ define the quantities
\begin{equation*}
L_k^+(t) := \sum_{m \in \Z} \big[s_k^{i,m} \big]^+, \qquad
L_k^-(t) := - \sum_{m \in \Z} \big[s_k^{i,m} \big]^-, \quad \text{if } t \in [i\e, (i+1)\e).
\end{equation*}
It is easy to see that $|L_k^+(t)| + |L_k^-(t)| \leq \const \TV(u_\e(t))$. A \emph{Lagrangian representation} for $u_\e$ is a set $\W$ called \emph{the set of waves}, together with 
\begin{itemize}
\item the maps
\begin{align*}
& {\rm family} : \W \to \{1, \dots, n\} && \text{the family of the wave $w \in \W$}, \\
& \mathcal S:\W \to \{\pm 1\} && \text{the sign of the wave $w \in \W$}, \\
& \tcr: \W \to [0, + \infty) && \text{the creation time of the wave $w \in \W$}, \\
& \tcanc: \W \to (0, + \infty] && \text{the cancellation time of the wave $w \in \W$},
\end{align*}
\item a relation, which we will denote by $\leq$,
\item the map, called \emph{position function},
\begin{equation*}
\mathtt x : \big\{ (t,w) \in [0, \infty) \times \W \ \big| \ \tcr(w) \leq t <\tcanc(w)\big\} \to \R,
\end{equation*}
\end{itemize}
which satisfy the conditions (1)-(4) below.

For the sake of convenience, set
\begin{equation*}
\begin{split}
\W_k &:= \big\{ w \in \W \ \big| \ {\rm family}(w) = k\big\}, \\ 
\W_k(t) &:= \big\{w \in \W_k \ \big| \ \tcr(w) \leq t < \tcanc(w) \big\}, \\
\W_k^{\pm}(t) &:= \big\{w \in \W_k(t) \ \big| \ \mathcal S(w) = \pm 1\big\}.
\end{split}
\end{equation*}

\begin{figure}
\resizebox{14cm}{8cm}{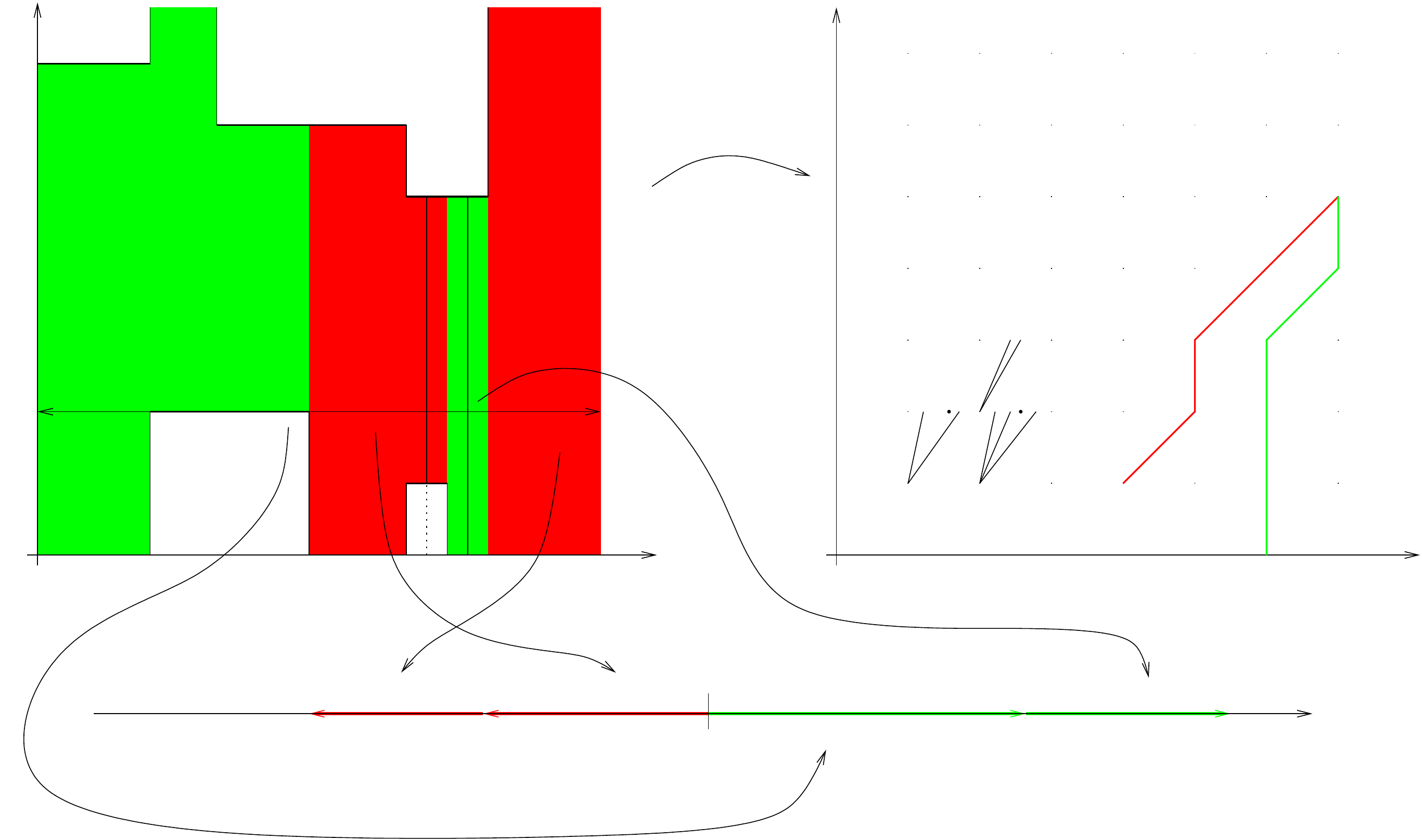}
\caption{The Lagrangian representation: at each $t$ the sum of the length of the red/green regions gives the set $L^-_k$, $L^+_k$, and $\mathtt x$ follows the trajectory of each wave $w$. The map $\Phi_k(t)$ is order reversing on $\W_k^-(t)$ (red) and order preserving on $\W^+_k(t)$ (green).}
\label{Fi_lagr_repr}
\end{figure}

The additional conditions to be satisfied by a Lagrangian representation are the following:
\begin{enumerate}
\item  for any family $k$, time $t$, sign $\pm 1$, the relation $\leq$ is a total order both on $\W_k^+(t)$ and on $\W_k^-(t)$; if $\mathcal I \subseteq \W_k^{\pm}(t)$ is an interval in the order set $(\W_k^{\pm}(t),\leq)$, we will say that $\mathcal I$ is \emph{an interval of waves (i.o.w.) at time $t$};
\item the map $\mathtt x$ satisfies:
\begin{enumerate}
\item for fixed time $t$, $\mathtt x(t,\cdot): \W_k(t) \to \R$ is increasing; 
\item for fixed $w \in \W$, the map $\mathtt x(\cdot, w): [\tcr(w), \tcanc(w)) \to \R$ is Lipschitz;
\item \label{Point_2c_lagra} for any point $(\bar t, \bar x) \in [0, +\infty) \times \R$, all the waves in
\begin{equation*}
\W_k(\bar t, \bar x) := \mathtt x(\bar t)^{-1}(\bar x) \cap \W_k
\end{equation*}
have the same sign; 
\end{enumerate}
\item \label{Pt_isomorph_ordered_sets} there exist maps $\Phi_k(t) : \W_k(t) \to \I\big(L_k^-(t) \big) \cup \I\big(L_k^+(t)\big)$ such that
$\Phi_k(t)|_{\W_k^+(t)}: \W_k^+(t) \to \I\big( L_k^+(t) \big)$ is an isomorphism of ordered sets, while $\Phi_k(t)|_{\W_k^-(t)}:\W_k^-(t) \to \I\big( L_k^-(t) \big)$ is an antisomorphism of ordered sets;
%
\item \label{Point_4_lagr_repr} there exist maps $\hat{\gamma}_k(t) : \W_k(t) \to \mathcal D_k \subseteq \R^m \times \R \times \R$, $\hat{\gamma}_k(t) = \big(\hat u_k(t), \hat v_k(t), \hat \sigma_k(t) \big)$, such that 
\begin{enumerate}
\item for any $\bar x \in \R$, setting
\begin{equation*}
u^L := \lim_{x \to \bar x^-} u_\e(t,x), \qquad u^R := \lim_{x \to \bar x^+} u_\e(t,x),
\end{equation*}
the collection of curves
\begin{equation*}
\Big\{ \W_k(t,\bar x) \ni w \mapsto \hat \gamma_k(t,w)  
\Big\}_{k = 1, \dots, n},
\end{equation*}
solves the Riemann problem $(u^L, u^R)$; 
\item for any $w \in \W_k^\pm(i\e)$, if $\tcanc(w) \geq (i+1)\e$, then for any time $t \in [i\e, (i+1)\e)$ it holds
\begin{equation}
\label{E_ode}
\mathtt x(t,w) =
\begin{cases}
\mathtt x(i\e,w) & \text{if } \vartheta_{i+1} \geq \hat \sigma_k(i\e,w), \\
\mathtt x(i\e,w) + (t-i\e) & \text{if } \vartheta_{i+1} < \hat \sigma_k(i\e,w).
\end{cases} 
\end{equation}
\end{enumerate}
\end{enumerate}

\subsection{Explicit construction of a Lagrangian representation}
\label{Ss_explic_lagr}

In this section we prove the following theorem.

\begin{theorem}
\label{T_lagrangian}
There exists at least one Lagrangian representation for the approximate solution $u_\e$ constructed by the Glimm scheme, which moreover satisfies the following conditions: for any grid point $(i\e,m\e) \in \N\e \times \Z\e$,
\begin{enumerate}[label=(\alph*)]
\item \label{Pt_iow_cons} the set $\W_k(i\e,m\e) \cap \W_k((i-1)\e)$ is an i.o.w. both at time $(i-1)\e$ and at time $i\e$, while the set $\W_k(i\e,m\e) \setminus \W_k((i-1)\e)$ is an i.o.w. at time $i\e$;
\item \label{Pt_affine} the map
\begin{equation*}
\Phi_k((i-1)\e)(\W_k(i\e,m\e) \cap \W_k((i-1)\e)) \overset{\Phi_k(i\e) \circ \Phi_k((i-1)\e)^{-1}}{\longrightarrow} \Phi_k(i\e)(\W_k(i\e,m\e) \cap \W_k((i-1)\e))
\end{equation*}
is an affine map with Lipschitz constant equal to $1$.
\end{enumerate}
\end{theorem}

Roughly speaking, the first condition means that we can insert/remove waves due to nonlinear interaction in an ordered way, the second condition focuses on the map at the level of $\I(L^-_k(t)) \cup \I(L^+_k(t))$.

\begin{proof}
The proof is divided in three steps:
\begin{enumerate}
\item first we define, for any family $k$ and any grid point $(i\e,m\e)$, $i \in \N$, $m \in \Z$, the set $\W_k(i\e,m\e)$ of $k$-th waves which at time $i\e$ are located at point $m\e$, together with the maps $\Phi_k^{i,m} = \Phi_k(i\e)|_{\W_k(i\e,m\e)}$;
\item then using the definitions given in Step (1), we construct all the other objects needed to have a Lagrangian representation, i.e. the set of waves $\W$, the sign $\mathcal S(w)$ of any wave $w$, the creation and cancellations time $\tcr(w), \tcanc(w)$, the relation $\leq$ and the position function $\mathtt x$;
\item finally we show that the additional properties (a) and (b) hold.
\end{enumerate}

\bigskip
\noindent \textit{Step 1.} The definitions of $\W_k(i\e,m\e)$ and $\Phi_k(i\e)$ are given by induction on times $i\e, i \in \N$, assuming that at each time
the map
\begin{equation}
\label{E_form_Phi}
\Phi_k^{i,m}: \W_k(i\e, m\e) \to \left( \sum_{\substack{m' < m \\ \sign(s_k^{i,m'}) = \sign(s_k^{i,m})}} s_k^{i, m'} \right) + \I(s_k^{i,m}),
\end{equation}
is a bijection;

At time $i\e = 0$, for any $m \in \Z$, define
\[
\W_k(0,m\e) := \I(s_k^{0,m}) \times \{0\} \times \{m\e\} \times \{k\} 
\]
and 
the bijection
\begin{equation*}
\Phi_k^{0,m}(w) := \left( \sum_{\substack{m' < m \\ \sign(s^{0,m'}_k) = \sign(s^{0,m}_k)}}s^{0,m'}_k \right) + \tau, \qquad \text{for } w = (\tau, 0, m\e, k).
\end{equation*}


For any $m \in \Z$, let us now define $\W_k(i\e, m\e)$ and $\Phi_k^{i,m}$
at time $i\e$, $i \geq 1$, assuming to have already defined, for any $m \in \Z$, the set $\W_k((i-1)\e, m\e) \subseteq \R \times \N\e \times \Z\e \times \{1, \dots, n\}$ and the bijections
\begin{equation*}
\Phi_k^{i-1,m}(t) : \W_k((i-1)\e,m\e) \to \I(L^-_k(t)) \cup \I(L^+_k(t))
\end{equation*}
at time $t = (i-1)\e$ of the form \eqref{E_form_Phi}.

If 
\begin{equation*}
\big\{\gamma_k^{i-1,m}\big\}_{k = 1, \dots, n}, \qquad \gamma_k^{i-1,m} = (u_k^{i-1,m}, v_k^{i-1,m}, \sigma_k^{i-1,m})
\end{equation*}
is the collection of curves which solves the Riemann problem $(u^{i-1,m-1}, u^{i-1,m})$, then, assuming that each $\gamma_k^{i-1,m}$ is defined on the set
\begin{equation*}
\left( \sum_{\substack{m' < m \\ \sign(s_k^{i-1,m'}) = \sign(s_k^{i-1,m})}} s_k^{i-1, m'} \right) + \I(s_k^{i-1,m})
\end{equation*}
instead of $\I(s_k^{i-1,m})$, set
\begin{equation*}
\hat \gamma_k^{i-1,m} = (\hat u_k^{i-1,m}, \hat v_k^{i-1,m}, \hat \sigma_k^{i-1,m}) := \gamma_k^{i-1,m} \circ \Phi_k^{i-1,m}.
\end{equation*}

\noindent Set also
\begin{equation*}
\begin{split}
& \W_k^{(0)}((i-1)\e, m\e) := \Big\{ w \in \W_k((i-1)\e, m\e) \ \big| \ \hat \sigma_k^{i-1,m}(w) \leq \vartheta_{i} \Big\}, \\
& \W_k^{(1)}((i-1)\e, m\e) := \Big\{ w \in \W_k((i-1)\e, m\e) \ \big| \ \hat \sigma_k^{i-1,m}(w) > \vartheta_{i} \Big\}.
\end{split}
\end{equation*}

Now fix $m \in \Z$ and define the set of waves located in $(i\e, m\e)$ as follows. First notice that, since $\Phi_k^{i-1,m}$ is a bijection and the collection of curves $\{\gamma_k^{i-1,m}\}_{k = 1, \dots, n}$ solves the Riemann problem $(u^{i-1,m-1}, u^{i-1,m})$, then there exist $a,b,s',s'' \in \R$ such that
\begin{equation}
\label{E_proiezione_phi}
\begin{split}
\Phi_k^{i-1,m-1} - a&: \W_k^{(1)}((i-1)\e, (m-1)\e) \to \I(s'), \\
\Phi_k^{i-1,m} - b&: \W_k^{(0)}((i-1)\e, m\e) \to  s' + \I(s'')
\end{split}
\end{equation}
are bijections. Set $s := s_k^{i,m}$ and define
\begin{equation}
\label{E_created_waves}
\begin{split}
\Sigma_k^{(1)}(i\e, m\e) &:= \Big[\Phi_k^{i-1,m-1} - a \Big]^{-1} \Big( \I(s') \cap \I(s'+s'') \cap \I(s) \Big), \\
\Sigma_k^{(0)}(i\e, m\e) &:= \Big[ \Phi_k^{i-1,m}  - b \Big]^{-1} \Big( \big(s'+  \I(s'')\big) \cap \I(s'+s'') \cap \I(s) \Big), \\
C_k(i\e, m\e) &:= \Big\{(\tau, i\e, m\e, k) \ \big| \ \tau \in \I(s) \setminus \I(s'+s'') \Big\}.
\end{split}
\end{equation}
Finally set
\begin{equation*}
\W_k(i\e, m\e) := \Sigma_k^{(1)}(i\e, m\e) \cup \Sigma_k^{(0)}(i\e, m\e) \cup C_k(i\e, m\e).
\end{equation*}
We will refer to $C_k(i\e, m\e)$ as the set of \emph{waves created at point $(i\e, m\e)$}. 

We have now to define the map $\Phi_k^{i,m}$.
%
%
%
First let us introduce the auxiliary map $\Psi: \Sigma_k^{(1)}(i\e, m\e) \cup \Sigma_k^{(0)}(i\e, m\e) \to \I(s) \cap \I(s'+s'')$ setting
\begin{equation}
\label{E_psi}
\Psi(w) := 
\begin{cases}
\Phi_k^{i-1,m-1}(w) - a & \text{if } w \in \Sigma^{(1)}(i\e, m\e), \\
\Phi_k^{i-1,m}(w) - b & \text{if } w \in \Sigma^{(0)}(i\e, m\e).
\end{cases}
\end{equation}
Since 
\[
\Big( \I(s') \cap \I(s'+s'') \cap \I(s) \Big)  \cap 
\Big( \big(s'+  \I(s'')\big) \cap \I(s'+s'') \cap \I(s) \Big) = \emptyset,
\]
it follows that $\Psi$ is a bijection. Define now the function $\Phi_k^{i,m}$ on $\W_k(i\e, m\e)$ as follows:
\begin{subequations}
\label{E_codom_phi_k}
\begin{equation}
\label{E_codomain_phi}
\Phi_k^{i,m} : \W_k(i\e, m\e) \to \left( \sum_{\substack{m' < m \\ \sign(s_k^{i,m'}) = \sign(s_k^{i,m})}} s_k^{i, m'} \right) + \I(s_k^{i,m}),
\end{equation}
\begin{equation}
\label{E_phi_k}
\Phi_k^{i,m}(w):= \left( \sum_{\substack{m' < m \\ \sign(s_k^{i,m'}) = \sign(s_k^{i,m})}} s_k^{i, m'} \right) + 
\begin{cases}
\Psi(w) & \text{if } w \in \Sigma^{(1)}(i\e, m\e) \cup \Sigma^{(0)}(i\e, m\e), \\
\tau & \text{if } w \in C(i\e, m\e) \text{ and } w = (\tau, i\e, m\e,k). \\
\end{cases}
\end{equation}
\end{subequations}
It is immediate to see that $\Phi_k^{i,m}$ is a bijection. We conclude this first step observing that
\begin{equation}
\label{E_Wkdisjoint}
\W_k(i\e,m\e) \cap \W_k(i\e,m'\e) = \emptyset \qquad \text{ if } m \neq m'.
\end{equation}

%
%
%


\begin{figure}
\resizebox{14cm}{11cm}{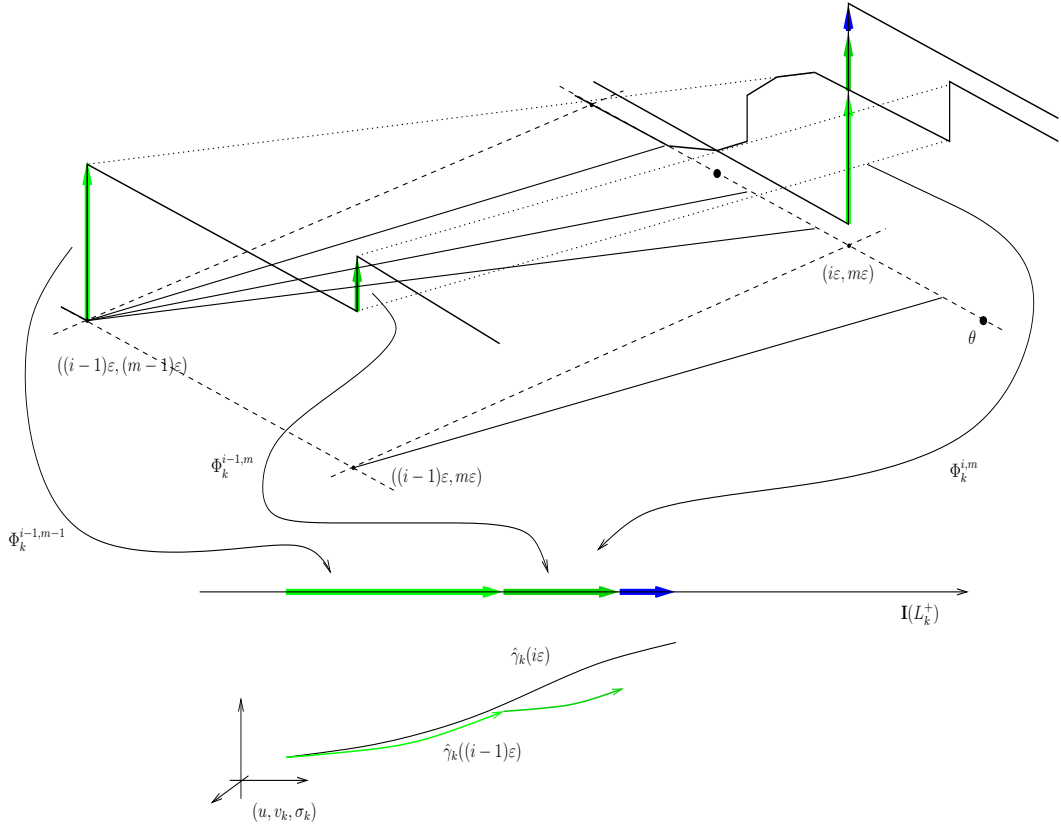}
\caption{The use of the auxillary map $\Phi_k(t)$ to induce the correct ordering on $\W_k(t)$: the colliding $k$-th waves (green) at $((i-1)\e,(m-1)\e)$, $((i-1)\e,m\e)$ are mapped to the $k$-th waves at $(i\e,m\e)$ preserving the ordering on $\I(L^+_k((i-1)\e))$, and the new waves (blue) are added to the right.}
\end{figure}

\bigskip

\noindent \textit{Step 2}. We now define all the other objects which appears in the definition of Langrangian representation.
For any $k \in \{1, \dots, n\}$, $i \in \N$,  set
\begin{equation*}
\W_k(i\e) := \bigcup_{m \in \Z} \W_k(i\e,m\e), 
\qquad
\W_k := \bigcup_{i \in \N} \W_k(i\e), 
\qquad
\W := \bigcup_{k =1}^{n} \W_k. 
\end{equation*}
It holds 
\[
\W_k(i\e), \W_k, \W \subseteq \R \times \N\e \times \Z\e \times \{1, \dots, n\}. 
\] 
Since $\W_k \cap \W_h = \emptyset$ if $k \neq h$, we can define the family of a wave $w \in \W_k$ as  ${\rm family}(w) = k $.
Now for any $k$-th wave $w \in \W_k$, define its creation and cancellation time as
\begin{equation*}
\tcr(w) := \min \big\{i \in \N \ \big| \ w \in \W_k(i\e) \big\}, \qquad
\tcanc(w) := \sup \big\{i \in \N \ \big | \ w \in \W_k(i\e) \big\} + \e;
\end{equation*}
it is not difficult to see that $w \in \W_k(i\e)$ for any $i\e \in [\tcr(w), \tcanc(w))$. 
 The sign $\mathcal S(w)$ of a wave $w \in \W$ is defined as
\begin{equation}
\label{E_def_segno}
\mathcal S(w) := \sign(s_k^{i,m}) \qquad \text{ if } w \in \W_k(i\e,m\e).
\end{equation}
To show that the definition \eqref{E_def_segno} is well posed, it is sufficient to prove the following lemma.

\begin{lemma}
\label{L_segno_cons}
For any $w \in \Sigma^{(1)}_k(i\e, m\e) \cup \Sigma^{(0)}_k(i\e, m\e)$, it holds
\[
\mathcal S\big(w\big) = \sign(s_k^{i,m}).
\]
\end{lemma}

\begin{proof}
Assume $w \in \Sigma^{(1)}_k(i\e, m\e)$, the case $w \in \Sigma^{(0)}_k(i\e, m\e)$ being completely similar. Define $a,b,s',s''$ as in \eqref{E_proiezione_phi}. It holds $\mathcal S(w) = \sign(s') $ and
\begin{equation}
\label{E_lemma_segno}
\Phi_k^{i-1,m-1}(w) - a \in \I(s') \cap \I(s'+s'') \cap \I(s),
\end{equation}
where $s = s_k^{i,m}$. From \eqref{E_lemma_segno}, since $ \I(s') \cap \I(s'+s'') \cap \I(s) \neq \emptyset$, it follows that $s s'  >0$.
\end{proof}

Set 
\begin{equation*}
\W_k^+(i\e) := \big\{w \in \W_k(i\e) \ \big| \ \mathcal S(w) = +1 \big\}, \qquad
\W_k^-(i\e) := \big\{w \in \W_k(i\e) \ \big| \ \mathcal S(w) = -1 \big\}.
\end{equation*}

Define now the position function $\mathtt x$ for times $i\e, i \in \N$, as follows: for any $w \in \W$, for any time $i\e \in [\tcr(w), \tcanc(w))$, set
\begin{equation*}
\mathtt x(i\e,w) := m\e, \quad \text{if } w \in \W_k(i\e,m\e).
\end{equation*}
The definition is well posed thanks to \eqref{E_Wkdisjoint}. 

Using \eqref{E_Wkdisjoint} we can define the map
\begin{equation*}
\Phi_k(i\e) : \W_k(i\e) \to \I\big(L_k^-(i\e) \big) \cup \I\big(L_k^+(i\e)\big)
\end{equation*}
through the formula 
\begin{equation*}
\Phi_k(i\e)|_{\W_k(i\e,m\e)} := \Phi_k^{i,m},
\end{equation*}
and the map
\begin{equation*}
\hat \gamma_k(i\e) : \W_k(i\e) \to \mathcal D_k
\end{equation*}
through the formula
\begin{equation*}
\hat \gamma_k(i\e)|_{\W_k(i\e,m\e)} := \hat \gamma_k^{i,m}.
\end{equation*}
From
\begin{equation*}
\left( \sum_{\substack{m' < m \\ \sign(s_k^{i,m'}) = \sign(s_k^{i,m})}} s_k^{i, m'} \right) + \I(s_k^{i,m}) \cap 
\left( \sum_{\substack{m' < r \\ \sign(s_k^{i,m'}) = \sign(s_k^{i,r})}} s_k^{i, m'} \right) + \I(s_k^{i,r}) = \emptyset,
\quad \text{ for $m \neq r$ },
\end{equation*}
we have that $\Phi_k(i\e)$ is a bijection. 

Next define the relation $\leq$ on $\W$ as any relation such that the maps
\begin{equation*}
\Phi_k(i\e)|_{\W_k^+(i\e)} : \W_k^+(i\e) \to \I\big(L_k^+(i\e)\big), \qquad \Phi_k(i\e)|_{\W_k^-(i\e)} : \W_k^-(i\e) \to \I\big(L_k^-(i\e)\big),
\end{equation*}
become respectively an isomorphism and an antisomorphism of ordered sets.

To prove that this definition is well posed, it is sufficient to show that the following lemma holds.

\begin{lemma}
Let $w',w'' \in \W_k((i-1)\e) \cap \W_k(i\e)$ and assume that they have the same sign. Then 
\begin{equation*}
\Phi_k((i-1)\e)(w') \leq \Phi_k((i-1)\e)(w'') \quad
\Longleftrightarrow \quad
\Phi_k(i\e)(w') \leq \Phi_k(i\e)(w'').
\end{equation*}
\end{lemma}

\begin{proof}
It is sufficient to prove the implication "$\Longrightarrow$". Suppose that both $w'$ and $w''$ are positive, the other case being similar, and that $w' \neq w''$; assume $w' \in \W_k(i\e,m'\e)$, $w'' \in \W_k(i\e,m''\e)$. Distinguish three cases:
\begin{enumerate}[label=(\roman*)]
\item $m' < m''$: in this case, by the definition of the codomains of the maps $\{\Phi_k^{i,r}\}_{r \in \Z}$ in \eqref{E_codomain_phi}, it is immediate to see that the $\Phi_k(i\e)(w') \leq \Phi_k(i\e)(w'')$;
\item $m'=m''$; using the notations as in \eqref{E_proiezione_phi}-\eqref{E_phi_k}, it is sufficient to prove that $\Psi(w') \leq \Psi(w'')$. If both $w'$ and $w''$ belong to $\Sigma_k^{(0)}(i\e,m'\e)$ (or if they both belong to $\Sigma_k^{(1)}(i\e,m'\e)$), the conclusion follows from 
\begin{equation*}
\Psi(w') = \Phi_k((i-1)\e)(w') - a \leq \Phi_k((i-1)\e)(w'') - a = \Psi(w');
\end{equation*}
if $w' \in \Sigma^{(0)}(i\e, m'\e)$, $w'' \in \Sigma^{(1)}(i\e, m'\e)$, then $s'$ and $s''$ must be greater than zero; by definition, $\Psi(w') = \Phi_k((i-1)\e,w') - a \in \I(s')$ and $\Psi(w'') = \Phi_k((i-1)\e, w'') - b \in s' + \I(s'')$, and thus, since $s', s'' > 0$, we have $\Psi(w') \leq \Psi(w'')$;
\item $m' > m''$; since by hypothesis $\Phi_k((i-1)\e)(w') \leq \Phi_k((i-1)\e)(w'')$, the only possibility is that
\begin{equation*}
w'  \in \W_k^{(1)}((i-1)\e, m''\e), \qquad w'' \in \W_k^{(0)}((i-1)\e, m''\e), \qquad m' = m''+1;
\end{equation*}
hence
\begin{equation*}
\sigma_k^{i-1,m''}(\Phi_k^{i-1,m''}(w'')) = \hat \sigma_k^{i-1,m''}(w'') \leq \vartheta_i < \hat \sigma_k^{i-1,m''}(w') = \sigma_k^{i-1,m''}(\Phi_k^{i-1,m''}(w'));
\end{equation*}
since $\sigma_k^{i-1,m''}$ is non decreasing, we get $\Phi_k^{i-1,m''}(w'') < \Phi_k^{i-1,m''}(w')$, a contradiction. \qedhere
\end{enumerate}
\end{proof}

Finally extend all the definitions above for times $t \in (i\e, (i+1)\e)$, defining $\mathtt x(t,w)$ according to \eqref{E_ode} and setting
\begin{equation*}
\Phi_k(t) := \Phi_k(i\e), \qquad \hat \gamma_k(t) := \hat \gamma_k(i\e), \quad \text{ if } t \in [i\e,(i+1)\e).
\end{equation*}
It is not difficult to see that Properties (1)-(4) of the Definition of Lagrangian representation hold.

\bigskip

\noindent \textit{Step 3}. We are left to show that the additional properties (a) and (b) hold.

\smallskip

\noindent (a) From \eqref{E_created_waves}, it follows
\begin{equation*}
\W_k(i\e,m\e) \cap \W_k((i-1)\e) = \Sigma_k^{(1)}(i\e,m\e) \cup \Sigma_k^{(0)}(i\e,m\e). 
\end{equation*}
To prove that this set is an interval of waves at time $i\e$, just observe that for any $w \in \Sigma_k^{(1)}(i\e,m\e) \cup \Sigma_k^{(0)}(i\e,m\e)$ and for any $w' \in C_k(i\e,m\e)$, it holds $\Phi_k(i\e)(w) = \Phi_k^{i,m}(w) \leq \Phi_k^{i,m}(w') = \Phi_k(i\e)(w')$ and thus $w \leq w'$. 

\noindent Let us now prove that $\W_k(i\e,m\e) \cap \W_k((i-1)\e)$ is an interval of waves also at time $(i-1)\e$. First notice that both $\Sigma_k^{(1)}(i\e,m\e)$ and $\Sigma_k^{(0)}(i\e,m\e)$ are intervals of waves at time $(i-1)\e$, since by \eqref{E_created_waves} they are pre-images of intervals through an (anti)isomorphism of ordered sets. Hence, to prove that  $\Sigma_k^{(1)}(i\e,m\e) \cup \Sigma_k^{(0)}(i\e,m\e)$ is an interval of waves at time $(i-1)\e$, take $w \in \Sigma_k^{(1)}(i\e,m\e)$, $w' \in \Sigma_k^{(0)}(i\e,m\e)$, and $z \in \W_k((i-1)\e)$ with $w \leq z \leq w'$: it is sufficient to show that $z \in \Sigma_k^{(1)}(i\e,m\e) \cup \Sigma_k^{(0)}(i\e,m\e)$. Since $\Sigma_k^{(1)}(i\e,m\e), \Sigma_k^{(0)}(i\e,m\e) \neq \emptyset$, we have that (using the same notations as in \eqref{E_proiezione_phi}-\eqref{E_created_waves}), $s',s'',s$ have the same sign (say positive) and $a = b$. Hence
\begin{equation*}
a \leq \Phi_k((i-1)\e)(w) \leq \Phi_k((i-1)\e)(z) \leq \Phi_k((i-1)\e)(w') \leq a + \min\{s'+s'',s\};
\end{equation*}
therefore $\Phi_k((i-1)\e)(z) - a \in \Big( \I(s') \cap \I(s'+s'') \cap \I(s) \Big) \cup  \Big( \big(s'+  \I(s'')\big) \cap \I(s'+s'') \Big)$ and thus $z \in \Sigma_k^{(1)}(i\e,m\e) \cup \Sigma_k^{(0)}(i\e,m\e)$, thus showing that $\W_k(i\e,m\e) \cap \W_k((i-1)\e)$ is an interval of waves at time $(i-1)\e$. 

\noindent Finally, since $\Phi_k(i\e)$ is an (anti)isomorphism of ordered set, using \eqref{E_created_waves} and \eqref{E_phi_k}, we immediately get that $C_k(i\e,m\e)$ is an interval of waves at time $i\e$, thus concluding the proof of Property (a).
\smallskip

\noindent (b) Assume both $\Sigma_k^{(1)}(i\e, m\e)$ and $\Sigma_k^{(0)}(i\e, m\e)$ are not empty and contains positive waves; if they contain negative waves or if one of them is empty, the proof is similar. Using again the same notations as in \eqref{E_proiezione_phi}-\eqref{E_phi_k}, since they both are not empty, then $a=b$ in \eqref{E_psi}. Hence $\Psi$ coincides with $\Phi_k((i-1)\e)$ up to a constant. Moreover, by \eqref{E_codomain_phi}, $\Phi_k(i\e)$ coincides with $\Psi$ up to a constant. Hence $\Phi_k(i\e) \circ \Phi_k((i-1)\e)^{-1}$ is an affine map, with slope equal to $1$. 
\end{proof}

As an immediate corollary of Properties \ref{Pt_iow_cons}-\ref{Pt_affine} we get
\begin{corollary}
\label{C_iow_same_pos}
The following hold. 
\begin{enumerate}
\item \label{Pt_arrivano} Let $i \in \N$, $m \in \Z$, $\mathcal I \subseteq \W_k^{(1)}((i-1)\e, (m-1)\e) \cup \W_k^{(0)}((i-1)\e, m\e)$ be an i.o.w. at time $(i-1)\e$. Then either $\mathcal I \cap \W_k(i\e)$ is empty or it is an i.o.w. both at time $(i-1)\e$ and at time $i\e$.
\item \label{Pt_sono} Let $\mathcal I \subseteq \W_k(i\e,m\e)$ be an i.o.w. at time $i\e$. Then either $\mathcal I \cap \W_k((i-1)\e)$ is empty or it is an i.o.w. both at time $(i-1)\e$ and at time $i\e$.
\item \label{Pt_misura_iow} For any $\mathcal I \subseteq \W_k(i\e,m\e) \cap \W_k((i-1)\e)$, it holds $\mathcal L^1\big( \Phi_k((i-1)\e)(\mathcal I)\big) = \mathcal L^1\big(\Phi_k(i\e)(\mathcal I)\big)$. 
\end{enumerate}
\end{corollary}

\subsection{Further definitions and remarks}
\label{Ss_further_def}

We conclude this section by introducing some useful notions that we will frequently use hereinafter, in particular, as we already said, the notion of \emph{effective flux $\feff_k(t)$ of the $k$-th family at time $t$}.  

\begin{definition}
Fix $\bar t \geq 0$. Let $\mathcal I \subseteq \W_k(\bar t)$ be an interval of waves at time $\bar t$. Set $I := \Phi_k(\bar t)(\mathcal I)$. By Property \eqref{Pt_isomorph_ordered_sets} of the Definition of Lagrangian representation, $I$ is an interval in $\R$ (possibly made by a single point). Let us define:
\begin{itemize}
\item \emph{the strength of $\mathcal I$} as 
\[
|\mathcal I| := \mathcal L^1(I);
\] 
\item the \emph{Rankine-Hugoniot speed given to the interval of waves $\mathcal I$ by a function $g: \R \to \R$} as
\begin{equation*}
\sigma^{\text{rh}}(g, \mathcal I) := 
\begin{cases}
\frac{g(\sup I)-g(\inf I)}{\sup I - \inf I} & \text{if  $I$ is not a singleton}, \\
g'(I) & \text{if $I$ is a singleton};
\end{cases}
\end{equation*}
\item for any $w \in \mathcal I$, the \emph{entropic speed given to the wave $w$ by the Riemann problem $\mathcal I$ and the flux function $g$} as 
\[
\sigma^{\text{ent}}(g,\mathcal I,w) :=
\left\{ \begin{array}{ll}
{\displaystyle \frac{d}{d\tau} \conv_I g \Big(\Phi_k(\bar t)(w)\Big)} & \text{if } \mathcal S_k(w) = +1, \\ [1em]
{\displaystyle \frac{d}{d\tau} \conc_I g \Big(\Phi_k(\bar t)(w)\Big)} & \text{if } \mathcal S_k(w) = -1.
\end{array} \right.
\]
\end{itemize}
If $\sigma^{\text{rh}}(g, \mathcal I) = \sigma^{\text{ent}}(g,\mathcal I,w)$ for any $w \in \mathcal I$, we will say that $\mathcal I$ is \emph{entropic} w.r.t. the function $g$.

We will also say that \emph{the Riemann problem $\mathcal I$ with flux function $g$ divides $w,w'$} if $\sigmaent(g,\mathcal I,w) \neq \sigmaent(g,\mathcal I,w')$.
\end{definition}

We recall that by definition an interval of waves is made of waves with the same sign.


\begin{remark}
\label{W_speed_increasing_wrt_waves}
Notice that $\sigmaent$ is always increasing on $\mathcal I$, whatever the sign of $\mathcal I$ is, by the monotonicity properties of the derivatives of the convex/concave envelopes. 
\end{remark}


\begin{remark}
\label{R_partition_iow}
Given a function $g$ and an interval of waves $\mathcal I$, we can always partition $\mathcal I$ through the equivalence relation
\[
z \sim z'  \quad \Longleftrightarrow \quad z,z' \text{ are not divided by the Riemann problem $\mathcal I$ with flux function $g$}.
\]
As a consequence of Remark \ref{W_speed_increasing_wrt_waves}, we have that each element of this partition is an entropic interval of waves and the relation induced by the order $\leq$ on the partition (see Section \ref{Ss_notation}) is still a total order.
\end{remark}

Let us conclude this section, introducing the following notion.

\begin{definition}
\label{D_effect_flux}
For each family $k =1, \dots n$ and for each time $t \geq 0$ define the \emph{effective flux of the $k$-th family at time $t$} as any function
\begin{equation*}
\feff_k(t, \cdot): [L_k^-, L_k^+] \to \R
\end{equation*}
whose second derivative satisfies the following relation:
\begin{equation*}
\begin{split}
&\frac{\partial^2 \feff_k(t, \cdot)}{\partial \tau^2}(\tau) := 
\frac{d \tilde \lambda(\hat \gamma(t, w))}{d \tau},
\end{split}
\end{equation*}
for $\mathcal L^1$-a.e. $\tau \in [L_k^-, L_k^+]$, where $w = \Phi_k(t)^{-1}(\tau)$.
\end{definition}

\begin{remark}
Let us observe the following:
\begin{enumerate}
\item $\feff_k(t, \cdot)$ is defined up to affine function; 
\item since the second derivative of $\feff_k(t, \cdot)$ is an $L^\infty$-function, it turns out that $\feff_k(t, \cdot)$ is a $C^{1,1}$-function;
\item $\feff_k(t, \cdot) = \feff_k(i\e, \cdot)$ for any $t \in [i\e, (i+1)\e)$;
\item it is quite easy to see that for any time $i\e, i\in \N$ and for any $m \in \Z$, on the interval 
\begin{equation*}
\left( \sum_{\substack{m' < m \\ \sign(s_k^{i,m'}) = \sign(s_k^{i,m})}} s_k^{i, m'} \right) + \I(s_k^{i,m})
\end{equation*}
the $k$-th effective flux function $\feff_k(i\e, \cdot)$ coincides, up to an affine function, with the $k$-th reduced flux associated to the Riemann problem located at $(i\e, m\e)$ as defined in \eqref{E_reduced_flux}.
\end{enumerate} 
\end{remark}

\section{Analysis of waves collision}
\label{S_analysis_wave_coll}

Starting with this section we enter in the heart of our construction. We introduce in fact the notion of \emph{pair of waves $(w,w')$ which have already interacted} and \emph{pair of waves $(w,w')$ which have never interacted} at time $\bar t$. For any pair of waves $(w,w')$ and for a fixed time $\bar t$, we define an interval of waves $\mathcal I(\bar t,w,w')$ and a partition $\mathcal P(\bar t,w,w')$ of this interval: these objects in some sense summarize the past ``common'' history of the two waves, from the moment in which they have split (if they have already interacted) or from the last time in which one of them is created, if they have never interacted.

The interval $\mathcal I(\bar t,w,w')$ and its partition $\mathcal P(\bar t,w,w')$ will be crucial in order to define the functional $\mathcal Q_k$ in Section \ref{S_fQ} and to prove that it satisfies the inequality \eqref{E_bound_on_fQ}. 


\subsection{Wave packages}
\label{Ss_wave_pack}

We start by defining an equivalence relation between waves, which will be useful to pass from the uncountable sets of waves $\W(t)$ at time $t$ to the finite quotient set, whose elements will be called \emph{wave packages}.

For any $\bar t \geq 0$ and $w \in \W_k(\bar t)$, $\bar t \in [i\e, (i+1)\e)$, define the \emph{wave package to which $w$ belongs} as the set
\begin{equation}
\label{E_big_waves}
\mathcal E(\bar t,w) := \bigg\{ w'  \in \W_{k}(\bar t) \ \Big| \ \tcr(w) = \tcr(w'), \ 
\mathtt x (t,w) = \mathtt x(t,w') \text{ for all }
t \in \big[ \tcr(w), (i+1)\e \big) \bigg\}.
\end{equation}
In Section \ref{Ss_interacting} we will denote this equivalence relation as $\bowtie$.

\begin{remark}
Notice that is it natural to require that the condition in \eqref{E_big_waves} holds on the time interval $[\tcr(w), (i+1)\e)$ instead of $[\tcr(w), i\e]$ since it could happen that $\mathtt x(i\e, w) = \mathtt x(i\e, w')$, but $\mathtt x(t,w) \neq \mathtt x(t,w')$ for $t > i\e$, while we want to give definitions which are "left-continuous in time".
\end{remark}

\begin{lemma}
\label{L_discrete_partition}
The collection 
$\big\{ \mathcal E(\bar t, w) \ \big| \ w \in \W(\bar t) \big\}$
is a finite partition of $\W(\bar t)$ and the order induced by the $\leq$ is a total order both on the set $\big\{ \mathcal E(\bar t, w) \ \big| \ w \in \W^+_{k}(\bar t) \big\}$ and on the set $\big\{ \mathcal E(\bar t, w) \ \big| \ w \in \W^-_{k}(\bar t) \big\}$, $k = 1,\dots,n$.
\end{lemma}

\begin{proof}
Clearly $\big\{ \mathcal E(\bar t, w) \ \big| \ w \in \W(\bar t) \big\}$ is a partition of $\W(\bar t)$. To see that it is finite, just observe that the curve $\mathtt x(t,\cdot)$ is uniquely determined by assigning the points $m\e = \mathtt x(i\e,\cdot)$, and for all fixed time $\bar t$ the set of nodal points supporting $D_x u_\e(t)$, $t \leq \bar t$, is finite. 
Finally, the monotonicity of $\mathtt x(\bar t, \cdot)$ implies the statement about the order.
\end{proof}

%

\subsection{Characteristic interval}
\label{W_waves_collision}

We now define the notion of pairs of waves which \emph{have never interacted before a fixed time $\bar t$} and pairs of waves which \emph{have already interacted at a fixed time $\bar t$} and to any pair of waves $(w,w')$ we will associate an interval of waves $\mathcal I(\bar t,w,w')$.

\begin{definition}
\label{D_interagite_non_interagite}
Let $\bar t$ be a fixed time and let $w,w' \in \W_k(\bar t)$. We say that 
\begin{itemize}
\item \emph{$w,w'$ interact at time $\bar t$} if $\mathtt x(\bar t, w) = \mathtt x(\bar t, w')$;
\item \emph{$w,w'$ have already interacted at time $\bar t$} if there is $t \leq \bar t$ such that $w,w'$ interact at time $t$;
\item \emph{$w,w'$ have never interacted at time $\bar t$} if for any $t \leq \bar t$, they do not interact at time $t$. 
\item \emph{$w,w'$ will interact after time $\bar t$} if there is $t > \bar t$ such that $w,w'$ interact at time $t$. 
\item \emph{$w,w'$ are joined in the real solution at time $\bar t$} if there is a right neighborhood of $\bar t$, say $[\bar t, \bar t+ \zeta)$, such that they interact at any time $t \in [\bar t, \bar t+\zeta)$;
\item \emph{$w,w'$ are divided in the real solution at time $\bar t$} if they are not joined at time $\bar t$.
\end{itemize}
\end{definition}

\begin{lemma}
\label{L_interagite_stesso_segno}
Assume that the waves $w,w'$ have already interacted at time $\bar t$. Then they have the same sign.
\end{lemma}

The proof is an easy consequence of Property \eqref{Point_2c_lagra} of the definition of Lagrangian representation, page \pageref{Point_2c_lagra}.

\begin{remark}
\label{R_divise_solo_in_cancellazioni}
It $\bar t \neq i\e$ for each $i \in \N$, then two waves are divided in the real solution if and only if they have different position. If $\bar t = i\e$, they are divided if there exists a time $t > \bar t$, arbitrarily close to $\bar t$, such that $w,w'$ have different positions at time  $t$. 
\end{remark}

\begin{definition}
Fix a time $\bar t$ and two $k$-waves $w, w' \in \W_k(\bar t)$, $w < w'$. Assume that $w,w'$ are divided in the real solution at time $\bar t$. Define the \emph{time of last splitting} $\tsp(\bar t,w,w')$ (if $w,w'$ have already interacted at time $\bar t$) and the \emph{time of next interaction} $\tint(\bar t, w, w')$ (if $w,w'$ will interact after time $\bar t$) by the formulas
\begin{subequations}
\begin{equation*}
\tsp(\bar t,w,w') := \max \big\{ t \leq \bar t \ | \ \mathtt x(t,w) = \mathtt x(t,w') \big\},
\end{equation*}
\begin{equation*}
\tint(\bar t, w,w') := \min \big\{ t > \bar t \ | \ \mathtt x(t,w) = \mathtt x(t,w') \big\}.
\end{equation*}
\end{subequations}
(In the case one of sets is empty we assume the corresponding time to be $\pm \infty$.)
Moreover set
\begin{equation*}
\xsp(\bar t,w,w') := \mathtt x(\tsp(\bar t,w,w'), w) = \mathtt x(\tsp(\bar t,w,w'),w')
\end{equation*}
and
\begin{equation*}
\xint(\bar t,w,w') := \mathtt x(\tint(\bar t,w,w'),w) = \mathtt x(\tint(\bar t,w,w'),w'),
\end{equation*}
whenever defined.
\end{definition}

Observe that $\tsp(\bar t,w,w'), \tint(\bar t,w,w') \in \N\e$ and $\xsp(\bar t,w,w'), \xint(\bar t,w,w') \in \Z\e$. 

\begin{definition}
\label{D_char_interval}
Let $w,w' \in \W(\bar t)$ be divided in the real solution at time $\bar t$ and assume they have the same sign. Define the \emph{characteristic interval of $w,w'$ at time $i\e$}, denoted by $\mathcal I(\bar t, w, w')$, as follows. \\
First we define $\mathcal I (\bar t,w,w')$ for times $\bar t = i\e$, $i \in \N$. 
\begin{enumerate}
\item If $w,w'$ have never interacted at time $i\e$, set
\begin{equation}
\label{E_I_mai_int}
\mathcal I(i\e,w,w') = 
\left\{ \begin{array}{ll}
\big\{z \in \W_k(i\e) \ \big| \ \mathcal S(z) = \mathcal S(w) \text{ and }   z < \E(i\e,w')\big\} \cup \E(i\e,w') & \text{if } \tcr(w) \leq \tcr(w'), \\ [1em]
\E(i\e,w) \cup \big\{ z \in \W_k(i\e) \ \big| \ \mathcal S(z) = \mathcal S(w) \text{ and }  z > \E(i\e,w) \big\} & \text{if } \tcr(w) > \tcr(w'); \\
\end{array} \right.
\end{equation}
\item If $w,w'$ have already interacted at time $i\e$, argue by recursion:
\begin{enumerate}
\item \label{Point_2a_def_inter} if $i\e = \tsp(i\e,w,w')$, set
\[
\mathcal I(i\e,w,w') := \W(i\e, \mathtt x(i\e,w)) = \W(i\e, \mathtt x(i\e,w'));
\]
\item \label{Point_2b_def_inter} if $i\e > \tsp(i\e,w,w')$, define $\mathcal I(i\e,w,w')$ as the smallest interval in $(\W_k^\pm(i\e), \leq)$ which contains $\mathcal I((i-1)\e, w,w') \cap \W_k(i\e)$, i.e.
\[
\begin{split}
\mathcal I(i\e,w,w') := 
\Big\{
z \in \W_k(i\e)\ & \Big|\  \mathcal S(z) = \mathcal S(w) = \mathcal S(w') \\ 
& \text{ and } \exists \, y,y' \in \mathcal I((i-1)\e,w,w') \cap \W_k(i\e)
\text{ such that } y \leq z \leq y'
\Big\}.
\end{split}
\]
\end{enumerate}
\end{enumerate}
Finally set
\[
\mathcal I(\bar t,w,w') := \mathcal I(i\e,w,w') \qquad \text{ for } \bar t \in [i\e,(i+1)\e).
\]
\end{definition}


\begin{lemma}
\label{L_iow_tecn}
Let $w,w' \in \W_k(i\e)$ be divided in the real solution at time $i\e$ and assume that they have the same sign. Then the following hold:
\begin{enumerate}
\item \label{Pt_I_is_iow} $\mathcal I(i\e,w,w')$ is an interval of waves at time $\bar t$;
\item \label{Pt_I_as_big_waves} $\mathcal I(i\e,w,w') = \bigcup_{z \in \mathcal I(i\e,w,w')} \E(i\e,z)$.
\end{enumerate}
Moreover, if $w,w'$ have already interacted at time $i\e$,
\begin{enumerate}[resume]
\item \label{Pt_I_waves_survived} if $i\e > \tsp(i\e,w,w')$, then 
\begin{equation*}
\mathcal I(i\e,w,w') \cap \W_k((i-1)\e) = 
\mathcal I((i-1)\e,w,w') \cap \W_k(i\e);
\end{equation*}
\item \label{Pt_z_inserita_dopo} if $z \in \mathcal I(i\e,w,w')$, then 
\[
z \in \mathcal I(t,w,w') \qquad \text{ for any } t \in \big[\max\big\{\tsp(i\e,w,w'), \tcr(z)\big\}, i\e\big];
\]
\item \label{Pt_I_contains_created} if $z,z' \in \W_k(i\e)$, $\tcr(z) = \tcr(z') = i\e$ and $\mathtt x(i\e,z) = \mathtt x(i\e,z')$, then
$z \in \mathcal I(i\e,w,w')$ if and only if $z' \in \mathcal I(i\e,w,w')$.
\end{enumerate}
\end{lemma}
\begin{proof}
\textit{Point \eqref{Pt_I_is_iow}}. Immediate from the definition of $\mathcal I(i\e,w,w')$.

\smallskip    

\noindent \textit{Point \eqref{Pt_I_waves_survived}}.  The inclusion "$\supseteq$" is straightforward. To prove the inclusion "$\subseteq$", take $z \in \mathcal I(i\e,w,w') \cap \W_k((i-1)\e)$. By definition of $\mathcal I(i\e,w,w')$, there are $y,y' \in \mathcal I((i-1)\e,w,w') \cap \W_k(i\e)$ such that $y \leq z \leq y'$. Since $z \in \W_k((i-1)\e)$ and, by Point  \eqref{Pt_I_is_iow}, $\mathcal I((i-1)\e,w,w')$ is an interval of waves at time $(i-1)\e$, it must be $z \in \mathcal I((i-1)\e,w,w')$.  

\smallskip

\noindent \textit{Point \eqref{Pt_z_inserita_dopo}}. Easy consequence of Point \eqref{Pt_I_waves_survived}.

\smallskip

\noindent \textit{Point \eqref{Pt_I_contains_created}}. 
Set $m\e :=  \mathtt x(i\e,z) = \mathtt x(i\e,z')$. 
By symmetry, it is sufficient to prove one implication. Let $z \in \mathcal I(i\e,w,w')$. Since $z,z'$ are created at time $i\e$, we have $z,z' \in C_k(i\e,m\e)$ (the case where $\tsp(t,w.w') = i\e$ is straightforward). From $z \in \mathcal I(i\e,w,w')$, we deduce that there exist $y,y' \in \mathcal I((i-1)\e,w,w') \cap \W_k(i\e)$ such that $y \leq z \leq y'$. We want to prove that $y \leq z' \leq y'$. Assume by contradiction that $y' < z$. This implies $z \leq y' < z$. Since $z,z' \in C_k(i\e,m\e)$, $y' \in \W_k(i\e)$ and, by Point \ref{Pt_iow_cons} of Theorem \ref{T_lagrangian}, $C_k(i\e,m\e)$ is an interval of waves at time $i\e$, we get $y' \in C_k(i\e,m\e)$, a contradiction, since $y' \in \W_k((i-1)\e)$. In a similar way one proves that $y \leq z'$ and thus we get $z' \in \mathcal I(i\e,w,w')$.

\smallskip

\noindent \textit{Point \eqref{Pt_I_as_big_waves}}. 
If $w,w'$ have never interacted, the proof is an immediate consequence of the definition \eqref{E_I_mai_int} and Lemma \ref{L_discrete_partition}. Assume thus $w,w'$ have already interacted at time $i\e$ and argue by induction on $i$. The only non-trivial inclusion is "$\supseteq$". Let $z \in \mathcal I(i\e,w,w')$, $z' \in \E(i\e,z)$; we would like to prove that $z' \in \mathcal I(i\e,w,w')$.
\begin{enumerate}
\item If $i\e = \tsp(i\e,w,w')$, then, $\mathtt x(i\e, w) = \mathtt x(i\e,w') = \mathtt x(i\e,z) = \mathtt x(i\e,z')$ and thus $z' \in \mathcal I(i\e,w,w')$ by definition (Point \eqref{Point_2a_def_inter} of Definition \ref{D_char_interval}).
\item If $i\e > \tsp(i\e,w,w')$, assume that the statement is proved at time $(i-1)\e$, i.e.
\begin{equation}
\label{E_I_union_E}
\mathcal I((i-1)\e,w,w') = \bigcup_{y \in \mathcal I((i-1)\e,w,w')} \E((i-1)\e,y).
\end{equation}
Distinguish two cases:
\begin{enumerate}
\item if $\tcr(z) = \tcr(z') = i\e$, then $\mathtt x(i\e,z) = \mathtt x(i\e,z')$ and thus, by Point \eqref{Pt_I_contains_created}, $z' \in \mathcal I(i\e,w,w')$;
\item if $\tcr(z) = \tcr(z') < i\e$, since $z \in \mathcal I(i\e,w,w') \cap \W_k((i-1)\e)$, then, by Point \eqref{Pt_I_waves_survived}, $z \in \mathcal I((i-1)\e,w,w')$ and thus by \eqref{E_I_union_E}, $\E((i-1),z) \subseteq \mathcal I((i-1)\e,w,w')$. Since $z' \in \E((i-1)\e,z)$, we get 
\[
z' \in \mathcal I((i-1)\e,w,w') \cap \W_k(i\e) = \mathcal I(i\e,w,w') \cap \W_k((i-1)\e),
\]
again by Point \eqref{Pt_I_waves_survived}.
\end{enumerate}
\end{enumerate}
\end{proof}

\subsection{Partition of the characteristic interval}
\label{Ss_partition}

We now define a partition $\mathcal P(\bar t,w,w')$ of the interval of waves $\mathcal I(\bar t,w,w')$ for any time $\bar t$ such that $w,w'$ are divided in the real solution at time $\bar t$, with the properties that each element of $\mathcal P(\bar t, w,w')$ is an  interval of waves at time $i\e$, entropic w.r.t. the flux $\feff_k(\bar t)$ of Definition \ref{D_effect_flux}.

We first give the definition by recursion on times $i\e, i \in \N$. 
\begin{enumerate}
\item If $w,w'$ have never interacted at time $i\e$, the equivalence classes of the partition $\mathcal P(i\e, w,w')$ are singletons.
\item Assume now that $w,w'$ have already interacted and are divided in the real solution at time $i\e$:
\begin{enumerate}
\item if $i\e = \tsp(i\e, w,w')$, then $\mathcal P(i\e,w,w')$ is given by the equivalence relation
\begin{eqnarray*}
z \sim z'  \quad \Longleftrightarrow \quad
\left\{ \begin{array}{l}
z,z' \text{ are not divided
by the Riemann problem $\W_k(i\e, \mathtt x(i\e,w))$} \\ [.2em]
\text{with flux function $\feff_k(i\e, \cdot)$};
\end{array} \right.
\end{eqnarray*}
\item \label{Point_2b_part_I} if $i\e > \tsp(i\e, w,w')$, (i.e. $w,w'$ are divided in the real solution also at time $(i-1)\e$), then $\mathcal P(i\e, w,w')$ is given by the equivalence relation
\begin{eqnarray*}
z \sim z' \quad \Longleftrightarrow \quad 
\left\{ \begin{array}{l}
\left[ \begin{array}{l}
z,z' \text{ belong to the same equivalence class $\mathcal J \in\mathcal P((i-1)\e,w,w')$ and} \\ [.2em]
\text{the Riemann problem $\mathcal J \cap \W(i\e)$ with flux $\feff_k(i\e, \cdot)$ does not divide them} 
\end{array} \right] \\ [1em]
\text{or } \\ [.5em]
\Big[ \tcr(z) = \tcr(z') = i\e \text{ and } z = z' \Big].
\end{array} \right.
\end{eqnarray*}
\end{enumerate}
\end{enumerate}
Observe that the previous definition is well posed, provided that $\mathcal J \cap \W(i\e)$ is an  interval of waves at time $i\e$. This will be an easy consequence of Proposition \ref{P_divise_partizione_implica_divise_realta} and Corollary \ref{C_iow_same_pos}, Point \eqref{Pt_arrivano}.

Finally extend the definition of $\mathcal P(\bar t,w,w')$ also for times $\bar t \in (i\e, (i+1)\e)$, setting
\begin{equation*}
\mathcal P(\bar t,w,w') = \mathcal P(i\e,w,w') \qquad \text{ for any }\bar t \in [i\e, (i+1)\e).
\end{equation*}

\begin{remark}
As a consequence of Remark \ref{R_partition_iow} 
we immediately see that each element of the partition $\mathcal P(\bar t,w,w')$ is an entropic interval of waves (w.r.t. the flux function $\feff_k(\bar t, \cdot)$) and the relation induced on $\mathcal P(\bar t,w,w')$ by the order $\leq$ is still a total order on $\mathcal P(\bar t,w,w')$.
\end{remark}

Let us prove now some properties of the partition $\mathcal P(\bar t,w,w')$.

\begin{lemma}
\label{L_big_waves}
Let $\bar t$ be a fixed time and let $w,w',z,z' \in \W_k(\bar t)$, $z \in \E(\bar t,w)$, $z' \in \E(\bar t,w')$. 
Then 
\begin{equation*}
\mathcal I(\bar t,w,w') = \mathcal I(\bar t,z,z') \qquad \text{and} \qquad \mathcal P(\bar t,w,w') = \mathcal P(\bar t,z,z').
\end{equation*}
\end{lemma}

\begin{proof}
The proof is  an easy consequence of the previous definitions. 
%
\end{proof}

\begin{lemma}
\label{L_partition_at_different_times}
Let $t_1,t_2 \geq 0$, $w,w' \in \W_k(t_1) \cap \W_k(t_2)$. Assume that $w,w'$ are divided in the real solution at time $t_1$ and 
\begin{equation*}
0 \leq \tsp(t_1,w,w') \leq t_1 \leq t_2 < \tint(t_1,w,w'). 
\end{equation*}
Let $\mathcal J \in \mathcal P(t_2,w,w')$. Then either $\mathcal J \cap \W_k(t_1) = \emptyset$ or $\mathcal J \cap \W_k(t_1) = \mathcal J$ and $\mathcal J$ is an  interval of waves at time $t_1$.
\end{lemma}

\begin{proof}
It is sufficient to prove the lemma for $t_1, t_2 \in \N\e$. Fix $t_1 \in \N\e$, $w,w' \in \W_k(t_1)$ as in the statement of the lemma. We prove the lemma by induction on times $t_2 \in \N\e$, $t_2 = t_1, \dots, \tint(t_1,w,w') - \e$.

If $t_2 = t_1$ the proof is trivial. Hence assume that the lemma is proved for time $t_2 - \e$ and let us prove it for time $t_2 \in \N\e$, with $t_1 + \e \leq t_2 \leq \tint(t_1,w,w') - \e$. Let $\mathcal J \in \mathcal P(t_2,w,w')$ and assume that $\mathcal J \cap \W_k(t_1) \neq \emptyset$. Let $z \in \mathcal J \cap \W_k(t_1)$, $z' \in \mathcal J$. Since $z \sim z'$ at time $t_2$ and $z \in \W_k(t_1)$, with $t_1 < t_2$, by definition of equivalence classes, there must be $\mathcal K \in \mathcal P(t_2 -\e, w,w')$ such that $z,z' \in \mathcal K$ and $\mathcal K \supseteq \mathcal J$. By inductive assumption, $\mathcal K \cap \W_k(t_1) = \mathcal K$ and thus
\begin{equation*}
\mathcal J \cap \W_k(t_1) = \mathcal J \cap \mathcal K \cap \W_k(t_1) = \mathcal J \cap \mathcal K = \mathcal J,
\end{equation*}
thus proving the first part of the statement.

Let now $z,z' \in \mathcal J \subseteq \mathcal K$, $y \in \W_k(t_1)$, $z \leq y \leq z'$. By inductive assumption $y \in \mathcal K$; since $\mathcal K \cap \W_k(t_2)$ with flux function $\feff_k(t_2,\cdot)$ does not divide $z,z'$ and $z \leq y \leq z'$, we have that $\mathcal K \cap \W_k(t_2)$ does not divide $z,z',y$ and thus $y \in \mathcal J$, thus proving also the second part of the lemma.
\end{proof}

\begin{proposition}
\label{P_divise_partizione_implica_divise_realta}
Let $\bar t \geq 0$, let $w,w' \in \W_k(\bar t)$ be divided in the real solution at time $\bar t$ and have the same sign and let $\mathcal J \in \mathcal P(\bar t, w,w')$. Then $\mathtt x(\bar t, \cdot)$ is constant on $\mathcal J$ and the Riemann problem $\W(\bar t, \mathtt x(\bar t, \mathcal J))$ with flux function $\feff_k(\bar t, \cdot)$ does not divide waves in $\mathcal J$, i.e. also the map $z \mapsto \hat \sigma_k(\bar t)(z)$ introduced in Point \eqref{Point_4_lagr_repr} of the definition of Lagrangian representation is constant on $\mathcal J$.
%
%
\end{proposition}

\begin{proof}
Clearly it is sufficient to prove the proposition for times $i\e, i \in \N$, since if the proposition is proved at time $i\e$ it holds for times $t \in [i\e, (i+1)\e)$. Hence let $\bar t = i\e$ for some $i \in \N$. If $w,w'$ have never interacted at time $i\e$, then the proof is trivial, because $\mathcal J$ is a singleton.

Assume thus that $w,w'$ have already interacted at time $i\e$. Let $\mathcal J \in \mathcal P(i\e,w,w')$ and let $z,z' \in \mathcal J$. We want to prove that 
\begin{equation}
\label{E_stessa_pos_stessa_vel}
\mathtt x(i\e,z) = \mathtt x(i\e,z') \qquad \text{ and } \qquad \hat \sigma_k(i\e,z) = \hat \sigma_k(i\e,z'). 
\end{equation}
We argue by induction on $i$.
\begin{enumerate}
\item If $i\e = \tsp(i\e,w,w')$, then \eqref{E_stessa_pos_stessa_vel} is an immediate consequence of the definition of $\mathcal P(i\e,w,w')$.
\item If $i\e > \tsp(i\e,w,w')$, two cases arise:
\begin{enumerate}
\item $\tcr(z) = \tcr(z') = i\e$ and $z=z'$: in this case the conclusion is trivial;
\item there is $\mathcal K \in \mathcal P((i-1)\e,w,w')$ such that $z,z' \in \mathcal K$ and the Riemann problem $\mathcal K \cap \W_k(i\e)$ with flux function $\feff_k(i\e,\cdot)$ does not divide $z,z'$ (Point \eqref{Point_2b_part_I} above); in this case by the inductive assumption $\mathtt x((i-\tfrac{1}{2})\e,\cdot)$ is constant on $\mathcal K$ and thus we can set
\begin{equation*}
m\e := \mathtt x(i\e,\mathcal K \cap \W(i\e)).
\end{equation*}
Moreover, since the Riemann problem $\mathcal K \cap \W_k(i\e) = \mathcal K \cap \W_k(i\e,m\e)$ does not divide $z,z'$, by Proposition \ref{differenza_vel} also the Riemann problem $\mathcal W_k(i\e,m\e)$ does not divide $z,z'$, i.e. $\hat \sigma_k(i\e,z) = \hat \sigma_k (i\e,z')$. \qedhere
\end{enumerate}
\end{enumerate}
\end{proof}

\begin{definition}
Let $A,B$ two sets, $A \subseteq B$. Let $\mathcal P$ be a partition of $B$. We say that $\mathcal P$  \emph{can be restricted to $A$} if for any $C \in \mathcal P$, either $C \subseteq A$ or $C \subseteq B \setminus A$. We also write
\[
\mathcal P|_{A} := \big\{ C \in \mathcal P \ \big| \ C \subseteq A \big\}.
\]
\end{definition}

Clearly $\mathcal P$ can be restricted to $A$ if and only if it can be restricted to $B \setminus A$.

\begin{proposition}
\label{P_partition_restr}
Let $\bar t \geq 0$ be a fixed time. Let $w,w',z,z' \in \W(\bar t)$, $z \leq w < w' \leq z'$, assume that $w,w'$ are divided in the real solution at time $\bar t$ and they have the same sign. Then $\mathcal P(\bar t, z,z')$ can be restricted both to $\mathcal I (\bar t,z,z') \cap \mathcal I(\bar t,w,w')$ and to $\mathcal I(\bar t,z,z') \setminus \mathcal I(\bar t,w,w')$.
%
\end{proposition}
\begin{proof}
As before, it is sufficient to prove the proposition for times $i\e, i \in \N$. 
If $z,z'$ have never interacted at time $i\e$, the proof is immediate being the equivalent classes singletons. Hence, assume that $z,z'$ have already interacted at time $i\e$. Let $\mathcal J \in \mathcal P(i\e,z,z')$ such that $\mathcal J \cap \mathcal I(i\e,w,w') \neq \emptyset$. We want to prove that $\mathcal J \subseteq \mathcal I(i\e,w,w')$.

Assume first that $w,w'$ have never interacted at time $i\e$. Suppose w.l.o.g. that $\tcr(w) \leq \tcr(w')$, the case $\tcr(w) > \tcr(w')$ being analogous. Since $w,w'$ have never interacted at time $i\e$, while $z,z'$ have already interacted, it must hold $\tcr(w') > \tsp(i\e,z,z')$. It holds
\begin{equation*}
\emptyset \neq \mathcal J \cap \mathcal I(i\e,w,w') = \bigg( \mathcal J \cap \Big\{ y \in \W_k(i\e) \ big| \ \mathcal S(y) = \mathcal S(w) \text{ and }  y < \E(i\e,w') \Big\} \bigg) \cup 
\bigg( \mathcal J \cap \E(i\e,w') \bigg).
\end{equation*}
Distinguish two cases:
\begin{enumerate}
\item if $\mathcal J \cap \E(i\e,w') \neq \emptyset$, since $\tcr(w') > \tsp(i\e,z,z')$, $\mathcal J$ is a singleton by Point \eqref{Point_2b_part_I}, page \pageref{Point_2b_part_I}, and thus $\mathcal J \subseteq \E(i\e,w') \subseteq \mathcal I(i\e,w,w')$;
\item otherwise, if $\mathcal J \cap \E(i\e,w') = \emptyset$ and $\mathcal J \cap   \big\{ y \in \W_k(i\e) \ \big| \ \mathcal S(y) = \mathcal S(w) \text{ and }  y < \E(i\e,w') \big\} \neq \emptyset$,  since $\mathcal J$ is an  interval of waves and $\E(i\e,w') \neq \emptyset$, it must hold $\mathcal J \subseteq \big\{ y \in \W_k(i\e) \ \big| \ y < \E(i\e,w') \big\} \subseteq \mathcal I(i\e,w,w')$.
\end{enumerate}

\noindent Assume now $w,w'$ have already interacted at time $i\e$. We argue by induction.
\begin{enumerate}
\item If $i\e = \tsp(i\e,w,w')$, then $\mathcal I(i\e,w,w') = \W_k(i\e, \mathtt x(i\e,w))$ and thus $\mathcal J \cap \W_k(i\e, \mathtt x(i\e,w)) \neq \emptyset$. By Proposition \ref{P_divise_partizione_implica_divise_realta}, it must hold $\mathcal J \subseteq \W_k(i\e, \mathtt x(i\e,w)) = \mathcal I(i\e,w,w')$.
\item If $i\e > \tsp(i\e,w,w')$, assume that the proposition is proved for time $(i-1)\e$. Distinguish two more cases:
\begin{enumerate}
\item at least one wave in $\mathcal J$ is created at time $i\e$; in this case, $\mathcal J$ is a singleton and thus $\mathcal J \subseteq \mathcal I(i\e,w,w')$;
\item all the waves in $\mathcal J$ already exist at time $(i-1)\e$; in this case by the definition of $\mathcal P(i\e,z,z')$, there is $\mathcal K \in \mathcal P((i-1)\e,z,z')$ such that $\mathcal J \subseteq \mathcal K$. Now observe that
\begin{equation*}
\begin{split}
\emptyset & \neq \mathcal J \cap \mathcal I(i\e,w,w') \\
& = \mathcal J \cap \mathcal I(i\e,w,w') \cap \W_k((i-1)\e) \\
\text{(by Lemma \ref{L_iow_tecn}, Point \eqref{Pt_I_waves_survived})} & = \mathcal J \cap  \mathcal I((i-1)\e,w,w') \cap \W_k(i\e) \\
& \subseteq \mathcal K \cap \mathcal I((i-1)\e,w,w') \cap \W_k(i\e) \\
& \subseteq \mathcal K \cap \mathcal I((i-1)\e,w,w').
\end{split}
\end{equation*}
Hence, by inductive assumption, $\mathcal K \subseteq \mathcal I((i-1)\e,w,w')$ and thus we can conclude, noticing that
\begin{equation*}
\mathcal J \subseteq \mathcal K \cap \W_k(i\e) \subseteq 
\mathcal I((i-1)\e,w,w') \cap \W_k(i\e) =
\mathcal I(i\e,w,w') \cap \W_k((i-1)\e) \subseteq \mathcal I(i\e,w,w'),
\end{equation*}
where we have again used Lemma \ref{L_iow_tecn}, Point \eqref{Pt_I_waves_survived}. \qedhere
\end{enumerate}
\end{enumerate}
%
%
%
%
%
%
%
%
%
\end{proof}

\begin{proposition}
\label{P_stessa_part}
Let $\bar t \geq 0$ be a fixed time and let $w,w', z,z' \in \W_k(\bar t)$, $z \leq w < w' \leq z'$, and assume that $w,w'$ are divided in the real solution at time $\bar t$.
\begin{enumerate}
\item \label{Point_1_stessa_part} If $w,w'$ have already interacted at time $\bar t$, $z,z' \in \mathcal I(\bar t,w,w')$ and $\tcr(z), \tcr(z') \leq \tsp(\bar t,w,w')$, then $\mathcal I(\bar t,z,z') = \mathcal I(\bar t,w,w')$ and $\mathcal P(\bar t,z,z') = \mathcal P(\bar t,w,w')$.
\item \label{Point_2_stessa_part} If $w,w'$ have already interacted at time $\bar t$, but at least one wave between $z,z'$ is created after $\tsp(\bar t,w,w')$, then $z,z'$ have never interacted at time $\bar t$.
\item \label{Point_3_stessa_part} If $w,w'$ have never interacted at time $\bar t$, 
\begin{itemize}
\item if $\tcr(w) \leq \tcr(w')$ and $z' \in \E(\bar t,w')$, then
$z,z'$ have never interacted at time $\bar t$; 
\item if $\tcr(w) > \tcr(w')$ and $z \in \E(\bar t,w)$, then 
$z,z'$ have never interacted at time $\bar t$.
\end{itemize}
\end{enumerate}
\end{proposition}

\begin{proof}
Let us prove the first point by recursion. As before it is sufficient to prove the proposition for times $i\e$, $i \in \N$. If $i\e = \tsp(i\e,w,w')$, then the proof is obvious. Let thus $i\e > \tsp(i\e,w,w')$ and assume that the proposition holds at time $(i-1)\e$. Then by Point \eqref{Point_2b_def_inter} of Definition \ref{D_char_interval}
\[
\begin{split}
\mathcal I(i\e,w,w') &= 
\Big\{
y \in \W_k(i\e)\ big|\  \mathcal S(y) = \mathcal S(w) = \mathcal S(w') \\ 
& \qquad \qquad \qquad \text{ and } \exists\, \tilde y,\tilde y' \in \mathcal I((i-1)\e,w,w') \cap \W_k(i\e)
\text{ such that } \tilde y \leq y \leq \tilde y'
\Big\} \\
(\text{recursion})\ &= \Big\{
y \in \W_k(i\e)\ big|\  \mathcal S(y) = \mathcal S(z) = \mathcal S(z') \\ 
& \qquad \qquad \qquad \text{ and } \exists\, \tilde y,\tilde y' \in \mathcal I((i-1)\e,z,z') \cap \W_k(i\e)
\text{ such that } \tilde y \leq y \leq \tilde y'
\Big\} \\
&= \mathcal I(i\e,z,z').
\end{split}
\] 
Now assume that $y,y' \in \mathcal I(i\e,w,w') = \mathcal I(i\e,z,z')$. Then it holds
\begin{equation*}
\begin{split}
y \sim y' \text{ w.r.t. $\mathcal P(i\e,w,w')$} \ &\Longleftrightarrow 
\left\{ \begin{array}{l}
\!\left[ \begin{array}{l}
y,y' \text{ belong to the same equivalence} \\
\text{class $\mathcal J \in\mathcal P((i-1)\e,w,w')$ at time $(i-1)\e$} \\
\text{and the Riemann problem $\mathcal J \cap \W(i\e)$} \\
\text{with flux  $\feff_k(i\e,\cdot)$ does not divide them}
\end{array} \right] \\
\text{or } \\
\left[ \ \tcr(y) = \tcr(y') \text{ and } y = y' \ \right]
\end{array} \right. \\
\Big( \text{by }& \mathcal P((i-1)\e,w,w') = \mathcal P((i-1)\e,z,z') \Big) \\
&\Longleftrightarrow
\left\{ \begin{array}{l}
\! \left[ \begin{array}{l}
y,y' \text{ belong to the same equivalence} \\
\text{class $\mathcal J \in\mathcal P((i-1)\e,z,z')$ at time $(i-1)\e$} \\
\text{and the Riemann problem $\mathcal J \cap \W(i\e)$} \\
\text{with flux function $\feff_k(i\e,\cdot)$ does not divide them}
\end{array} \right] \\
\text{or } \\
\left[ \ \tcr(y) = \tcr(y') \text{ and } y = y' \ \right]
\end{array} \right. \\
&\Longleftrightarrow \ \, y \sim y' \text{ w.r.t. the partition $\mathcal P(i\e, z,z')$}.
\end{split}
\end{equation*}
Hence $\mathcal P(i\e,w,w')  = \mathcal P(i\e,z,z')$.

Let us now prove the second point, assuming w.l.o.g. that $\tcr(z) > \tsp(\bar t,w,w')$. Assume by contradiction that $z,z'$ have already interacted at time $\bar t$. This means that there exists a time $\tilde t \leq \bar t$ such that $\mathtt x(\tilde t,z) = \mathtt x(\tilde t,z')$.  Clearly $\tilde t \geq \tcr(z) > \tsp(\bar t,w,w')$. Therefore, at time $\tilde t$, $w,w',z,z' \in \W_k(\tilde t)$ and thus, by the monotonicity of $\mathtt x$, it should happen $\mathtt  \mathtt x(\tilde t, z) = \mathtt x(\tilde t,w) = \mathtt x(\tilde t,w')  = \mathtt x(\tilde t,z')$, a contradiction, since $\bar t \geq \tilde t > \tsp(\bar t,w,w')$.

Let us now prove the third part of the proposition. We consider only the case $\tcr(w) \leq \tcr(w')$, the case $\tcr(w) > \tcr(w')$ being completely similar. By contradiction, assume that $z,z'$ have already interacted at time $\bar t$. This means that there is a time $\tilde t \leq \bar t$ such that $\mathtt x(\tilde t, z) = \mathtt x(\tilde t,z')$. Since $z' \in \E(\bar t,w')$, it must hold $\tilde t \geq \tcr(z') = \tcr(w') \geq \tcr(w)$. Hence $w,w',z,z' \in \W_k(\tilde t)$ and by the monotonicity of $\mathtt x$, we have $\mathtt x(z) = \mathtt x(w) = \mathtt x(w') = \mathtt x(z')$, a contradiction since $w,w'$ have never interacted at time $\bar t \geq \tilde t$.
\end{proof}

\section{\texorpdfstring{The functional $\mathfrak Q_k$}{The new interaction functional}}
\label{S_fQ}

Now we have all the tools we need to define the functional $\fQ_k$ for $k =1,\dots,n$ and to prove that it satisfies the inequality \eqref{E_bound_on_fQ}, thus obtaining the \emph{global} part of the proof of Theorem \ref{T_main}.

In Section \ref{Ss_def_fQ} we give the definition of $\fQ_k$, using the intervals $\mathcal I(\bar t,w,w')$ and their partitions $\mathcal P(\bar t,w,w')$. In Section \ref{Ss_statement_main_thm} we state the main theorem of this last part of the paper, i.e. inequality \eqref{E_bound_on_fQ} and we give a sketch of its proof, which will be written down in details in Sections \ref{Ss_one_created}, \ref{Ss_both_conserved}, \ref{Ss_interacting}.

\subsection{\texorpdfstring{Definition of the functional $\mathfrak Q_k$}{Definition of the functional}}
\label{Ss_def_fQ}

We define now for each family $k =1,\dots, n$, the functional $\fQ_k = \fQ_k(t)$, which bounds the change in speed of the waves in the approximate solution $u_\e$, or more precisely, which satisfies \eqref{E_bound_on_fQ}. We first define the weight $\mathfrak q_k(t,\tau,\tau')$ of a pair of waves
\begin{equation*}
w = \Phi_k(t)^{-1}(\tau), \ w' = \Phi_k(t)^{-1}(\tau')) \in \W_k^\pm, \qquad w < w',
\end{equation*}
at time $t$, and then we define the functional $\fQ_k(t)$ as the sum (integral) over all pairs $(\tau,\tau')$ of the weight $\mathfrak q_k(t,\tau,\tau')$. 

Let us start with the definition of the weights $\mathfrak q_k(t,\tau,\tau')$ at times $t = i\e$, $i \in \N$. Fix $i\e \in \N\e$ and let $\tau, \tau' \in (0, L_k^+(i\e)]$ (resp. $\tau, \tau' \in [L_k^-(i\e), 0)$); $\tau, \tau'$ correspond to the two waves $w:=\Phi_k(i\e)^{-1}(\tau)$, $w':=\Phi_k(i\e)^{-1}(\tau')$ in $\W_k(i\e)$ having the same sign. We define the \emph{weight associated to the pair $(\tau,\tau')$ at time $i\e$} as follows:
\begin{itemize}
\item if $w,w'$ are not divided in the real solution at time $i\e$ or if they are divided but they will never interact after time $i\e$, set 
\begin{equation*}
\mathfrak q_k(i\e,\tau,\tau') := 0;
\end{equation*}
\item if $w,w'$ are divided in the real solution at time $i\e$ and they will interact after time $i\e$, set
\begin{equation*}
\mathfrak q_k(i\e,\tau,\tau') := \frac{\pi_k(i\e,\tau,\tau')}{d_k(i\e,\tau,\tau')},
\end{equation*}
\end{itemize}
where $\pi_{k}(i\e,\tau, \tau')$, $d_{k}(i\e, \tau, \tau')$ are defined as follows. Since $w,w'$ will interact after time $i\e$, then $i\e < \tint(i\e,w,w')$.
Let 
\begin{equation}
\label{E_element_of_part}
\begin{split}
&\mathcal J, \mathcal J' \in \mathcal P(i\e,w,w'), \text{ such that } w \in \mathcal J, w' \in \mathcal J', \\
&\mathcal K, \mathcal K' \in \mathcal P(\tint(i\e,w,w') - \e, w,w'), \text{ such that } w \in \mathcal K, w' \in \mathcal K'.
\end{split}
\end{equation} 
be the element of the partition containing $w,w'$ at time $i\e$ and at time $\tint(i\e,w,w') - \e$ respectively, and set 
\begin{equation}
\label{E_gg'}
\mathcal G := \mathcal K \cup \big\{z \in \mathcal J \ \big| \ z > \mathcal K \big\}, \qquad
\mathcal G' := \mathcal K' \cup  \big\{z \in \mathcal J' \ \big| \ z < \mathcal K \big\},
\end{equation}
and
\begin{equation*}
\mathcal B := \mathcal K \cup \Big\{ z \in \W_k(i\e) \ \big| \ \mathcal S(z) = \mathcal S(w) = \mathcal S(w') \text{ and } \mathcal K < z < \mathcal K' \Big\} \cup \mathcal K'.
\end{equation*}
By Lemma \ref{L_partition_at_different_times}, $\mathcal G, \mathcal G'$ are i.o.w.s at time $i\e$. We can now define 
%
\begin{equation}
\label{E_def_pi}
\pi_k(i\e,\tau,\tau') := 
\Big[\sigmarh(\feff_k(i\e), \mathcal G) - \sigmarh(\feff_k(i\e), \mathcal G')\Big]^+
\end{equation}
and
\begin{equation}
\label{E_def_d}
d_k(i\e,\tau,\tau') := \mathcal L^1 \big(\Phi_k(i\e)(\mathcal B) \big).
\end{equation}
%

\noindent As usual, set
\begin{equation*}
\mathfrak q_k(t,\tau, \tau') = \mathfrak q_k(i\e,\tau, \tau') \qquad \text{for } t \in [i\e, (i+1)\e).
\end{equation*}

\begin{remark}
It is easy to see that $\mathfrak q_k(i\e,\tau,\tau')$ is uniformly bounded: in fact,
\begin{equation*}
0 \leq \mathfrak q_k(i\e, \tau, \tau') = \frac{\pi_k(i\e,\tau, \tau')}{d_k(i\e, \tau, \tau')} \leq \|D^2 \feff_k(\bar t)\|_\infty \leq \const.
\end{equation*}
\end{remark}

We can finally define the functional $\fQ_k(t)$ as 
\begin{equation*}
\fQ_k(t) := \fQ^+_k(t) + \fQ^-_k(t),
\end{equation*}
where 
\begin{equation*}
\fQ^+_k(t) := \int_0^{L_k^+(t)} d\tau \int_\tau^{L^+_k(t)} d \tau' \mathfrak q_k\big(t, \tau,\tau'\big)
\end{equation*}
and
\begin{equation*}
\fQ^-_k(t) := \int_{L_k^-(t)}^0 d\tau \int_\tau^0 d \tau' \mathfrak q_k\big(t, \tau',\tau\big).
\end{equation*}

\begin{remark}
Clearly $\fQ_k(t)$ is constant on the time intervals $[i\e, (i+1)\e)$ and it changes its value only at times $i\e$, $i \in \N$. 
\end{remark}

\subsection{Statement of the main theorem and sketch of the proof}
\label{Ss_statement_main_thm}

We now state the main theorem of this last part of the paper and give a sketch of its proof: with this theorem, the proof of the Theorem \ref{T_main} is completed.


\begin{theorem}
\label{T_variation_fQ}
For any $i \in \N$, $i \geq 1$, it holds
\begin{equation}
\label{E_main}
\begin{split}
\fQ_k(i\e) - \fQ_k((i-1)\e)
 \leq
- \sum_{m \in \Z} \Aquadr_k(i\e,m\e) 
 + \const \TV(u(0); \R) 
\sum_{m \in \Z} \mathtt A(i\e,m\e). 
\end{split}
\end{equation}
\end{theorem}

\begin{proof}[Sketch of the proof]
First of all observe that it is sufficient to prove inequality \eqref{E_main} separately for $\fQ^+_k$ and $\fQ^-_k$. In particular, we will prove only that
\begin{equation*}
\fQ_k^+(i\e) - \fQ_k^+((i-1)\e)
 \leq
- \sum_{\substack{m \in \Z \\ \mathcal S(\W_k(i\e,m\e)) = 1}} \Aquadr_k(i\e,m\e)
 + \const \TV(u(0); \R) 
\sum_{m \in \Z} \mathtt A(i\e,m\e),
\end{equation*}
since the proof of the same inequality for $\fQ^-_k$ is completely similar. 

For any $m \in \Z$, set (see Figure \ref{Fi_quadr_inter})
\begin{equation}
\label{E_trans_surv}
\begin{split}
J_m^L &:= \Phi_k ((i-1)\e )\Big( \W_k^{(1)} \big((i-1)\e, (m-1)\e\big) \cap \W_k^+((i-1)\e) \Big), \\
J_m^R &:= \Phi_k ((i-1)\e) \Big( \W_k^{(0)} \big((i-1)\e, m\e\big) \cap \W_k^+((i-1)\e) \Big), \\
J_m &:= J_m^L \cup J_m^R, \\
K_m &:= \Phi_k(i\e)\Big(\W_k(i\e,m\e) \cap \W_k^+(i\e)  \Big), \\
S_m &:= \Phi_k ((i-1)\e)\Big( \big( \Sigma_k^{(1)}(i\e,m\e) \cup \Sigma_k^{(0)}(i\e,m\e)\big) \cap \W_k^+((i-1)\e) \Big), \\
T_m &:= \Phi_k (i\e) \Big(\big(\Sigma_k^{(1)}(i\e,m\e) \cup \Sigma_k^{(0)}(i\e,m\e)\big) \cap \W_k^+(i\e) \Big).
\end{split}
\end{equation}
Observe that if $\tau, \tau' \in J_m^L$ (or $\tau, \tau' \in J_m^R)$, then $\Phi_k^{-1}((i-1)\e)(\tau)$, $\Phi_k^{-1}((i-1)\e)(\tau')$ are not divided in the real solution at time $(i-1)\e$ and thus $\mathfrak q_k((i-1)\e, \tau, \tau') = 0$. 

\begin{figure}
\resizebox{14cm}{8cm}{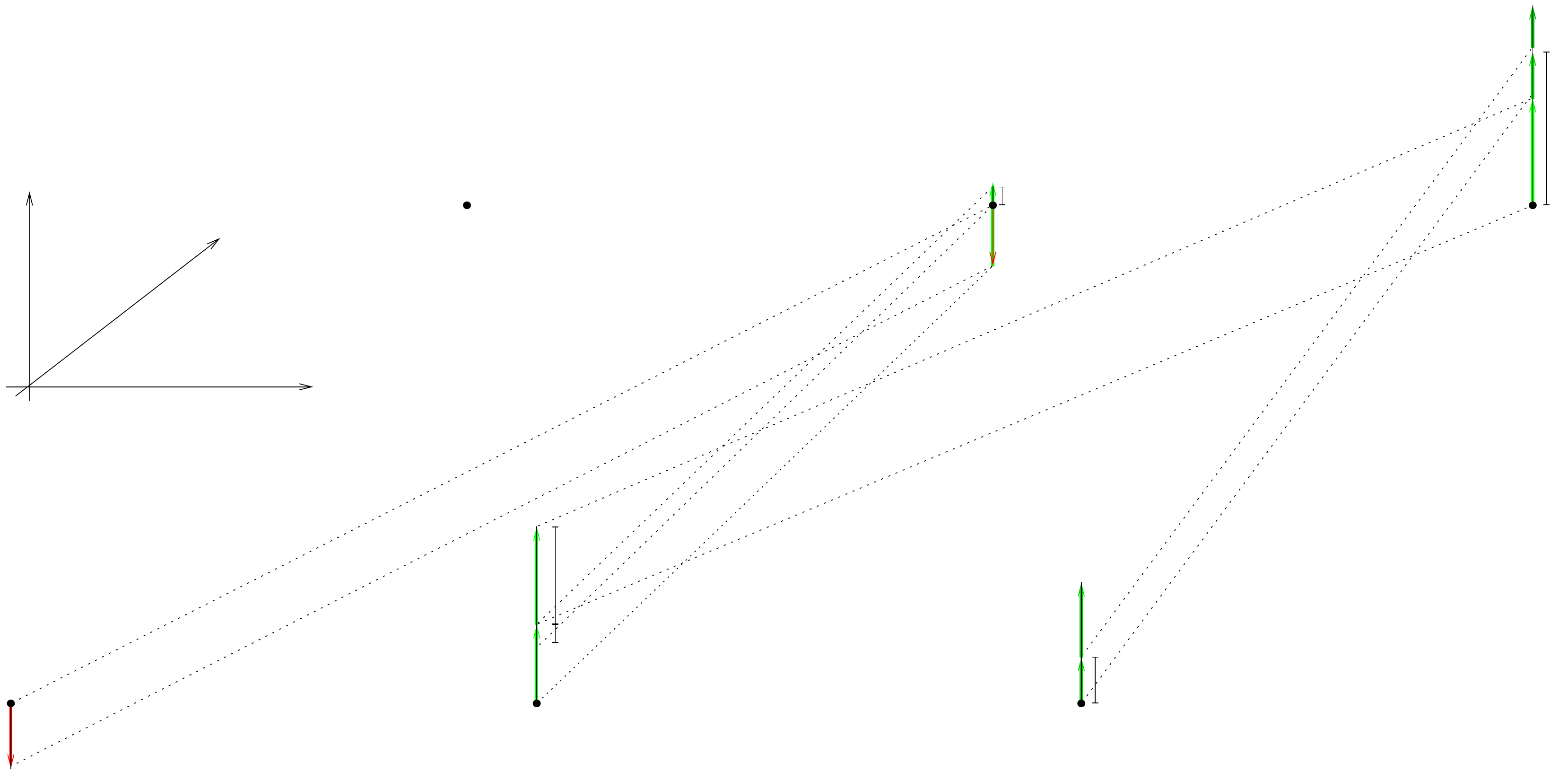}
\caption{The quantities defined in \eqref{E_trans_surv}: $J_m^L$ (resp. $J_m^R$) are the waves at $((i-1)\e, (m-1)\e)$ (resp. $((i-1)\e, m\e))$ which travel towards $(i\e,m\e)$; $S_m$ is made by all the waves in $J_m := J_m^L \cup J_m^R$ which are not canceled at time $i\e$, while the map of $S_m$ at $(i\e,m\e)$ gives the transmitted waves $T_m$; finally $K_m$ is the set of all waves at $(i\e,m\e)$ and thus the created waves at $(i\e,m\e)$ are those in $K_m \setminus T_m$.}
%
%
%
\label{Fi_quadr_inter}
\end{figure}

\noindent Similarly, if $\tau, \tau' \in K_m$, $\tau < \tau'$, setting 
$w:=\Phi_k^{-1}(i\e)(\tau)$, $w':=\Phi_k^{-1}(i\e)(\tau')$ 
then either $w,w'$ are not divided at time $i\e$, and thus $\mathfrak q_k(i\e, \tau, \tau') = 0$, or they are divided at time $i\e$, i.e. they have different positions at times $t \in (i\e,(i+1)\e)$; in this second case, with the same notations as in \eqref{E_element_of_part}-\eqref{E_gg'}, we can use  the monotonicity properties of the derivative of the convex envelope and the fact that the element of the partition $\mathcal P(i\e,w,w')$ are entropic w.r.t. the function $\feff_k(i\e)$ to obtain
\begin{equation*}
0 \geq \sigmarh(\feff_k(i\e), \mathcal J) - \sigmarh(\feff_k(i\e), \mathcal J') \geq \sigmarh(\feff_k(i\e), \mathcal G) - \sigmarh(\feff_k(i\e), \mathcal G'),
\end{equation*}
and thus $\pi_k(i\e,\tau, \tau') = 0 = \mathfrak q_k(i\e,\tau,\tau')$.

We can thus perform the following computation:
\begin{equation*}
\begin{split}
\fQ_k^+(i\e) - \fQ_k^+((i-1)\e) &\leq \sum_{m < m'} \Bigg[ \iint_{T_m \times T_{m'}} \mathfrak q_k(i\e, \tau, \tau') d\tau d\tau' 
+ \iint_{(K_m \times K_{m'}) \setminus (T_m \times T_{m'})} \mathfrak q_k(i\e, \tau, \tau') d\tau d\tau' \\
&~ \quad \qquad \qquad \qquad - \iint_{S_m \times S_{m'}} \mathfrak q_k((i-1)\e, \tau, \tau') d\tau d\tau' \Bigg] \\
&~ \quad - \sum_{m \in \Z} \iint_{J_m^L \times J_m^R} \mathfrak q_k((i-1)\e, \tau, \tau') d\tau d\tau'.
\end{split}
\end{equation*}
We will now separately study:
\begin{enumerate}
\item in Section \ref{Ss_one_created}, the integral over \emph{pairs of waves such that at least one of them is created at time $i\e$}:
\begin{equation}
\label{E_one_created}
\sum_{m < m'} \iint_{(K_m \times K_{m'}) \setminus (T_m \times T_{m'})} \mathfrak q_k(i\e,\tau,\tau') d\tau d\tau' \leq \const \TV(u(0)) \sum_{m \in \Z}  
\mathtt A(i\e,m\e).
\end{equation}
\item in Section \ref{Ss_both_conserved}, the variation of the integral over \emph{pairs of waves which exist both at time $(i-1)\e$ and at time $i\e$}:
\begin{equation}
\label{E_both_conserved}
\sum_{m < m'} \Bigg[ \iint_{T_m \times T_{m'}} \mathfrak q_k(i\e) d\tau d\tau' - \iint_{S_m \times S_{m'}} \mathfrak q_k((i-1)\e) d\tau d\tau' \Bigg] \leq  \const \TV(u(0)) \sum_{r \in \Z} \mathtt A(i\e,r\e). 
\end{equation}
\item in Section \ref{Ss_interacting}, the (negative) term, related to \emph{pairs of waves which are divided at time $(i-1)\e$ and are interacting at time $i\e$}:
\begin{equation}
\label{E_interacting}
- \sum_{m \in \Z} \iint_{J_m^L \times J_m^R} \mathfrak q_k((i-1)\e) d\tau d\tau' \leq 
- \sum_{\substack{m \in \Z \\ \mathcal S(\W_k(i\e,m\e)) = 1}} \Aquadr_k(i\e,m\e)
+
\const \TV(u(0)) \sum_{m \in \Z} \mathtt A(i\e,m\e).
\end{equation}
\end{enumerate}

It is easy to see that inequality \eqref{E_main} in the statement of Theorem \ref{T_variation_fQ} follows from \eqref{E_one_created}, \eqref{E_both_conserved}, \eqref{E_interacting}.
\end{proof}

\subsection{Analysis of pairs with at least one created wave}
\label{Ss_one_created}

The integral over pair of waves such that at least one of them is created at time $i\e$ is estimated in the following proposition.

\begin{proposition}
\label{P_prop_crea_fin}
It holds
\begin{equation*}
\sum_{m<m'} \iint_{(K_m \times K_{m'}) \setminus (T_m \times T_{m'})} \mathfrak q_k(i\e,\tau,\tau') d\tau d\tau'  \leq \const \TV(u(0)) \sum_{m \in \Z}  
\mathtt A(i\e,m\e).
\end{equation*}
\end{proposition}

\begin{proof}
In fact,
\begin{equation*}
\begin{split}
\mathcal L^2\big((K_m \times K_{m'}) \setminus (T_m \times T_{m'}) \big) &\leq 
\mathcal L^2\big((K_m \setminus T_m) \times K_{m'}\big)
+ \mathcal L^2\big(K_m  \times (K_{m'} \setminus T_{m'}) \big) \\
&\leq \mathcal L^1(K_{m'}) \mathcal L^1(K_m \setminus T_m) + \mathcal L^1 (K_m) \mathcal L^1 (K_{m'} \setminus T_{m'}).
\end{split}
\end{equation*}
Hence
\begin{equation*}
\begin{split}
\sum_{m<m'} \iint_{(K_m \times K_{m'}) \setminus (T_m \times T_{m'})} \mathfrak q_k&(i\e,\tau,\tau') d\tau d\tau' \\
&\leq \const \sum_{m<m'} \mathcal L^2\big((K_m \times K_{m'}) \setminus (T_m \times T_{m'}) \big) \\
&\leq \const \sum_{m<m'} \mathcal L^1(K_{m'}) \mathcal L^1(K_m \setminus T_m) + \mathcal L^1 (K_m) \mathcal L^1 (K_{m'} \setminus T_{m'}) \\
&\leq \const \sum_{m' \in \Z} \mathcal L^1(K_{m'}) \sum_{m \in \Z} \mathcal L^1(K_m \setminus T_m) \\
&\leq \const L_k^+ (i\e) \sum_{m \in \Z} \Acr_k(i\e,m\e) \\
\text{(by \eqref{E_bd_on_TV} and Corollary \ref{C_aocr})} &\leq 
\const \TV(u(0)) \sum_{m \in \Z} \mathtt A(i\e,m\e). \qedhere
\end{split}
\end{equation*}
\end{proof}

\subsection{\texorpdfstring{Analysis of pairs of waves which exist both at time $(i-1)\e$ and at time $i\e$}{Analysis of pairs of waves which exist at both times}}
\label{Ss_both_conserved}

The aim of this section is to estimate the variation of the integral over pair of waves which exist both at time $(i-1)\e$ and at time $i\e$. More precisely we prove the following theorem.

\begin{theorem}
\label{T_decreasing}
It holds
\begin{equation}
\label{E_pair_in_bij}
\begin{split}
\sum_{m < m'} \Bigg[ \iint_{T_m \times T_{m'}} \mathfrak q_k(i\e,\tau,\tau') d\tau d\tau' - \iint_{S_m \times S_{m'}} \mathfrak q_k(&(i-1)\e,\tau,\tau') d\tau d\tau' \Bigg] \\ 
&\leq \const \TV(u(0)) \sum_{r \in \Z} \mathtt A(i\e,r\e). 
\end{split}
\end{equation}
\end{theorem}

We first need a preliminary result. As a starting point, observe that for each $m \in \Z$, by Points \ref{Pt_iow_cons} and \ref{Pt_affine} of Theorem \ref{T_lagrangian}, the map
\begin{equation*}
\Theta: S_m \to T_m, \qquad \Theta := \Phi_k(i\e) \circ \Phi_k((i-1)\e)^{-1},
\end{equation*}
is an affine function with slope $1$. We now prove the following lemma, which estimates the change of the numerator $\pi_{k}$ and the denominator $d_{k}$ in the definition of $\mathfrak q_{k}$, formulas \eqref{E_def_pi} and \eqref{E_def_d}.

\begin{lemma}
\label{L_delta_dpi}
For any $m < m'$ and for  any $\tau \in S_m$, $\tau' \in S_{m'}$, setting 
\begin{subequations}
\label{E_delta_d_delta_pi}
\begin{equation}
\label{E_delta_d}
\Delta d_{k}(\tau,\tau') := d_{k}\big(i\e, \Theta(\tau), \Theta(\tau')\big) - d_{k}\big((i-1)\e, \tau, \tau'\big),
\end{equation}
\begin{equation}
\label{E_delta_pi}
\Delta \pi_{k}(\tau, \tau') := \pi_{k}\big(i\e, \Theta(\tau), \Theta(\tau')\big) - \pi_{k}\big((i-1)\e, \tau, \tau'\big).
\end{equation}
\end{subequations}
the following inequalities hold:
\begin{equation*}
\begin{split}
\big|\Delta d_{k}(\tau, \tau')\big| \leq \const \sum_{r =  m}^{ m'} \mathtt A(i\e,r\e), \qquad
\Delta \pi_{k}(\tau, \tau') \leq \const \sum_{r = m}^{m'} \mathtt A(i\e,r\e).
\end{split}
\end{equation*}
\end{lemma}

\begin{proof}
Fix $w:= \Phi_k((i-1)\e)^{-1}(\tau)$, $w':= \Phi_k((i-1)\e)^{-1}(\tau')$ and let 
\begin{align*}
\mathcal J, \mathcal J' &\in \mathcal P((i-1)\e,w,w'), && w \in \mathcal J, w' \in \mathcal J', \\
\tilde{\mathcal J}, \tilde{\mathcal J}' &\in \mathcal P(i\e,w,w'), && w \in \tilde{\mathcal J}, w' \in \tilde{\mathcal J}', \\ 
\mathcal K, \mathcal K' &\in \mathcal P(\tint(i\e,w,w') - \e, w,w'), && w \in \mathcal K, w' \in \mathcal K'.
\end{align*} 

\noindent Set also 
\begin{equation*}
\mathcal A := \mathcal K \cup \big\{z \in \W_k^+((i-1)\e) \ \big| \ \mathcal K < z < \mathcal K' \big\} \cup \mathcal K', \qquad
\mathcal B := \mathcal K \cup \big\{z \in \W_k^+(i\e) \ \big| \ \mathcal K < z < \mathcal K' \big\} \cup \mathcal K'.
\end{equation*}
It is easy to see that 
\begin{equation}
\label{E_a_b_tra_m_m'}
\mathcal A \subseteq \bigcup_{r = m}^{m'} \Big\{w \in \W_k((i-1)\e) \ \big| \ \lim_{t \rightarrow i\e} \mathtt x(t,w) = r\e \Big\}, \qquad
\mathcal B \subseteq \bigcup_{r = m}^{m'} \W_k(i\e,m\e).
\end{equation}
Moreover, it holds
\begin{equation*}
\mathcal A  = \big( \mathcal A \setminus \W_k(i\e) \big) \cup \big( \mathcal A \cap \W_k(i\e) \big)
 = \big( \mathcal A \setminus \W_k(i\e) \big) \cup  \bigcup_{r = m}^{m'} \big( \mathcal A \cap \W_k(i\e,r\e) \big)
\end{equation*}
and 
\begin{equation*}
\begin{split}
\mathcal B &= \big( \mathcal B \setminus \W_k((i-1)\e) \big) \cup \big( \mathcal B \cap \W_k((i-1)\e) \big) \\
&= \big( \mathcal B \setminus \W_k((i-1)\e) \big) \cup  \bigcup_{r = m}^{m'} \big( \mathcal B \cap \W_k((i-1)\e) \cap \W_k(i\e,r\e) \big). 
\end{split}
\end{equation*}
Notice that for any $r = m, \dots, m'$
\begin{equation}
\label{E_dist_waves_cons}
\mathcal A \cap \W_k(i\e,r\e) = \mathcal B \cap \W_k((i-1)\e) \cap \W_k(i\e,r\e),
\end{equation}
and thanks to this equality and Corollary \ref{C_iow_same_pos}, the set in \eqref{E_dist_waves_cons} is an interval of waves both at time $(i-1)\e$ and at time $i\e$. Hence, by Point \eqref{Pt_misura_iow} of Corollary \ref{C_iow_same_pos},
\begin{equation}
\label{E_theta}
\begin{array}{ccccc}
\Theta &:& \Phi_k \big((i-1)\e\big) \Big( \mathcal A \cap \W_k(i\e,r\e) \Big) &\to& \Phi_k \big(i\e\big) \Big( \mathcal B \cap \W_k((i-1)\e) \cap \W_k(i\e,r\e) \Big), \\ [.5em]
&& \tau &\mapsto& \Theta := \Phi_k(i\e) \circ \Phi_k((i-1)\e)^{-1}(\tau)
\end{array}
\end{equation}
is an affine function with slope equal to $1$ and thus
\begin{equation}
\label{E_dist_meas_waves_cons}
\mathcal L^1 \Big( \Phi_k ((i-1)\e) \big( \mathcal A \cap \W_k(i\e,r\e) \big) \Big) = \mathcal L^1 \Big( \Phi_k (i\e) \big(\mathcal B \cap \W_k((i-1)\e) \cap \W_k(i\e,r\e) \big)  \Big).
\end{equation}

We now prove separately the two inequalities of the statement.

\smallskip
\noindent \textit{Proof of \eqref{E_delta_d}.}
We have
\begin{equation*}
\begin{split}
\big| \Delta d(\tau, \tau') \big| 
&=  \big| d\big(i\e, \Theta(\tau), \Theta(\tau')\big) - d\big((i-1)\e, \tau, \tau'\big) \big|\\
&= \Big| \mathcal L^1 \Big( \Phi_k(i\e)\big(\mathcal B\big) \Big) - \mathcal L^1 \big(\Phi_k((i-1)\e)\big(\mathcal A\big) \big)
\Big| \\
&= \bigg| \mathcal L^1 \big( \Phi_k (i\e) \big( \mathcal B \setminus \W_k((i-1)\e) \big) \big) + \sum_{r = m}^{m'} 
\mathcal L^1 \Big( \Phi_k(i\e) \big( \mathcal B \cap \W_k((i-1)\e) \cap \W_k(i\e, m\e) \big) \Big) \\
&\qquad \qquad 
- \mathcal L^1 \Big( \Phi_k((i-1)\e) \big(\mathcal A \setminus \W_k(i\e) \big) \Big)
- \sum_{r = m}^{m'} \mathcal L^1 \Big( \Phi_k((i-1)\e) \big(\mathcal A \cap \W_k(i\e,r\e) \big)
 \Big)
\bigg| \\
\text{(by \eqref{E_dist_meas_waves_cons})} 
&\leq \Big| \mathcal L^1 \Big( \Phi_k (i\e) \big( \mathcal B \setminus \W_k((i-1)\e) \big) \Big) \Big| +
\Big| \mathcal L^1 \Big( \Phi_k((i-1)\e) \big(\mathcal A \setminus \W_k(i\e) \big) \Big) \Big| \\
\text{(by \eqref{E_a_b_tra_m_m'})} &\leq \sum_{r = m}^{m'} \Acr_k(i\e,r\e) + \Acanc_k(i\e,r\e) \\
\text{(by Cor. \ref{C_aocr})} &\leq \sum_{r = m}^{m'} \mathtt A(i\e,r\e). 
\end{split}
\end{equation*}

\smallskip
\noindent \textit{Proof of \eqref{E_delta_pi}.} The proof of the second inequality in \eqref{E_delta_d_delta_pi} is more involved. Define
\begin{align*}
\mathcal F &:= \mathcal K \cup \Big(\mathcal J \cap \big\{z \in \W_k((i-1)\e) \ \big| \ z > \mathcal K \big\} \Big), &
\mathcal F' &:= \mathcal K' \cup \Big(\mathcal J' \cap \big\{z \in \W_k((i-1)\e) \ \big| \ z < \mathcal K' \big\} \Big), \\
\mathcal G &:= \mathcal K \cup \Big(\tilde{\mathcal J} \cap \big\{z \in \W_k(i\e) \ \big| \ z > \mathcal K \big\} \Big), &
\mathcal G' &:= \mathcal K' \cup \Big(\tilde{\mathcal J}' \cap \big\{z \in \W_k(i\e) \ \big| \ z < \mathcal K' \big\} \Big);
\end{align*}
$\mathcal F, \mathcal F'$ are i.o.w.s at time $(i-1)\e$, while $\mathcal G, \mathcal G'$ are i.o.w.s at time $i\e$. Hence the sets
\begin{align*}
F &:= \Phi_k((i-1)\e) (\mathcal F), &
F'&:= \Phi_k((i-1)\e)(\mathcal F'), \\
G&:= \Phi_k(i\e)(\mathcal G), &
G'&:= \Phi_k(i\e)(\mathcal G'),
\end{align*}
are intervals in $\R$. 
Moreover, since $\tau \mapsto \feff_k(t)(\tau)$ is defined up to affine function, we can assume that
\begin{equation}
\label{E_feff_der_sec}
\frac{d \feff_k(i\e)}{d \tau}\big(\inf \Phi_k(i\e)(\mathcal K)\big) 
= \frac{d \feff_k((i-1)\e)}{d \tau}\big(\inf \Phi_k((i-1)\e)(\mathcal K)\big) = 0.
\end{equation}

We divide now the proof of the second inequality in several steps.

\begin{figure}
\resizebox{14cm}{8cm}{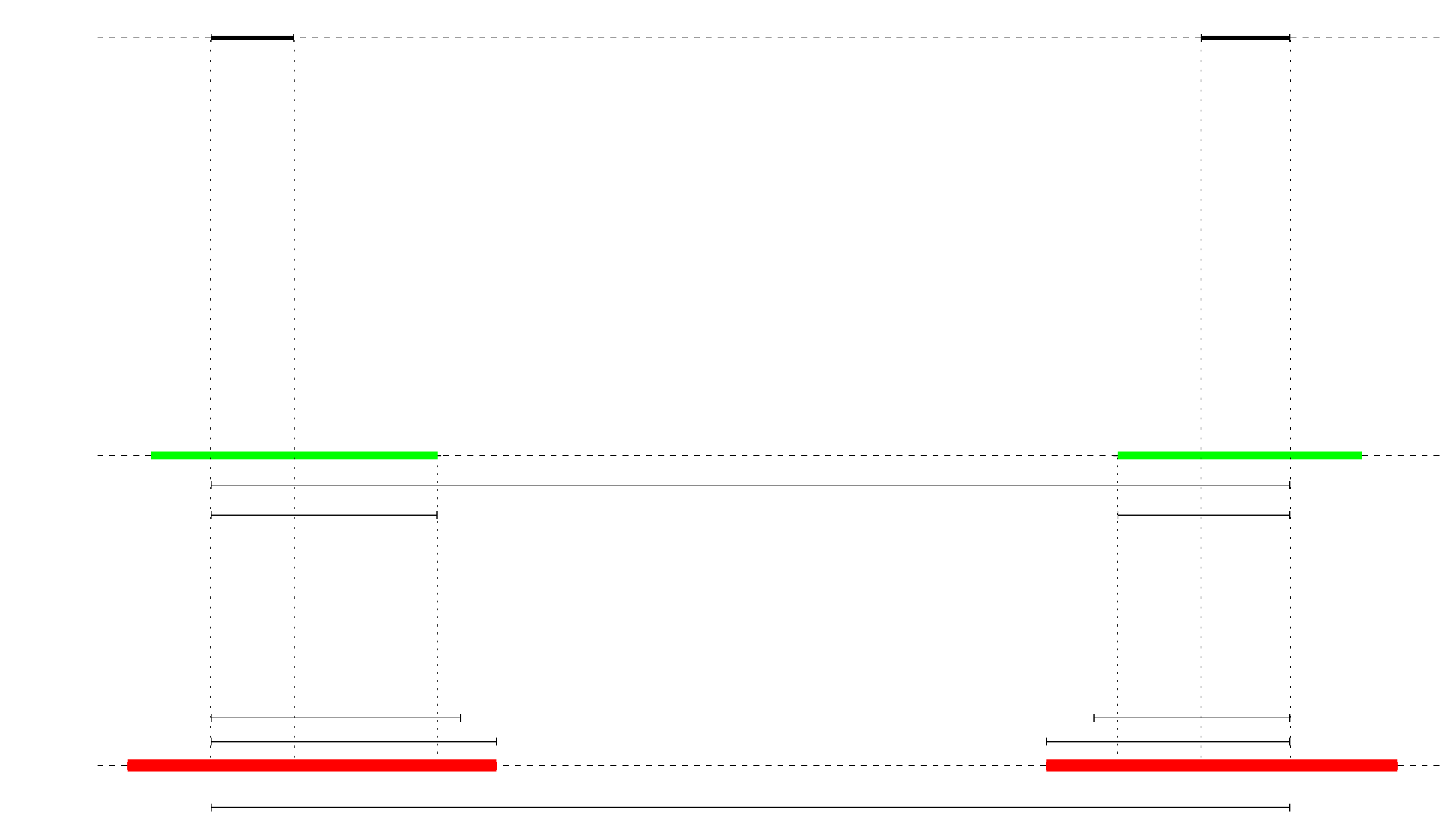}
\caption{The various set used in the proof of \eqref{E_delta_pi}: in Step 2 pass from the waves in $\mathcal F, \mathcal F'$ to the waves $\mathcal H, \mathcal H'$ which survives at $t = i\e$; in Step 3 change the flux $\feff_k((i-1)\e)$ to $\feff_k(i\e)$ for the intervals $\mathcal H,\mathcal H'$; in Step 4 observe that $\mathcal G,\mathcal G'$ are shorter that $\mathcal H,\mathcal H'$ because of a splitting has occurred.}
\label{Fi_estim_sped}
\end{figure}

\noindent \textit{Step 1.}
Define
\begin{equation*}
\mathcal H := 
\mathcal K \cup \big\{z 
\in \mathcal J \cap \W_k(i\e) \ \big| \ z > \mathcal K \big\}, \qquad
\mathcal H' := 
\mathcal K' \cup \big\{z 
\in \mathcal J' \cap \W_k(i\e) \ \big| \ z < \mathcal K' \big\}.
\end{equation*}
We now show that the sets $\mathcal H, \mathcal H'$ are i.o.w.s both at time $i\e$ and at time $(i-1)\e$ and 
\begin{equation*}
\mathcal H \subseteq \mathcal J \cap \W_k(i\e), \qquad \mathcal H' \subseteq \mathcal J' \cap \W_k(i\e).
\end{equation*}
Moreover also the sets 
\begin{align*}
H_{i-1} &:= \Phi_k((i-1)\e)(\mathcal H), & H_{i} &:= \Phi_k(i\e)(\mathcal H), \\
H_{i-1}' &:= \Phi_k((i-1)\e)(\mathcal H'), & H_{i}' &:= \Phi_k(i\e)(\mathcal H'),
\end{align*}
are intervals in $\R$. 

\smallskip

\noindent \textit{Proof of Step 1.}
We prove only the statements related to $\mathcal H$, those related to $\mathcal H'$ being completely analogous. Clearly $\mathcal H \subseteq \mathcal J \cap \W_k(i\e)$. Moreover the set 
\begin{equation*}
\mathcal M:= \big\{z \in \W_k(i\e) \ \big| \ z \in \mathcal K \text{ or } z > \mathcal K \big\}
\end{equation*}
is clearly an i.o.w. at time $i\e$. Since we can write $\mathcal H$ as intersection of two i.o.w.s at time $i\e$ as
\begin{equation*}
\mathcal H = \mathcal M \cap \big( \mathcal J \cap \W_k(i\e) \big),
\end{equation*}
it follows that also $\mathcal H$ is an i.o.w. at time $i\e$.
Moreover, since $\mathcal H = \mathcal H \cap \W_k((i-1)\e)$, by Proposition \ref{P_divise_partizione_implica_divise_realta} and Corollary  \ref{C_iow_same_pos}, Point \eqref{Pt_sono}, $\mathcal H$ is an i.o.w. also at time $(i-1)\e$. As an immediate consequence $H_{i-1}$ and $H_i$ are intervals in $\R$. 

\bigskip

\noindent \textit{Step 2.}
We have
\begin{equation*}
H_{i-1} \subseteq F, \qquad \mathcal L^1(F) - \mathcal L^1(H_{i-1}) \leq \Acanc(i\e,m\e),
\end{equation*}
and
\begin{equation*}
\Big|\sigmarh(\feff_k((i-1)\e),\mathcal H) - \sigmarh(\feff_k((i-1)\e),\mathcal F) \Big| \leq \const \Acanc_k(i\e,m\e).
\end{equation*}
Similarly, $H_{i-1}' \subseteq F'$, $\mathcal L^1(F') - \mathcal L^1(H_{i-1}') \leq \Acanc(i\e,m'\e)$ and
\begin{equation*}
\Big|\sigmarh(\feff_k((i-1)\e),\mathcal H') - \sigmarh(\feff_k((i-1)\e),\mathcal F') \Big| \leq \const \Acanc_k(i\e,m'\e).
\end{equation*}

\smallskip
\noindent \textit{Proof of Step 2.}
We prove only the first part of the statement, the second one being completely similar. Clearly $\mathcal H \subseteq \mathcal F$. Hence $H_{i-1} = \Phi_k((i-1)\e)(\mathcal H) \subseteq \Phi_k((i-1)\e)(\mathcal F) = F$. Moreover, by Proposition \ref{P_divise_partizione_implica_divise_realta}, it follows that
\begin{equation*}
\mathcal F \setminus \mathcal H \subseteq \Big\{w \in \W_k((i-1)\e) \setminus \W_k(i\e) \ \big| \ \lim_{t \rightarrow i\e} \mathtt x(t,w) = m\e \Big\}
\end{equation*}
and thus 
\begin{equation*}
\begin{split}
\mathcal L^1(F) - \mathcal L^1(H_{i-1}) &= \mathcal L^1(\Phi_k((i-1)\e)(\mathcal F)) - \mathcal L^1(\Phi_k((i-1)\e)(\mathcal H)) \\
&= \mathcal L^1(\Phi_k((i-1)\e)(\mathcal F \setminus \mathcal H)) \\
&\leq \mathcal L^1\Big(\Big\{w \in \W_k((i-1)\e) \setminus \W_k(i\e) \ \big| \ \lim_{t \rightarrow i\e} \mathtt x(t,w) = m\e \Big\}\Big) \\
&\leq \Acanc_k(i\e,m\e).
\end{split}
\end{equation*}
Moreover, by Proposition \ref{P_ri_lip},
\begin{equation*}
\Big|\sigmarh(\feff_k((i-1)\e,\mathcal H) - \sigmarh(\feff_k((i-1)\e,\mathcal F) \Big| \leq \mathcal L^1(F) - \mathcal L^1(H_{i-1})\leq \const \Acanc_k(i\e,m\e).
\end{equation*}

\bigskip

\noindent \textit{Step 3.}
It holds
\begin{equation*}
\Big| \sigmarh(\feff_k(i\e),\mathcal H) - \sigmarh \big( \feff_k((i-1)\e),\mathcal H \big) \Big| \leq \const \sum_{r=m}^{m'} \mathtt A(i\e,r\e),
\end{equation*}
\begin{equation*}
\Big| \sigmarh(\feff_k(i\e),\mathcal H') - \sigmarh \big( \feff_k((i-1)\e),\mathcal H' \big) \Big| \leq \const \sum_{r=m}^{m'} \mathtt A(i\e,r\e).
\end{equation*}

\smallskip
\noindent \textit{Proof of Step 3.}
In this step, we prove only the second inequality and assume that $\mathcal L^1(H_i) = \mathcal L^1(H_{i-1}) >0$, since the first inequality and the other cases can be treated similarly (and actually the computations are simpler). 

We have
\begin{equation*}
\begin{split}
\Big| \sigmarh(\feff_k(i\e),\mathcal H'&) - \sigmarh(\feff_k((i-1)\e),\mathcal H') \Big| \\
&= \bigg| \frac{1}{\mathcal L^1(H_i')} \int_{H_i'} \frac{d \feff_k(i\e)}{d\varsigma}(\varsigma) d\varsigma - \frac{1}{\mathcal L^1(H_{i-1}')} \int_{H_{i-1}'} \frac{d \feff_k((i-1)\e)}{d\tau}(\tau) d\tau \bigg| \\
\text{(by \eqref{E_feff_der_sec})} &= \bigg| \frac{1}{\mathcal L^1(H_i')} \int_{H_i'} \int_{\inf \Phi_k(i\e)(\mathcal K)}^{\varsigma} \frac{d^2 \feff_k(i\e)}{d\xi^2}(\xi) d\xi d\varsigma \\
& \qquad \qquad \qquad \qquad - \frac{1}{\mathcal L^1(H_{i-1}')} \int_{H_{i-1}'} \int_{\inf \Phi_k((i-1)\e)(\mathcal K)}^{\tau} \frac{d^2 \feff_k((i-1)\e)}{d\eta^2}(\eta) d \eta d\tau \bigg|,
\end{split}
\end{equation*}
and, remembering that $\mathcal L^1(H_i') = \mathcal L^1(H_{i-1}')$ and integrating by parts,
\begin{equation*}
\begin{split}
\ldots &= \frac{1}{\mathcal L^1(H_i')} \bigg| 
\int_{\inf \Phi_k(i\e)(\mathcal K)}^{\sup H_i'} \frac{d^2 \feff_k(i\e)}{d\xi^2}(\xi) \Big(\sup H_i' - \max \big\{\xi, \inf H_i' \big\} \Big) d\xi \\
& \qquad \qquad \qquad \qquad
- \int_{\inf \Phi_k((i-1)\e)(\mathcal K)}^{\sup H_{i-1}'} \frac{d^2 \feff_k((i-1)\e)}{d\eta^2}(\eta) \Big(\sup H_{i-1}' - \max \big\{\eta, \inf H_{i-1}' \big\} \Big) d\eta 
\bigg| \\
&= \frac{1}{\mathcal{L}^1(H_i')} \Bigg| 
\sum_{r = m}^{m'} \int_{[ \inf \Phi_k(i\e)(\mathcal K), \sup H_i'] \cap K_r} \frac{d^2 \feff_k(i\e)}{d\xi^2}(\xi) \Big(\sup H_i' - \max \big\{\xi, \inf H_i' \big\} \Big) d\xi \\
&\qquad \qquad
- \sum_{r = m}^{m'} \int_{[ \inf \Phi_k((i-1)\e)(\mathcal K), \sup H_{i-1}'] \cap J_r}
\frac{d^2 \feff_k((i-1)\e)}{d\eta^2}(\eta) \Big(\sup H_{i-1}' - \max \big\{\eta, \inf H_{i-1}' \big\} \Big) d\eta 
\Bigg|,
\end{split}
\end{equation*}
where $K_r,J_r$ are defined in \eqref{E_trans_surv}; using now that $\mathcal L^1 (K_r \setminus T_r) = \Acr_k(i\e,r\e)$, while $\mathcal L^1(J_{r} \setminus S_{r}) = \Acanc_k(i\e,r\e)$ we can proceed as
\begin{equation*}
\begin{split}
\ldots &= \frac{1}{\mathcal L^1(H_i')} \Bigg|
\sum_{r =m}^{m'} \const \mathcal L^1(H_i) \Big( \Acr_k(i\e,r\e) + \Acanc_k(i\e,r\e) \Big) \\
& \qquad \qquad \qquad + 
\sum_{r =m}^{m'} \int_{[ \inf \Phi_k(i\e)(\mathcal K), \sup H_i'] \cap T_r} \frac{d^2 \feff_k(i\e)}{d\xi^2}(\xi) \Big(\sup H_i' - \max \big\{\xi, \inf H_i' \big\} \Big) d\xi \\
& \qquad \qquad \qquad
- \sum_{r =m}^{m'} \int_{[ \inf \Phi_k((i-1)\e)(\mathcal K), \sup H_{i-1}'] \cap S_r}
\frac{d^2 \feff_k(i\e)}{d\eta^2}(\eta) \Big(\sup H_{i-1}' - \max \big\{\eta, \inf H_{i-1}' \big\} \Big) d\eta 
\Bigg|,
\end{split}
\end{equation*}
using that the function $\Theta: [ \inf \Phi_k((i-1)\e)(\mathcal K), \sup H_{i-1}'] \cap S_r \to [ \inf \Phi_k(i\e)(\mathcal K), \sup H_i'] \cap T_r$ introduced in \eqref{E_theta} is an affine bijection with derivative $1$,
\begin{equation*}
\begin{split}
\ldots &\leq \frac{1}{\mathcal L^1(H_i')}
\sum_{r= m}^{m'} \const \mathcal L^1(H_i) \Big( \Acr_k(i\e,m\e) + \Acanc_k(i\e,m\e) \Big) \\
& \qquad + \frac{1}{\mathcal L^1(H_i')} \Bigg| \sum_{r= m}^{m'} \int_{[ \inf \Phi_k((i-1)\e)(\mathcal K), \sup H_{i-1}'] \cap S_r} \bigg[ \frac{d^2 \feff_k(i\e)}{d\xi^2}\big(\Theta(\eta) \big) - 
\frac{d^2 \feff_k((i-1)\e)}{d\eta^2}(\eta) \bigg] \\
& \qquad \qquad \qquad \qquad \qquad \qquad \qquad \qquad \qquad \qquad \qquad \qquad \cdot \Big( \sup H_{i-1}' - \max \big\{\eta, \inf H_{i-1}' \big\} \Big) d\eta 
 \Bigg| \\
& \leq \const \sum_{r=m}^{m'} \Big( \Acr_k(i\e,r\e) + \Acanc_k(i\e,r\e) \Big) 
+ \sum_{r=m}^{m'} \int_{S_r} \bigg| \frac{d^2 \feff_k(i\e)}{d\xi^2}\big(\Theta(\eta)\big) - 
\frac{d^2 \feff_k((i-1)\e)}{d\eta^2}(\eta) \bigg| d\eta \\
& \leq 
\const \sum_{r=m}^{m'} \Bigg( \Acr_k(i\e,r\e) + \Acanc_k(i\e,r\e) 
+ \bigg\|\frac{d^2 \feff_k(i\e)}{d\xi^2} \circ \Theta - 
\frac{d^2 \feff_k((i-1)\e)}{d\eta^2} \bigg\|_{L^1(S_r)} \Bigg),
\end{split}
\end{equation*}
and finally by Theorem \ref{T_general} and Corollary \ref{C_aocr}
\begin{equation*}
\begin{split}
\Big| \sigmarh(\feff_k(i\e),\mathcal H') &- \sigmarh(\feff_k((i-1)\e),\mathcal H') \Big| \\
& \leq 
\const \sum_{r=m}^{m'} \Bigg( \Acr_k(i\e,r\e) + \Acanc_k(i\e,r\e) 
+ \bigg\|\frac{d^2 \feff_k(i\e)}{d\xi^2} \circ \Theta - 
\frac{d^2 \feff_k((i-1)\e)}{d\eta^2} \bigg\|_{L^1(S_r)} \Bigg) \\
&\leq \const \sum_{r=m}^{m'} \mathtt A(i\e,r\e).
\end{split}
\end{equation*}

%
%
%

\bigskip

\noindent \textit{Step 4.}
It holds
\begin{equation*}
\Big[\sigmarh(\feff_k(i\e),\mathcal G) - \sigmarh(\feff_k(i\e),\mathcal G') \Big]^+ - \Big[\sigmarh(\feff_k(i\e),\mathcal H)  - \sigmarh(\feff_k(i\e),\mathcal H')\Big]^+ \leq 0.
\end{equation*}

%
%

\smallskip
\noindent \textit{Proof of Step 4.} 
We want to use Proposition \ref{P_diff_ri} with
\begin{equation*}
g = \feff_k(i\e), \quad [a,b] = \Phi_k(i\e)(\mathcal J \cap \W_k(i\e)), \quad 
\bar u = \sup \Phi_k(i\e) (\tilde{\mathcal J}), \quad
u = \inf \Phi_k(i\e)(\mathcal K).
\end{equation*}
Indeed, by definition of the partition $\mathcal P(i\e,w,w')$ (Point \eqref{Point_2b_part_I} at page \pageref{Point_2b_part_I}), it holds
\begin{equation*}
\conv_{\Phi_k(i\e)(\mathcal J \cap \W_k(i\e))} \feff_k(i\e) (\sup \Phi_k(i\e)( \tilde{\mathcal J})) 
= \feff_k(i\e) (\sup \Phi_k(i\e)(\tilde{\mathcal J})),
\end{equation*}
i.e. ${\displaystyle \conv_{[a,b]} g}(\bar u) = g(\bar u)$.

We thus have
\begin{equation}
\label{E_senza_segno_sx}
\begin{split}
\sigmarh(\feff_k(i\e), \mathcal G) &= \sigmarh \Big( \feff_k(i\e), \big[ \inf \Phi_k(i\e)(\mathcal K), \sup \Phi_k(i\e)(\tilde{\mathcal J}) \big] \Big) \\
&\leq \sigmarh \Big( \feff_k(i\e), \big[ \inf \Phi_k(i\e)(\mathcal K), \sup \Phi_k(i\e)(\mathcal J \cap \W_k(i\e)) \big] \Big) \\
&= \sigmarh(\feff_k(i\e), \mathcal H).
\end{split}
\end{equation}
In a similar way one can prove that
\begin{equation}
\label{E_senza_segno_dx}
\sigmarh(\feff_k(i\e, \mathcal G') \geq \sigmarh(\feff_k(i\e,\mathcal H').
\end{equation}
Using \eqref{E_senza_segno_sx} and \eqref{E_senza_segno_dx}, one gets the conclusion.

\bigskip

\bigskip

\noindent \textit{Step 5.} We can finally conclude the proof of \eqref{E_delta_pi}, showing that
\begin{equation*}
\Delta \pi(\tau, \tau') \leq \const \sum_{r = m}^{m'} \mathtt A(i\e,m\e).
\end{equation*}

\smallskip

\noindent \textit{Proof of Step 5.}
We can perform the following computation:
\begin{equation*}
\begin{split}
\Delta &\pi(\tau, \tau') \\
&= \pi(i\e, \Theta(\tau), \Theta (\tau')) - \pi ((i-1)\e, \tau, \tau') \\
& = \Big[\sigmarh(\feff_k(i\e),\mathcal G) - \sigmarh(\feff_k(i\e),\mathcal G') \Big]^+ - \Big[\sigmarh(\feff_k((i-1)\e),\mathcal F)  - \sigmarh(\feff_k((i-1)\e),\mathcal F')\Big]^+ \\
& = \Big[\sigmarh(\feff_k(i\e),\mathcal G) - \sigmarh(\feff_k(i\e),\mathcal G') \Big]^+ - \Big[\sigmarh(\feff_k(i\e),\mathcal H)  - \sigmarh(\feff_k(i\e),\mathcal H')\Big]^+ \\
&~ \quad + 
\Big[\sigmarh(\feff_k(i\e),\mathcal H) - \sigmarh(\feff_k(i\e),\mathcal H') \Big]^+ - \Big[\sigmarh(\feff_k((i-1)\e),\mathcal H)  - \sigmarh(\feff_k((i-1)\e),\mathcal H')\Big]^+ \\
&~ \quad + 
\Big[\sigmarh(\feff_k((i-1)\e),\mathcal H) - \sigmarh(\feff_k((i-1)\e),\mathcal H') \Big]^+ - \Big[\sigmarh(\feff_k((i-1)\e),\mathcal F)  - \sigmarh(\feff_k((i-1)\e),\mathcal F')\Big]^+ \\
& \leq \Big[\sigmarh(\feff_k(i\e),\mathcal G) - \sigmarh(\feff_k(i\e),\mathcal G') \Big]^+ - \Big[\sigmarh(\feff_k(i\e),\mathcal H)  - \sigmarh(\feff_k(i\e),\mathcal H')\Big]^+ \\
&~ \quad +
\Big| \sigmarh(\feff_k(i\e),\mathcal H) - \sigmarh(\feff_k((i-1)\e),\mathcal H) \Big| + 
\Big| \sigmarh(\feff_k(i\e),\mathcal H') - \sigmarh(\feff_k((i-1)\e),\mathcal H') \Big| \\
&~ \quad + \Big|\sigmarh(\feff_k((i-1)\e),\mathcal H) - \sigmarh(\feff_k((i-1)\e),\mathcal F) \Big| + 
\Big|\sigmarh(\feff_k((i-1)\e),\mathcal H') - \sigmarh(\feff_k((i-1)\e),\mathcal F') \Big| \\
& \text{(by Steps 2, 3, 4 above)} \\
& \leq \const \sum_{r=m}^{m'} \mathtt A(i\e),m\e). \qed
\end{split}
\end{equation*}

%
This concludes the proof of Lemma \ref{L_delta_dpi}.
\end{proof}

\noindent We can now prove Theorem \ref{T_decreasing}.

\begin{proof}[Proof of Theorem \ref{T_decreasing}]
Fix $m < m'$, $\tau \in S_m, \tau' \in S_{m'}$ and define
\begin{equation*}
\Delta \mathfrak q_k(\tau, \tau'):= \mathfrak q_k \big( i\e, \Theta(\tau), \Theta(\tau') \big) -  \mathfrak q_k \big( (i-1)\e, \tau, \tau' \big).
\end{equation*}
Since $m < m'$, it follows that $\Delta \mathfrak q_k(\tau, \tau')$ can be greater or equal than $0$ only if $\Phi_k((i-1)\e)^{-1}(\tau)$, $\Phi_k((i-1)\e)^{-1}(\tau')$ will interact after time $i\e$. In this case it holds
\begin{equation}
\label{E_pair_in_bij_2}
\begin{split}
\Delta \mathfrak q_k(\tau, \tau') &=  \mathfrak q_k \big( i\e, \Theta(\tau), \Theta(\tau') \big) -  \mathfrak q_k \big( (i-1)\e, \tau, \tau' \big) \crcr
& = \frac{\pi\big(i\e, \Theta(\tau), \Theta(\tau')\big)}{d\big(i\e, \Theta(\tau), \Theta(\tau')\big)} 
- \frac{\pi\big((i-1)\e, \tau, \tau'\big)}{d\big((i-1)\e, \tau, \tau'\big)} \crcr
& \leq  \pi\big(i\e, \Theta(\tau), \Theta(\tau')\big) \bigg(\frac{1}{d\big(i\e, \Theta(\tau), \Theta(\tau')\big)} 
- \frac{1}{d\big((i-1)\e, \tau, \tau'\big)}  \bigg) \crcr
& \qquad \qquad 
+ \frac{1}{d\big((i-1)\e, \tau, \tau'\big)} \bigg( 
\pi\big(i\e, \Theta(\tau), \Theta(\tau')\big) 
- \pi\big((i-1)\e, \tau, \tau'\big)
\bigg) \crcr
& = \frac{1}{d\big((i-1)\e, \tau, \tau'\big)}
\frac{\pi\big(i\e, \Theta(\tau), \Theta(\tau')\big) }{d\big(i\e, \Theta(\tau), \Theta(\tau')\big)} \Delta d(\tau, \tau')
+ \frac{1}{d\big((i-1)\e, \tau, \tau'\big)}  \Delta \pi(\tau, \tau') \crcr
& \leq \const \frac{1}{d\big((i-1)\e, \tau, \tau'\big)} \Big( \big|\Delta d(\tau, \tau')\big| + \Delta \pi(\tau, \tau')\Big) \\
& \leq \const \frac{1}{\tau' - \tau} \Big( \big|\Delta d(\tau, \tau')\big| + \Delta \pi(\tau, \tau')\Big) \crcr
\text{(by Lemma \ref{L_delta_dpi})} & \leq \const \frac{1}{\tau' - \tau} \sum_{r=m}^{m'} \mathtt A(i\e,r\e). 
\end{split}
\end{equation}
%

\noindent As observed before, for each $m \in \Z$, by Point \ref{Pt_affine} of Theorem \ref{T_lagrangian}, the map
\begin{equation*}
\Theta: S_m \to T_m, \qquad \Theta := \Phi_k(i\e) \circ \Phi_k((i-1)\e)^{-1},
\end{equation*}
is an affine function with slope $1$. Hence for any pair of integers $m < m'$ it is well defined the change of variable
\begin{equation*}
\Theta \times \Theta : S_m \times S_{m'} \to T_m \times T_{m'},
\end{equation*}
and we have
\begin{equation}
\label{E_pair_in_bij_1}
\begin{split}
\iint_{T_m \times T_{m'}} &\mathfrak q_k(i\e, \varsigma, \varsigma') d\varsigma d\varsigma' = 
\iint_{S_m \times S_{m'}} \mathfrak q_k (i\e, \Theta(\tau), \Theta(\tau')) d\tau d\tau'.
\end{split}
\end{equation}
We can now estimate the l.h.s of \eqref{E_pair_in_bij} as follows:
\begin{equation*}
\begin{split}
\sum_{m < m'} \Bigg[ \iint_{T_m \times T_{m'}} \mathfrak q_k(i\e) d\tau& d\tau' - \iint_{S_m \times S_{m'}} \mathfrak q_k((i-1)\e) d\tau d\tau' \Bigg] \\
\text{(by \eqref{E_pair_in_bij_1})} &= \sum_{m < m'} \iint_{S_m \times S_{m'}} \Big[ \mathfrak q_k \big( i\e, \Theta(\tau), \Theta(\tau') \big) -  \mathfrak q_k \big( (i-1)\e, \tau, \tau' \big) \Big] d\tau d\tau' \\
\text{(by \eqref{E_pair_in_bij_2})}  
& \leq \const \sum_{m < m'} \iint_{S_m \times S_{m'}} \frac{1}{\tau'-\tau} \sum_{r = m}^{m'} \mathtt A(i\e,r\e) d\tau d\tau' \\
& = \const \Bigg[ 
\sum_{r \in \Z} \mathtt A(i\e,r\e) \sum_{m = - \infty}^{r} \sum_{m' = r+1}^{+\infty} \iint_{S_m \times S_{m'}} \frac{1}{\tau'-\tau} d\tau d\tau' \\
& \qquad \qquad \qquad \qquad \qquad \qquad
+ \sum_{r \in \Z} \mathtt A(i\e,r\e) \sum_{m = - \infty}^{r-1} \iint_{S_m \times S_{r}} \frac{1}{\tau'-\tau} d\tau d\tau' \Bigg] \\
& \leq \const \Bigg[\sum_{r \in \Z} \mathtt A(i\e,r\e) \int_0^{\sup S_r} \int_{\inf S_{r+1}}^{+\infty} \frac{1}{\tau'-\tau} d\tau d\tau \\
& \qquad \qquad \qquad \qquad \qquad \qquad
+  \sum_{r \in \Z} \mathtt A(i\e,r\e) \int_0^{\sup S_{r-1}} \int_{\inf S_{r}}^{+\infty} \frac{1}{\tau'-\tau} d\tau d\tau' 
\Bigg],
\end{split}
\end{equation*}
and since $\sup S_{r-1}< \inf S_r \leq \sup S_r < \inf S_{r+1}$ after an elementary integration by parts,
\begin{equation*}
\begin{split}
\ldots & \leq \const L_k\big((i-1)\e\big) \sum_{r \in \Z} \mathtt A(i\e,m\e) \\
\text{(by \eqref{E_bd_on_TV})} & \leq \const \TV(u(0)) \sum_{r \in \Z} \mathtt A(i\e,m\e),
\end{split}
\end{equation*} 
thus concluding the proof of Theorem \ref{T_decreasing}.
\end{proof}

\subsection{Analysis of interacting waves}
\label{Ss_interacting}

This section is devoted to conclude the proof of Theorem \ref{T_variation_fQ}, showing that inequality \eqref{E_interacting} holds, i.e. estimating the (negative) term related to pairs of waves which are divided at time $(i-1)\e$ and are interacting at time $i\e$. In particular we will prove the following theorem:

\begin{theorem}
\label{T_increasing}
The following estimate holds:
\begin{equation}
\label{E_increasing}
- \sum_{m \in \Z}
 \iint_{J_m^L \times J_m^R} \mathfrak q_k((i-1)\e) d\tau d\tau' 
\leq 
- \sum_{m \in \Z} \Aquadr_k(i\e,m\e)
+
\const \TV(u(0)) \sum_{m \in \Z} \mathtt A(i\e,m\e).
\end{equation}
\end{theorem}

To prove this theorem, we first study the change of the l.h.s. of \eqref{E_increasing} due to transversal interactions (Lemma \ref{L_tilde_q}) and then we study the interaction between waves of the same family and same sign (Proposition \ref{P_increasing}). 

First of all we introduce the following notations. Fix $m \in \Z$ and assume $J_m^L, J_m^R \neq \emptyset$ (see \eqref{E_trans_surv} for the definition). Consider the grid point $(i\e,m\e)$ and the two incoming Riemann problems $(u^{i,m-1}, u^{i-1,m-1})$ and $(u^{i-1,m-1}, u^{i,m})$. Assume that
\begin{equation*}
u^{i-1,m-1} = T^{n}_{s'_n} \circ \dots \circ T^{1}_{s'_1} u^{i,m-1}, \qquad  
u^{i,m} = T^{n}_{s''_n} \circ \dots \circ T^{1}_{s''_1} u^{i-1,m-1},
\end{equation*}
and that their elementary curves are 
\begin{align*}
\gamma_1' &= (u_1', v_1', \sigma_1') := \gamma_1(u^{i,m-1}, s_1'), & \gamma_h' &= (u_h', v_h', \sigma_h') := \gamma_h\big(u_{h-1}'(s_{h-1}'), s_h'\big), \quad \text{for } h = 2, \dots, n, \crcr
\gamma_1''&=  (u_1'', v_1'', \sigma_1'') := \gamma_1(u^{i-1,m-1}, s_1''), & \gamma_h'' &= (u_h'', v_h'', \sigma_h'')  := \gamma_h\big(u_{h-1}''(s_{h-1}''), s_h''\big), \quad \text{for } h = 2, \dots, n,
\end{align*}
see Section \ref{Ss_Rp} for the notation.

We are interested in the $k$-th family. Denote by $f_k', f_k''$ the reduced fluxes associated respectively to $\gamma_k', \gamma_k''$ and w.l.o.g. assume that $f_k', f_k''$ are defined respectively on $J_m^L$, $J_m^R$. Now consider the collection of curves
\begin{align*}
&\tilde \gamma_1'  = (\tilde u_1', \tilde v_1', \tilde \sigma_1') := \gamma_1(u^{i,m-1}, s_1'),  && \tilde \gamma_1''  = (\tilde u_1'', \tilde v_1'', \tilde \sigma_1'') := \gamma_1(u_1'(s_1'), s_1''), \crcr
& \tilde \gamma_h'  = (\tilde u_h', \tilde v_h', \tilde \sigma_h') := \gamma_h \big(\tilde u_{h-1}''(s_{h-1}''), s_h'\big), && 
\tilde \gamma_h''  = (\tilde u_h'', \tilde v_h'', \tilde \sigma_h'') := \gamma_h \big(\tilde u_{h}'(s_{h}'), s_h''\big), 
\quad \text{for } h = 2, \dots, n.
\end{align*}
This is the situation after the transversal interactions considered at the beginning of Section \ref{S_local}, see Figure \ref{F_el_curves_after_trans}. For the $k$-th (fixed) family, denote by $\tilde f_k', \tilde f_k''$ the reduced fluxes associated to the curves $\tilde \gamma_k', \tilde \gamma_k''$ and let
\begin{equation*}
\tilde f := \tilde f_k' \cup \tilde f_k''.
\end{equation*}
As before, w.l.o.g. assume that $\tilde f_k', \tilde f_k''$ are defined respectively on $J_m^L$, $J_m^R$.

For any $(\tau, \tau') \in J_m^L \times J_m^R$, consider $w:= \Phi_k((i-1)\e)^{-1}(\tau)$, $w' :=  \Phi_k((i-1)\e)^{-1}(\tau')$ and define the quantity
\begin{equation*}
\tilde{\mathfrak q}_k(\tau, \tau') := 
\frac{\tilde \pi_k(\tau, \tau')}{d_k((i-1)\e,\tau, \tau')},
\end{equation*}
where $\tilde \pi(\tau, \tau')$ is defined as follows:
if $\mathcal J, \mathcal J' \in \mathcal P((i-1)\e,w,w')$ are the element of the partition at time $(i-1)\e$  containing $w,w'$ respectively,
then $\tilde \pi(\tau, \tau')$ is defined as in \eqref{E_def_pi}, with $\tilde f$ instead of $\feff_k((i-1)\e)$, i.e.
\begin{equation*}
\tilde \pi(\tau, \tau'):=
\Big[ \sigmarh(\tilde f, \mathcal J) - \sigmarh(\tilde f, \mathcal J') \Big]^+.
\end{equation*}
Recall that since $w,w'$ are interacting at time $i\e$, then $\mathcal J = \mathcal K$, $\mathcal J' = \mathcal K'$ in \eqref{E_element_of_part}. We can now study the change of the l.h.s. of \eqref{E_increasing}, due to transversal interaction. This is done in the next lemma.

\begin{lemma}
\label{L_tilde_q}
It holds
\begin{equation*}
\begin{split}
\iint_{J_m^L \times J_m^R} \Big[ - \mathfrak q_k((i-1)\e, \tau, \tau') +  \tilde{\mathfrak q}_k(\tau, \tau') \Big] d\tau d\tau' &\leq \const \Atrans(i\e,m\e) \mathcal L^1(J_m) \\
&\leq \const \TV(u(0)) \Atrans(i\e,m\e).
\end{split}
\end{equation*}
\end{lemma}

\begin{proof}
We first prove that for any $(\tau, \tau') \in J_m^L \times J_m^R$, 
\begin{equation}
\label{E_tilde_pi}
\big|\pi((i-1)\e, \tau, \tau') - \tilde \pi(\tau, \tau')\big| \leq \const \Atrans(i\e,m\e).
\end{equation}
As in \eqref{E_feff_der_sec}, choose $\feff_k((i-1)\e)$ such that 
\begin{equation*}
\frac{d\feff_k((i-1)\e)}{d\tau}\big(\inf J_m^L\big) =
\frac{d \tilde f_k'}{d \tau}\big(\inf J_m^L\big).
\end{equation*} 
For any $\bar \varsigma \in J_m$, it holds
\begin{equation}
\label{E_varsigma}
\begin{split}
\bigg|
\frac{d \feff_k((i-1)\e)}{d\tau}(\bar \varsigma)
 - \frac{d \tilde f}{d\tau}(\bar \varsigma)
 \bigg| &\leq 
\int_0^{\bar \varsigma} \bigg|\frac{d^2 \feff_k((i-1)\e)}{d\tau^2}(\varsigma) - \frac{d^2 \tilde f}{d\tau^2}(\varsigma)\bigg| d\varsigma 
\crcr
& \leq 
\bigg\|
\frac{d^2 \feff_k((i-1)\e)}{d\tau^2}
 - \frac{d^2 \tilde f}{d\tau^2}
 \bigg\|_{L^1(J_m)} \crcr	
\bigg(\text{by Lemma \ref{L_general_transversal}}&\text{ and since }\frac{d^2 \feff_k((i-1)\e)}{d\tau^2} = \frac{d^2  (f_k' \cup f_k'')}{d\tau^2} \text{ a.e. on } J_m\bigg) \crcr
& \leq \const \Atrans(i\e,m\e).
\end{split}
\end{equation}

\noindent For $(\tau, \tau') \in J^L_m \times J^R_m$, set $w := \Phi_k((i-1)\e)^{-1}(\tau)$, $w' := \Phi_k((i-1)\e)^{-1}(\tau')$; let $\mathcal J, \mathcal J' \in \mathcal P((i-1)\e,w,w')$, $w \in \mathcal J$, $w' \in \mathcal J'$ be the element of the partition containing $w,w'$ respectively; since
\begin{equation*}
\tint((i-1)\e,w,w') - \e = (i-1)\e,
\end{equation*}
using \eqref{E_element_of_part} we have 
\begin{equation*}
\pi_k((i-1)\e, \tau, \tau') = \Big[\sigmarh(\feff_k((i-1)\e), \mathcal J) - \sigmarh(\feff_k((i-1)\e), \mathcal J') \Big]^+.
\end{equation*}
Hence
\begin{equation*}
\begin{split}
\big|\pi_k((i-1)&\e, \tau, \tau') - \pi_k(\tau, \tau')\big| \\
& = \bigg|
\Big[\sigmarh(\feff_k((i-1)\e), \mathcal J) - \sigmarh(\feff_k((i-1)\e), \mathcal J') \Big]^+
- \Big[\sigmarh(\tilde f, \mathcal J) - \sigmarh(\tilde f, \mathcal J') \Big]^+ \bigg|\\
& \leq 
\Big|
 \sigmarh(\feff_k((i-1)\e), \mathcal J) - \sigmarh(\tilde f, \mathcal J)
\Big|
+ \Big|
\sigmarh(\feff_k((i-1)\e), \mathcal J') - \sigmarh(\tilde f, \mathcal J') 
\Big| \\
& \leq \bigg\|
\frac{d \feff_k((i-1)\e)}{d\tau}
 - \frac{d \tilde f}{d\tau}
 \bigg\|_{L^\infty(J_m)} \\
\text{(by \eqref{E_varsigma})} & 
\leq \const \Atrans(i\e,m\e),
\end{split}
\end{equation*}
thus proving \eqref{E_tilde_pi}. 

As an immediate consequence, we have that for any $(\tau, \tau') \in J_m^L \times J_m^R$, it holds
\begin{equation*}
\Big| \tilde {\mathfrak q}_k(\tau, \tau') - \mathfrak q_k((i-1)\e, \tau, \tau') \Big| \leq \const \frac{1}{\tau' - \tau} \Atrans(i\e,m\e).
\end{equation*}
We thus have as in the proof of Theorem \ref{T_decreasing}
\begin{equation*}
\begin{split}
\iint_{J_m^L \times J_m^R} \Big[- \mathfrak q_k((i-1)\e, \tau, \tau') + \tilde{\mathfrak q}(\tau, \tau') \Big] d\tau d\tau' 
& \leq  \const \Atrans(i\e,m\e) \iint_{J_m^L \times J_m^R}  \frac{1}{\tau' - \tau} d\tau d\tau' \\
& \leq \const \Atrans(i\e,m\e) \mathcal L^1(J_m), 
\end{split}
\end{equation*}
thus concluding the proof of the lemma, because $\mathcal L^1(J_m) \leq \const \TV(u(0))$ by \eqref{E_bd_on_TV}.
\end{proof}

Now, to conclude the proof of Theorem \ref{T_increasing} it is sufficient to prove the following proposition. 

\begin{proposition}
\label{P_increasing}
It holds
\begin{equation*}
\iint_{J_m^L \times J_m^R} \tilde{\mathfrak q}_k(\tau, \tau') d\tau d\tau'
\geq \Aquadr_k(i\e,m\e).
\end{equation*}
\end{proposition}

\begin{proof}
Set 
\begin{equation*}
\tau_M := \sup J_m^L = \inf J_m^R, 
\end{equation*}
and
\begin{equation*}
\begin{split}
\tau_L &:= \max \Big\{\tau \in J_m^L \ \big| \ \conv_{J_m^L} \tilde f(\tau) = \conv_{J_m^L \cup J_m^R} \tilde f(\tau) \Big\}, \\
\tau_R &:= \min \Big\{\tau \in J_m^R \ \big| \ \conv_{J_m^R} \tilde f(\tau) = \conv_{J_m^L \cup J_m^R} \tilde f(\tau) \Big\}.
\end{split}
\end{equation*}

\noindent W.l.o.g. we assume that $\tau_L < \tau_M < \tau_R$, otherwise there is nothing to prove. 

It is quite easy to see that
\begin{equation*}
\begin{split}
\Aquadr(i\e,m\e) & = \tilde f(\tau_M) - \conv_{[\inf J_m^L, \sup J_m^R]} \tilde f  \\
& = \tilde f(\tau_M) - \conv_{[\tau_L, \tau_R]} \tilde f  \\
& =  \frac{1}{\tau_R - \tau_L} 
\bigg[\bigg(\frac{\tilde f(\tau_M)- \tilde f(\tau_L)}{\tau_M - \tau_L} - \frac{\tilde f(\tau_R) - \tilde f (\tau_M)}{\tau_R - \tau_M} \bigg)
\big(\tau_M - \tau_L \big) \big(\tau_R - \tau_M \big)
 \bigg] \\
& = \frac{1}{\tau_R - \tau_L} 
\Big[ \sigmarh \big( \tilde f, (\tau_L, \tau_M] \big) - \sigmarh \big( \tilde f, (\tau_M, \tau_R] \big) \Big] \mathcal L^2 \big( (\tau_L, \tau_M] \times (\tau_M, \tau_R]  \big).
\end{split}
\end{equation*}

\noindent It is thus sufficient to prove that
\begin{equation}
\label{E_quadr_est}
\begin{split}
\frac{1}{\tau_R - \tau_L} 
\Big[\sigmarh \big( \tilde f, (\tau_L, \tau_M ] \big) - \sigmarh \big( \tilde f, (\tau_M, \tau_R] \big) \Big] \mathcal L^2 \big( (\tau_L, \tau_M] \times (\tau_M, \tau_R] \big)  \leq \int_{\tau_L}^{\tau_M} \int_{\tau_M}^{\tau_R} \tilde {\mathfrak q}_k(\tau, \tau') d\tau d\tau'.
\end{split} 
\end{equation}

\noindent Observe that, by Proposition \ref{P_divise_partizione_implica_divise_realta}, 
\begin{equation*}
d((i-1)\e, \tau, \tau') \leq \tau_R - \tau_L;
\end{equation*}
hence \eqref{E_quadr_est} will follow if we prove that
\begin{equation}
\label{E_delta_sigma}
\Big[\sigmarh \big( \tilde f,(\tau_L, \tau_M] \big) - \sigmarh \big( \tilde f, (\tau_M, \tau_R] \big) \Big] \mathcal L^2\big((\tau_L, \tau_M] \times (\tau_M, \tau_R] \big) 
\leq \int_{\tau_L}^{\tau_M} \int_{\tau_M}^{\tau_R} \tilde \pi(\tau, \tau') d\tau d\tau'.
\end{equation}


Let
\begin{equation*}
\mathcal L := \Phi_k((i-1)\e)^{-1} \big((\tau_L, \tau_M] \big), \qquad
\mathcal R := \Phi_k((i-1)\e)^{-1} \big(( \tau_M, \tau_R] \big).
\end{equation*}
We will identify waves through the equivalence relation $\bowtie$ introduced in \eqref{E_big_waves}: for any couple of waves $w,w' \in \mathcal J \cup \mathcal R$, set $w \bowtie w'$ if and only if 
\begin{equation*}
\tcr(w) = \tcr(w') \text{ and } 
\mathtt x (t,w) = \mathtt x(t,w') \quad \text{ for any }
t \in \Big[ \tcr(w), i\e \Big).
\end{equation*}
As observed in Lemma \ref{L_discrete_partition}, the sets
\begin{equation*}
\widehat{\mathcal L} := \mathcal L \ \big/ \bowtie, \qquad 
\widehat{\mathcal R} := \mathcal R \ \big/ \bowtie
\end{equation*}
are finite and totally ordered by the order $\leq$ on $\W_k^+((i-1)\e)$. Moreover for any $\xi \in \widehat{\mathcal L}$, $\xi' \in \widehat{\mathcal R}$, let $w \in \xi$, $w' \in \xi'$ and set
\begin{equation*}
\mathcal I((i-1)\e, \xi, \xi') := \mathcal I((i-1)\e, w, w'), \qquad
\mathcal P((i-1)\e, \xi, \xi') := \mathcal P((i-1)\e, w,w'),
\end{equation*}
and
\begin{equation*}
\widehat{\mathcal I}((i-1)\e, \xi, \xi') := \mathcal I((i-1)\e, \xi,\xi') \ \big/ \bowtie.
\end{equation*}
The above definitions are well posed thanks to Lemma \ref{L_big_waves} and Lemma \ref{L_iow_tecn}, Point \eqref{Pt_I_as_big_waves}. It is moreover quite easy to see that $\widehat{\mathcal I} \subseteq \widehat{\mathcal L} \cup \widehat{\mathcal R}$. 

\begin{figure}
\begin{tikzpicture}
\draw (0,0) rectangle (7,8);
\draw (1,2) rectangle (6,7);
\draw[dashed] (4,2) to (4,7);
\draw[dashed] (1,5) to (6,5);
\draw[<->] (0,-0.2) to (7,-0.2);
\draw[<->] (1,1.8) to (6,1.8);
\draw[<->] (0,-0.2) to (7,-0.2);
\draw[<->] (7.2,8) to (7.2,0);
\draw[<->] (6.2,2) to (6.2,7);
\node[below] at (3.5,-0.2) {\Large $\widehat{\mathcal L}$};
\node[below] at (3.5,1.8) {\Large $\widehat{\mathcal L}_{\mathcal C}$};
\node[right] at (7.2,4) {\Large  $\widehat{\mathcal R}$};
\node[right] at (6.2,4.5) {\Large  $\widehat{\mathcal R}_{\mathcal C}$};
\node at (5,3.5) {\Large  $\Pi_0(\mathcal{\widehat C})$};
\node at (5,6) {\Large $\Pi_1(\mathcal{\widehat C})$};
\node at (2.5,6) {\Large $\Pi_2(\mathcal{\widehat C})$};
\node at (2.5,3.5) {\Large $\Pi_3(\mathcal{\widehat C})$};
\end{tikzpicture}
\caption{Partition of $\widehat C := \widehat{\mathcal L}_{\mathcal C} \times \widehat{\mathcal R}_{\mathcal C}$.}
\label{F_partition_lxr}
\end{figure}
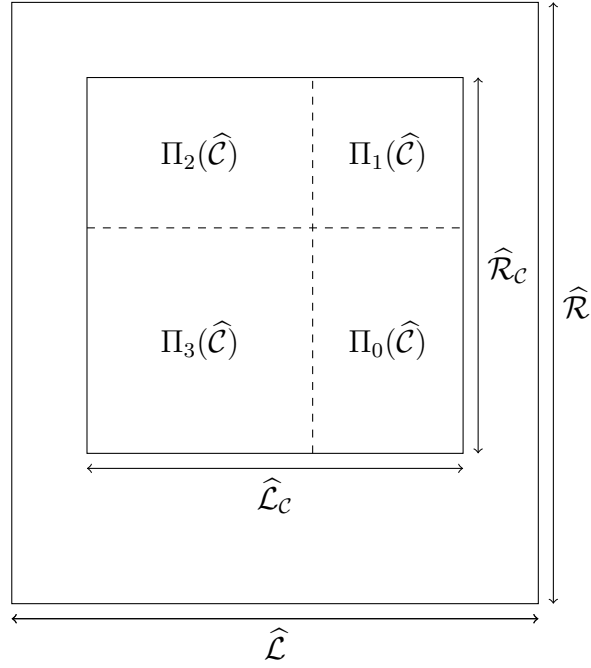

Now we partition the rectangle $\widehat{\mathcal L} \times \widehat{\mathcal R}$ in sub-rectangles, as follows. For any rectangle $\widehat{\mathcal C} := \widehat{\mathcal L}_{\mathcal C} \times \widehat{\mathcal R}_{\mathcal C} \subseteq \widehat{\mathcal L} \times \widehat{\mathcal R}$, define (see Figure \ref{F_partition_lxr})
\[
\Pi_0 (\widehat{\mathcal C}) := 
\begin{cases}
\emptyset, & \widehat{\mathcal C} = \emptyset, \\
\Big[\widehat{\mathcal L}_{\mathcal C} \cap \widehat{\mathcal I}((i-1)\e, \max \widehat{\mathcal L}_{\mathcal C}, \min \widehat{\mathcal R}_{\mathcal C}) \Big] \times \Big[\widehat{\mathcal R}_{\mathcal C} \cap \widehat{\mathcal I}((i-1)\e, \max \widehat{\mathcal L}_{\mathcal C}, \min \widehat{\mathcal R}_{\mathcal C})\Big], & \widehat{\mathcal C} \neq \emptyset,
\end{cases}
\]

\[
\Pi_1 (\widehat{\mathcal C}) := 
\begin{cases}
\emptyset, 
& \widehat{\mathcal C} = \emptyset, \\
\Big[\widehat{\mathcal L}_{\mathcal C} \cap \widehat{\mathcal I}((i-1)\e, \max \widehat{\mathcal L}_{\mathcal C}, \min \widehat{\mathcal R}_{\mathcal C}) \Big] \times \Big[\widehat{\mathcal R}_{\mathcal C} \setminus \widehat{\mathcal I}((i-1)\e, \max \widehat{\mathcal L}_{\mathcal C}, \min \widehat{\mathcal R}_{\mathcal C})\Big], &
\widehat{\mathcal C} \neq \emptyset,
\end{cases}
\]

\[
\Pi_2 (\widehat{\mathcal C}) := 
\begin{cases}
\emptyset, 
& \widehat{\mathcal C} = \emptyset, \\
\Big[\widehat{\mathcal L}_{\mathcal C} \setminus \widehat{\mathcal I}((i-1)\e, \max \widehat{\mathcal L}_{\mathcal C}, \min \widehat{\mathcal R}_{\mathcal C}) \Big] \times \Big[\widehat{\mathcal R}_{\mathcal C} \setminus \widehat{\mathcal I}((i-1)\e, \max \widehat{\mathcal L}_{\mathcal C}, \min \widehat{\mathcal R}_{\mathcal C})\Big], &
\widehat{\mathcal C} \neq \emptyset,
\end{cases}
\]

\[
\Pi_3 (\widehat{\mathcal C}) := 
\begin{cases}
\emptyset, 
& \widehat{\mathcal C} = \emptyset, \\
\Big[\widehat{\mathcal L}_{\mathcal C} \setminus \widehat{\mathcal I}((i-1)\e, \max \widehat{\mathcal L}_{\mathcal C}, \min \widehat{\mathcal R}_{\mathcal C}) \Big] \times \Big[\widehat{\mathcal R}_{\mathcal C} \cap \widehat{\mathcal I}((i-1)\e, \max \widehat{\mathcal L}_{\mathcal C}, \min \widehat{\mathcal R}_{\mathcal C})\Big], &
\widehat{\mathcal C} \neq \emptyset,
\end{cases}
\]
Clearly $\Big\{\Pi_0(\widehat{\mathcal C}), \Pi_1(\widehat{\mathcal C}), \Pi_2(\widehat{\mathcal C}), \Pi_3(\widehat{\mathcal C}) \Big\}$ is a disjoint partition of $\widehat{\mathcal C}$.

For any set $A$, denote by $A^{<\N}$ the set of all finite sequences taking values in $A$. We assume that $\emptyset \in A^{<\N}$, called the \emph{empty sequence}. There is a natural ordering $\unlhd$ on $A^{<\N}$: given $\alpha, \beta \in A^{<\N}$,
\[
\alpha \unlhd \beta \quad \Longleftrightarrow \quad \text{$\beta$ is obtained from $\alpha$ by adding a finite sequence.}
\]
A subset $D \subseteq A^{<\N}$ is called a \emph{tree} if for any $\alpha, \beta \in A^{<\N}$, $\alpha \unlhd \beta$, if $\beta \in D$, then $\alpha \in D$. 

Define a map $\widehat{\Psi}: \{0, 1,2,3\}^{<\N} \longrightarrow 2^{\widehat{\mathcal L} \times \widehat{\mathcal R}}$, by setting
\[
\begin{split}
\widehat{\Psi}_\alpha  = 
\begin{cases}
\widehat{\mathcal L} \times \widehat{\mathcal R}, & \text{ if } \alpha = \emptyset, \\
\Pi_{a_n} \circ \dots \circ \Pi_{a_1}(\widehat{\mathcal L} \times \widehat{\mathcal R}), & \text{ if }\alpha = (a_1, \dots, a_n) \in \{0,1,2,3\}^{<\N} \setminus \{\emptyset\}.
\end{cases}
\end{split}
\]
For $\alpha \in \{0,1,2,3\}^{<\N}$, let $\widehat{\mathcal L}_\alpha, \widehat{\mathcal R}_\alpha$ be defined by the relation $\widehat{\Psi}_\alpha = \widehat{\mathcal L}_\alpha \times \widehat{\mathcal R}_\alpha$. 
Define a tree $D$ in $\{0,1,2,3\}^{<\N}$ setting 
\begin{equation*}
D := \big\{\emptyset\big\} \cup \bigg\{ \alpha = (a_1, \dots, a_n) \in \{0,1,2,3\}^{<\N} \ \Big| \ n \in \N, \ \widehat{\Pi}_\alpha \neq \emptyset, \ a_{k} \neq 0 \ \text{for} \ k=1,\dots,n-1 \bigg\}.
\end{equation*}
See Figure \ref{F_tree}.

\begin{figure}
\begin{tikzpicture}
\draw (0,0) rectangle (7,8);
\draw (4,0) to (4,8);
\draw (0,3) to (7,3);
\draw[dashed] (5.6, 3) to (5.6, 8);
\draw[dashed] (4,5) to (7,5);
\draw[dashed] (1.6, 3) to (1.6, 8);
\draw[dashed] (0,4.6) to (4,4.6);
\draw[dashed] (0,1.8) to (4,1.8);
\node[below] at (3.5,-0.2) {\Large $\widehat{\mathcal L}$};
\node[right] at (7.2,4) {\Large  $\widehat{\mathcal R}$};
\node at (5.5,1.5) {\Large $\hat{\Pi}_0$};
\node at (6.3,4) {\Large $\hat \Pi_{10}$};
\node at (6.3,6.5) {\Large $\hat \Pi_{11}$};
\node at (4.8,6.5) {\Large $\hat \Pi_{12}$};
\node at (4.8,4) {\Large $\hat \Pi_{13}$};
\node at (2.8,3.8) {\Large $\hat \Pi_{20}$};
\node at (2.8,6.3) {\Large $\hat \Pi_{21}$};
\node at (0.8,6.3) {\Large $\hat \Pi_{22}$};
\node at (0.8,3.8) {\Large $\hat \Pi_{23}$};
\node at (2,0.8) {\Large $\hat \Pi_{30}$};
\node at (2,2.4) {\Large $\hat \Pi_{31}$};
\end{tikzpicture}
\caption{Partition of $\mathcal L \times \mathcal R$ using the tree $D$.}
\label{F_tree}
\end{figure}
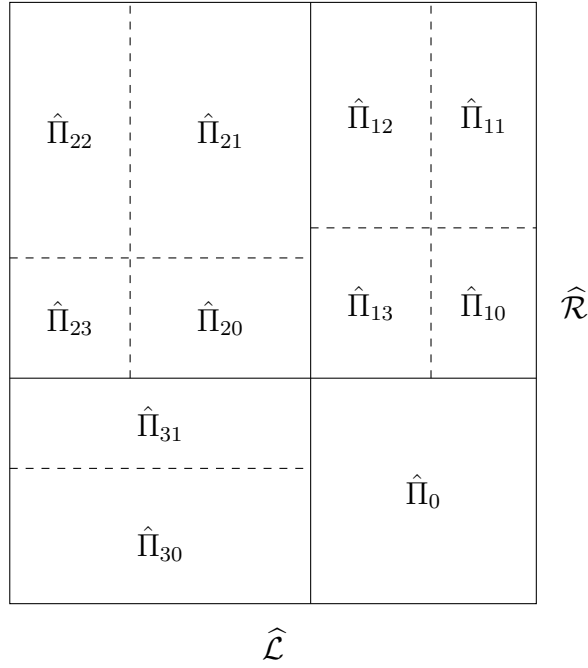

\noindent Since $\Pi_0(\Pi_0(\widehat{\mathcal C})) = \Pi_0(\widehat{\mathcal C})$ for any $\widehat{\mathcal C} \subseteq \widehat{\mathcal L} \times \widehat{\mathcal R}$,  this implies, together with the fact that $\widehat{\mathcal L} \times \widehat{\mathcal R}$ is a finite set, that $D$ is a finite tree.

For any $\alpha \in D$, set
\begin{align*}
\mathcal L_\alpha & := \bigcup_{\xi \in \widehat{\mathcal L}_\alpha} \xi, & 
\mathcal R_\alpha & := \bigcup_{\xi' \in \widehat{\mathcal R}_\alpha} \xi', \\
L_\alpha & := \Phi_k((i-1)\e) (\mathcal L_\alpha), &
R_\alpha & := \Phi_k((i-1)\e) (\mathcal R_\alpha).
\end{align*}
The idea of the proof is to show that, for each $\alpha \in D$, on the rectangle $L_\alpha \times  R_\alpha$ it holds
\begin{equation}
\label{E_fund_ineq}
\begin{split}
\big[ \sigmarh(\tilde f, L_\alpha) - \sigmarh(\tilde f, R_\alpha) \big] 
\mathcal L^2(L_\alpha \times R_\alpha)
\leq \int_{L_\alpha \times R_\alpha} \tilde \pi(\tau, \tau') d\tau d\tau'.
\end{split}
\end{equation}
The conclusion will follow just considering that $\emptyset \in D$ and $L_\emptyset = (\tau_L, \tau_M]$, $R_\emptyset = (\tau_M, \tau_R]$.


We now need the following two Lemmas. 

\begin{lemma}
\label{L_incastro_2}
For any $\beta \in D$, the partition $\mathcal P((i-1)\e, \max \widehat{\mathcal L}_\beta, \min \widehat{\mathcal R}_\beta)$ of the characteristic interval $\mathcal I((i-1)\e, \max \widehat{\mathcal L}_\beta, \min \widehat{\mathcal R}_\beta)$ can be restricted to 
\[
\mathcal L_\beta \cap \mathcal I((i-1)\e, \max \widehat{\mathcal L}_\beta, \min \widehat{\mathcal R}_\beta)
\]
and to
\[
\mathcal R_\beta \cap \mathcal I((i-1)\e, \max \widehat{\mathcal L}_\beta, \min \widehat{\mathcal R}_\beta).
\]
\end{lemma}

\begin{proof}
Let us prove only the first part of the statement, the second one being completely similar. We will show by induction the following stronger claim:
 
\smallskip
\noindent for each $\gamma \unlhd \beta$, the partition $\mathcal P((i-1)\e, \max \widehat{\mathcal L}_\beta, \min \widehat{\mathcal R}_\beta)$ of the interval $\mathcal I((i-1)\e, \max \widehat{\mathcal L}_\beta, \min \widehat{\mathcal R}_\beta)$ can be restricted to $\mathcal L_\gamma \cap \mathcal I((i-1)\e, \max \widehat{\mathcal L}_\beta, \min \widehat{\mathcal R}_\beta)$.
\smallskip

\noindent For $\gamma = \emptyset$, by definition $\mathcal L_{\emptyset} = \mathcal L$ and thus the proof is an easy consequence of Proposition \ref{P_divise_partizione_implica_divise_realta}. 
Thus assume the claim is true for some $\gamma \lhd \beta$ and let us prove it for $\gamma a $, with $a \in \{0,1,2,3\}$.

If $a = 0,1$, by definition it holds
\[
\mathcal L_{\gamma a} = \mathcal L_\gamma \cap \mathcal I((i-1)\e, \max \widehat{\mathcal L}_{\gamma}, \min \widehat{\mathcal R}_\gamma).
\]
Hence 
\[
\mathcal L_{\gamma a} \cap \mathcal I((i-1)\e, \max \widehat{\mathcal L}_{\beta}, \min \widehat{\mathcal R}_\beta) 
= \mathcal L_\gamma \cap \mathcal I((i-1)\e, \max \widehat{\mathcal L}_{\gamma}, \min \widehat{\mathcal R}_\gamma) \cap \mathcal I((i-1)\e, \max \widehat{\mathcal L}_{\beta}, \min \widehat{\mathcal R}_\beta).
\]
By inductive assumption, the partition $\mathcal P((i-1)\e, \max \widehat{\mathcal L}_\beta, \min \widehat{\mathcal R}_\beta)$ of $\mathcal I((i-1)\e, \max \widehat{\mathcal L}_\beta, \min \widehat{\mathcal R}_\beta)$ can be restricted to $\mathcal L_\gamma \cap \mathcal I((i-1)\e, \max \widehat{\mathcal L}_{\beta}, \min \widehat{\mathcal R}_\beta)$, while, since $\gamma \lhd \beta$,
\begin{equation*}
\max \widehat{\mathcal L}_\beta \leq \max \widehat{\mathcal L}_\gamma \leq \min \widehat{\mathcal R}_\gamma \leq \min \widehat{\mathcal R}_\beta
\end{equation*}
and therefore, by Proposition \ref{P_partition_restr}, the partition $\mathcal P((i-1)\e, \max \widehat{\mathcal L}_\beta, \min \widehat{\mathcal R}_\beta)$ can be restricted also to $\mathcal I((i-1)\e, \max \widehat{\mathcal L}_{\gamma}, \min \widehat{\mathcal R}_\gamma) \cap \mathcal I((i-1)\e, \max \widehat{\mathcal L}_{\beta}, \min \widehat{\mathcal R}_\beta)$, and thus we are done.

If $a =2,3$, by definition it holds
\[
\mathcal L_{\gamma a} = \mathcal L_\gamma \setminus \mathcal I((i-1)\e, \max \widehat{\mathcal L}_{\gamma}, \min \widehat{\mathcal R}_\gamma).
\]
Hence 
\begin{align*}
\mathcal L_{\gamma a} &\cap \mathcal I((i-1)\e, \max \widehat{\mathcal L}_{\beta}, \min \widehat{\mathcal R}_\beta) \\
&= \Big( \mathcal L_\gamma \setminus \mathcal I((i-1)\e, \max \widehat{\mathcal L}_{\gamma}, \min \widehat{\mathcal R}_\gamma) \Big) \cap \mathcal I((i-1)\e, \max \widehat{\mathcal L}_{\beta}, \min \widehat{\mathcal R}_\beta) \\
&= \Big( \mathcal L_\gamma \cap \mathcal I((i-1)\e, \max \widehat{\mathcal L}_{\beta}, \min \widehat{\mathcal R}_\beta)  \Big) 
\cap \Big(\mathcal I((i-1)\e, \max \widehat{\mathcal L}_{\beta}, \min \widehat{\mathcal R}_\beta) \setminus  \mathcal I((i-1)\e, \max \widehat{\mathcal L}_{\gamma}, \min \widehat{\mathcal R}_\gamma) \Big).
\end{align*}
As in the case $a =0,1$, by inductive assumption, the partition $\mathcal P((i-1)\e, \max \widehat{\mathcal L}_\beta, \min \widehat{\mathcal R}_\beta)$ of the interval $\mathcal I((i-1)\e, \max \widehat{\mathcal L}_\beta, \min \widehat{\mathcal R}_\beta)$ can be restricted to $\mathcal L_\gamma \cap \mathcal I((i-1)\e, \max \widehat{\mathcal L}_{\beta}, \min \widehat{\mathcal R}_\beta)$, while, as before, by Proposition \ref{P_partition_restr} using $\gamma \lhd \beta$, it can be restricted also to $\mathcal I((i-1)\e, \max \widehat{\mathcal L}_{\beta}, \min \widehat{\mathcal R}_\beta) \setminus \mathcal I((i-1)\e, \max \widehat{\mathcal L}_{\gamma}, \min \widehat{\mathcal R}_\beta)$, and thus we are done also in this case.
\end{proof}

\begin{lemma}
\label{L_estim_phi_zero_psi_alpha}
For each $\alpha =(a_1, \dots a_n) \in D$, if $a_n = 0$, then it holds
\[
\begin{split}
\big[ \sigmarh(\tilde f, L_\alpha) - \sigmarh(\tilde f, R_\alpha) \big] \mathcal L^2(L_\alpha \times R_\alpha)
\leq 
\iint_{L_\alpha \times R_\alpha} \tilde \pi(\tau, \tau')d\tau d\tau'.\end{split}
\]
\end{lemma}

\begin{proof}

Set $\beta := (a_1, \dots, a_{n-1})$. Since $a_n = 0$, then 
\[
\widehat{\Psi}_\alpha = \Pi_0(\widehat{\Psi}_\beta) = \Big(\widehat{\mathcal L}_\beta \cap \widehat{\mathcal I}((i-1)\e, \max \widehat{\mathcal L}_\beta, \min \widehat{\mathcal R}_\beta) \Big) \times \Big(\widehat{\mathcal R}_\beta \cap \widehat{\mathcal I}((i-1)\e, \max \widehat{\mathcal L}_\beta, \min \widehat{\mathcal R}_\beta) \Big),
\]
and thus
\begin{equation*}
\mathcal L_\alpha = \mathcal L_\beta \cap \mathcal I((i-1)\e, \max \widehat{\mathcal L}_\beta, \min \widehat{\mathcal R}_\beta), \qquad
\mathcal R_\alpha = \mathcal R_\beta \cap \mathcal I((i-1)\e, \max \widehat{\mathcal L}_\beta, \min \widehat{\mathcal R}_\beta).
\end{equation*}

Consider the partition $\mathcal P((i-1)\e, \max \widehat{\mathcal L}_\beta, \min \widehat{\mathcal R}_\beta)$ of the interval $\mathcal I((i-1)\e, \max \widehat{\mathcal L}_\beta, \min \widehat{\mathcal R}_\beta)$ and set
\begin{equation*}
\mathbf P := \Big\{ \Phi_k((i-1)\e) (\mathcal J) \ \big| \ \mathcal J \in \mathcal P((i-1)\e, \max \widehat{\mathcal L}_\beta, \min \widehat{\mathcal R}_\beta) \Big\}.
\end{equation*}
By definition of the partition in Section \ref{Ss_partition}, the elements of $\mathbf P$ are intervals in $\R$, possibly singletons. Clearly the non-singleton intervals in $\mathbf P$ are at most countable; moreover by Lemma \ref{L_incastro_2}, the partition $\mathcal P((i-1)\e, \max \widehat{\mathcal L}_\beta, \min \widehat{\mathcal R}_\beta)$  can be restricted both to $\mathcal L_\alpha$ and to $\mathcal R_\alpha$; hence, denoting by $\{U_r\}_{r \in \N}$ the non-singleton elements of $\mathbf P$ contained in $L_\alpha$ and by $\{V_{r'}\}_{r' \in \N}$ the non-singleton elements of $\mathbf P$ contained in $R_\alpha$, we can write $L_\alpha, R_\alpha$ as
\begin{equation*}
\begin{split}
L_\alpha & = \Phi_k((i-1)\e) (\mathcal L_\alpha) = \bigg( \bigcup_{r \in \N} U_r \bigg) \cup \bigg( L_\alpha \setminus \bigcup_{r \in \N} U_r \bigg), \\
R_\alpha & = \Phi_k((i-1)\e) (\mathcal R_\alpha) = \bigg( \bigcup_{r' \in \N} V_{r'} \bigg) \cup \bigg( R_\alpha \setminus \bigcup_{r' \in \N} V_{r'} \bigg);
\end{split}
\end{equation*}
%
set also, for shortness:
\begin{equation*}
U :=  \bigcup_{r \in \N} U_r, \qquad V:=  \bigcup_{r' \in \N} V_{r'}.
\end{equation*}

Now observe that for $(\tau, \tau') \in L_\alpha \times R_\alpha$, setting
\begin{equation*}
w:= \Phi_k((i-1)\e)^{-1}(\tau), \qquad w' := \Phi_k((i-1)\e)^{-1}(\tau'),
\end{equation*}
it holds:
\begin{equation}
\label{E_tilde_pi_su_phi_0}
\tilde \pi(\tau, \tau') = 
\left\{ \begin{array}{ll}
\big[ \sigmarh(\tilde f, U_r) - \sigmarh(\tilde f, V_{r'}) \big]^+ & \text{if } \tau \in U_r, \text{ and } \tau' \in V_{r'}, \\ [.2em]
{\displaystyle \bigg[ \sigmarh(\tilde f, U_r) - \frac{d \tilde f}{d\tau}(\tau') \bigg]^+} & \text{if } \tau \in U_r, \text{ and } \tau' \in R_\alpha \setminus V, \\ [.9em]
{\displaystyle \bigg[ \frac{d \tilde f}{d\tau}(\tau) - \sigmarh(\tilde f, V_{r'}) \bigg]^+} & \text{if } \tau \in L_\alpha \setminus U, \text{ and } \tau' \in  V_{r'}, \\ [.9em]
{\displaystyle \bigg[ \frac{d \tilde f}{d\tau}(\tau) - \frac{d \tilde f}{d\tau}(\tau') \bigg]^+} & \text{if } \tau \in L_\alpha \setminus U, \text{ and } \tau' \in R_\alpha \setminus V.
\end{array} \right.
\end{equation}
Indeed, if $\max \widehat{\mathcal L}_\beta, \min \widehat{\mathcal R}_\beta$ have never interacted at time $(i-1)\e$, then by Proposition \ref{P_stessa_part}, Point \eqref{Point_3_stessa_part}, $w,w'$ have never interacted at time $(i-1)\e$ and thus
\begin{equation*}
\tilde \pi(\tau, \tau') =  \bigg[ \frac{d \tilde f}{d\tau}(\tau) - \frac{d \tilde f}{d\tau}(\tau') \bigg]^+
\end{equation*}
(by \eqref{E_def_pi} and the fact that $\mathcal K,\mathcal K'$ in \eqref{E_gg'} are singletons). In particular \eqref{E_tilde_pi_su_phi_0} holds.
\\
On the other hand, if $\max \widehat{\mathcal L}_\beta, \min \widehat{\mathcal R}_\beta$ have already interacted at time $(i-1)\e$, distinguish two cases:
\begin{enumerate}
\item if $\tcr(w), \tcr(w') \leq \tsp((i-1)\e, \max \widehat{\mathcal L}_\beta, \min \widehat{\mathcal R}_\beta)$, then by Proposition \ref{P_stessa_part}, Point \eqref{Point_1_stessa_part}, $w,w'$ have already interacted at time $(i-1)\e$ and
\begin{equation*}
\mathcal I((i-1)\e,\max \widehat{\mathcal L}_\beta, \min \widehat{\mathcal R}_\beta) = \mathcal I((i-1)\e,w,w'), \ 
\mathcal P((i-1)\e,\max \widehat{\mathcal L}_\beta, \min \widehat{\mathcal R}_\beta) = \mathcal P((i-1)\e,w,w'),
\end{equation*}
which implies \eqref{E_tilde_pi_su_phi_0} (remember that
\begin{equation*}
\tsp \big( (i-1)\e,w,w' \big) = \tsp \big( (i-1)\e, \max \widehat{\mathcal L}_\beta, \min \widehat{\mathcal R}_\beta \big),
\end{equation*}
since the intervals are not further partitioned by $\Pi_0$);
\item if one or both among $w,w'$ is created after $\tsp((i-1)\e, \max \widehat{\mathcal L}_\beta, \min \widehat{\mathcal R}_\beta)$, then, by Proposition \ref{P_stessa_part}, Point \eqref{Point_2_stessa_part}, $w,w'$ have never interacted at time $(i-1)\e$ and thus
\begin{equation*}
\tilde \pi(\tau, \tau') =  \bigg[ \frac{d \tilde f}{d\tau}(\tau) - \frac{d \tilde f}{d\tau}(\tau') \bigg]^+
\end{equation*}
(by \eqref{E_def_pi} and the fact that $\mathcal K,\mathcal K'$ in \eqref{E_gg'} are singletons). In particular \eqref{E_tilde_pi_su_phi_0} holds also in this case.
\end{enumerate}

We are now able to conclude the proof of the lemma as follows:
\begin{equation*}
\begin{split}
\big[ \sigmarh&(\tilde f, L_\alpha) - \sigmarh(\tilde f, R_\alpha) \big] \mathcal L^2(L_\alpha \times R_\alpha) \\
& = \iint_{L_\alpha \times R_\alpha} \bigg[\frac{d \tilde f}{d\tau}(\tau) - \frac{d \tilde f}{d\tau}(\tau')\bigg] d\tau d\tau' \\
& = \sum_{r, r' \in \N} \iint_{U_r \times V_{r'}} \bigg[\frac{d \tilde f}{d\tau}(\tau) - \frac{d \tilde f}{d\tau}(\tau')\bigg] d\tau d\tau' 
+ \sum_{r \in \N} \iint_{U_r \times (R_\alpha \setminus V)} \bigg[\frac{d \tilde f}{d\tau}(\tau) - \frac{d \tilde f}{d\tau}(\tau')\bigg] d\tau d\tau' \\
& \quad  + \sum_{r' \in \N} \iint_{(L_\alpha \setminus U) \times V_{r'}} \bigg[\frac{d \tilde f}{d\tau}(\tau) - \frac{d \tilde f}{d\tau}(\tau')\bigg] d\tau d\tau' 
+ \iint_{(L_\alpha \setminus U) \times (R_\alpha \setminus V)} \bigg[\frac{d \tilde f}{d\tau}(\tau) - \frac{d \tilde f}{d\tau}(\tau')\bigg] d\tau d\tau' \\
& = \sum_{r, r' \in \N} \mathcal L^2(U_r \times V_{r'}) \Big[\sigmarh(\tilde f, U_r) - \sigmarh(\tilde f, V_{r'}) \Big] 
+ \sum_{r \in \N} \mathcal L^1(U_r) \int_{R_\alpha \setminus V} \bigg[\sigmarh(\tilde f, U_r) - \frac{d \tilde f}{d\tau}(\tau')\bigg] d\tau' \\
& \quad  + \sum_{r' \in \N} \mathcal L^1(V_{r'}) \int_{L_\alpha \setminus U} \bigg[\frac{d \tilde f}{d\tau}(\tau) - \sigmarh(\tilde f, V_{r'}) \bigg] d\tau 
+ \iint_{(L_\alpha \setminus U) \times (R_\alpha \setminus V)} \bigg[\frac{d \tilde f}{d\tau}(\tau) - \frac{d \tilde f}{d\tau}(\tau')\bigg] d\tau d\tau' \\
& \leq \sum_{r, r' \in \N} \mathcal L^2(U_r \times V_{r'}) \Big[\sigmarh(\tilde f, U_r) - \sigmarh(\tilde f, V_{r'}) \Big]^+ 
+ \sum_{r \in \N} \mathcal L^1(U_r) \int_{R_\alpha \setminus V} \bigg[\sigmarh(\tilde f, U_r) - \frac{d \tilde f}{d\tau}(\tau')\bigg]^+ d\tau' \\
& \quad  + \sum_{r' \in \N} \mathcal L^1(V_{r'}) \int_{L_\alpha \setminus U} \bigg[\frac{d \tilde f}{d\tau}(\tau) - \sigmarh(\tilde f, V_{r'}) \bigg]^+ d\tau 
+ \iint_{(L_\alpha \setminus U) \times (R_\alpha \setminus V)} \bigg[\frac{d \tilde f}{d\tau}(\tau) - \frac{d \tilde f}{d\tau}(\tau')\bigg]^+ d\tau d\tau' \\
& \!\!\! \overset{\eqref{E_tilde_pi_su_phi_0}}{\leq} \sum_{r, r' \in \N} \iint_{U_r \times V_{r'}} \tilde \pi(\tau, \tau') d\tau d\tau' 
+ \sum_{r \in \N} \iint_{U_r \times (R_\alpha \setminus V)} \tilde \pi(\tau, \tau') d\tau d\tau' \\
& \quad  + \sum_{r' \in \N} \iint_{(L_\alpha \setminus U) \times V_{r'}} \tilde \pi(\tau, \tau') d\tau d\tau' 
+ \iint_{(L_\alpha \setminus U) \times (R_\alpha \setminus V)} \tilde \pi(\tau, \tau') d\tau d\tau' \\
& = \iint_{L_\alpha \times R_\alpha} \tilde \pi(\tau, \tau') d\tau d\tau',
\end{split}
\end{equation*}
which is what we wanted to prove. 
\end{proof}

\noindent \textbf{Conclusion of the proof of Proposition \ref{P_increasing}.}
In the previous lemma we proved inequality \eqref{E_fund_ineq} for the elements $\alpha \in D$ of the tree whose last component is equal to $0$. Now we use this fact to prove \eqref{E_fund_ineq} for any $\alpha \in D$. We proceed by (inverse) induction on the tree.

If $\alpha$ is a leaf of the tree, then, by definition, the last component of $\alpha$ is equal to zero, and thus Lemma \ref{L_estim_phi_zero_psi_alpha} applies.

If $\alpha$ is not a leaf, then 
\begin{equation*}
\widehat{\Psi}_\alpha = \widehat{\Psi}_{\alpha 0} \cup \widehat{\Psi}_{\alpha 1}  \cup \widehat{\Psi}_{\alpha 2}  \cup \widehat{\Psi}_{\alpha 3}
\end{equation*}
and thus
\begin{equation*}
L_\alpha \times R_\alpha = 
\Big( L_{\alpha 0} \times R_{\alpha 0} \Big) \cup
\Big( L_{\alpha 1} \times R_{\alpha 1} \Big) \cup
\Big( L_{\alpha 2} \times R_{\alpha 2} \Big) \cup
\Big( L_{\alpha 3} \times R_{\alpha 3} \Big).
\end{equation*}
The estimate \eqref{E_fund_ineq} holds on $L_{\alpha 0} \times R_{\alpha 0}$ by Lemma \ref{L_estim_phi_zero_psi_alpha}, while it holds on $L_{\alpha a} \times R_{\alpha a}$, $a = 1,2,3$, by inductive assumption. Hence we can write
\begin{equation*}
\begin{split}
\big[ \sigmarh(\tilde f, L_\alpha) - \sigmarh(\tilde f, R_\alpha) \big] \mathcal L^2(L_\alpha \times R_\alpha) 
& = \iint_{L_\alpha \times R_\alpha} \bigg[\frac{d \tilde f}{d\tau}(\tau) - \frac{d \tilde f}{d\tau}(\tau')\bigg] d\tau d\tau' \\
& = \sum_{a = 0}^3 \iint_{L_{\alpha a} \times R_{\alpha a}} \bigg[\frac{d \tilde f}{d\tau}(\tau) - \frac{d \tilde f}{d\tau}(\tau')\bigg] d\tau d\tau' \\
& = \sum_{a = 0}^3  \big[ \sigmarh(\tilde f, L_{\alpha a}) - \sigmarh(\tilde f, R_{\alpha a}) \big] \mathcal L^2(L_{\alpha a} \times R_{\alpha a}) \\
& \leq \sum_{a = 0}^3 \iint_{L_{\alpha a} \times R_{\alpha a}} \tilde \pi(\tau, \tau') d\tau d\tau' \\
& = \iint_{L_{\alpha a} \times R_{\alpha a}} \tilde \pi(\tau, \tau') d\tau d\tau'. 
\end{split}
\end{equation*}
As already observed, for $\alpha = \emptyset$, we get inequality \eqref{E_delta_sigma}, thus concluding the proof of the proposition. 
\end{proof}


We can finally use Lemma \ref{L_tilde_q} and Proposition \ref{P_increasing} to prove Theorem \ref{T_increasing}.

\begin{proof}[Proof of Theorem \ref{T_increasing}]

It holds

\begin{equation*}
\begin{split}
- \sum_{m \in \Z} &\iint_{J_m^L \times J_m^R} \mathfrak q_k((i-1)\e) d\tau d\tau' \\
& \leq  \sum_{m \in \Z} \iint_{J_m^L \times J_m^R} \Big[- \mathfrak q_k((i-1)\e) + \tilde{\mathfrak q}_k(\tau, \tau') \Big]d\tau d\tau'
- \sum_{m \in \Z} \iint_{J_m^L \times J_m^R} \tilde{\mathfrak q}_k(\tau, \tau') d\tau d\tau' \\
& \leq 
\const \sum_{m \in \Z} \Atrans(i\e,m\e) \mathcal L^1(J_m) - \sum_{m \in \Z} \iint_{J_m^L \times J_m^R}  \tilde{\mathfrak q}_k(\tau, \tau') d\tau d\tau' \\
& \leq -\sum_{m \in \Z} \Aquadr_k(i\e,m\e) + \const \TV(u(0))  \sum_{m \in \Z} \Atrans(i\e,m\e) \\
& \leq -\sum_{m \in \Z} \Aquadr_k(i\e,m\e) + \const \TV(u(0))  \sum_{m \in \Z} \mathtt A(i\e,m\e). \qedhere
\end{split}
\end{equation*}
\end{proof}

\newpage

\appendix

\section{Real analysis results}
\label{S_appendix}

In this section collect some results about convex envelopes of continuous functions and slopes of secant lines; these results are frequently used in the paper. The statements related to convex envelopes are already proven in \cite{bia_03}, \cite{bia_mod_13}, \cite{bia_mod_14}, with some minimal variations, while the results regarding the slopes of the secant lines will be explicitly proved.
%

\subsection{Convex envelopes}
\label{Ss_convex_env_app}

We recall here the notion of \emph{convex envelope} of a continuous function $g: \R \to \R$ and we state some results about convex envelops.

\begin{definition}
\label{convex_fcn}
Let $g: \R \to \R$ be continuous and $[a,b] \subseteq \R$. We define \emph{the convex envelope of $g$ in the interval $[a,b]$} as
\[
\conv_{[a,b]} g (u) := \sup\bigg\{h(u) \ \Big| \ h: [a,b] \to \R \text{ is convex and } h \leq g\bigg\}.
\]
\end{definition}

A similar definition holds for \emph{the concave envelope of $g$ in the interval $[a,b]$} denoted by ${\displaystyle \conc_{[a,b]} g}$. All the results we present here for the convex envelope of a continuous function $g$ hold, with the necessary changes, for its concave envelope.

\begin{lemma}
In the same setting of Definition \ref{convex_fcn}, ${\displaystyle \conv_{[a,b]} g}$ is a convex function and ${\displaystyle \conv_{[a,b]} g(u) \leq g(u)}$ for each $u \in [a,b]$.  
\end{lemma}

The proof is straightforward.

The following theorem is classical and provides a description of the regularity of the convex envelope of a given function $g$. For a self contained proof (of a bit less general result), see Theorem 2.5 of \cite{bia_mod_13}.
 
\begin{theorem}
\label{convex_fundamental_thm}
Let $g: [a,b] \to \R$ be a Lipschitz function. Then:
\begin{enumerate}
\item \label{convex_fundamental_thm_1} the convex envelope ${\displaystyle \conv_{[a,b]} g}$ of $g$ in the interval $[a,b]$ is Lipschitz on $[a,b]$ and 
\begin{equation*}
\Lip \Big( \conv_{[a,b]}g \Big) \leq \Lip(g);
\end{equation*} 
\item \label{convex_fundamental_thm_2} if $g \in C^1([a,b])$, then ${\displaystyle \conv_{[a,b]} g} \in C^1([a,b])$ and, for any point $u \in (a,b)$ such that $\displaystyle g(u) = \conv_{[a,b]}g(u)$, it holds
\[
\frac{d}{du}g(u) = \frac{d}{du}\conv_{[a,b]}g(u);
\]
\item \label{convex_fundamental_thm_3} if $g \in C^{1,1}([a,b])$, then $\displaystyle \conv_{[a.b]} g \in C^{1,1}([a,b])$ and 
\begin{equation*}
\Lip \bigg( \frac{d}{du}{\displaystyle \conv_{[a,b]} g} \bigg) \leq \Lip \bigg( \frac{dg}{du} \bigg).
\end{equation*}

\end{enumerate}
\end{theorem}

By "$C^1([a,b])$" we mean that $\displaystyle \conv_{[a,b]}g$ is $C^1$ on $(a,b)$ in the classical sense and that in $a$ (resp. $b$) the right (resp. the left) derivative exists. 

We now state some useful results about convex envelopes, which we frequently use in the paper.

\begin{proposition}
\label{differenza_vel}
Let $g: \R \to \R$ be $C^1$ and $a < \bar{u} < b$. Then
\begin{enumerate}
\item for each $u_1, u_2 \in [a, \bar{u}]$, $u_1 < u_2$,
\begin{align*}
\bigg(\frac{d}{du}\conv_{[a,\bar{u}]}g\bigg)(u_2) - \bigg( \frac{d}{du}\conv_{[a,\bar{u}]}g \bigg)(u_1) \geq \bigg(\frac{d}{du}\conv_{[a,b]}g\bigg)(u_2) - \bigg(\frac{d}{du}\conv_{[a,b]}g\bigg)(u_1);  
\end{align*}
\item for each $u_1, u_2 \in [\bar{u},b]$, $u_1 < u_2$,
\begin{align*}
\bigg(\frac{d}{du}\conv_{[\bar{u},b]}g\bigg)(u_2) -\bigg(\frac{d}{du}\conv_{[\bar{u},b]}g\Big)(u_1) \geq \bigg(\frac{d}{du}\conv_{[a,b]}g\bigg)(u_2) - \bigg(\frac{d}{du}\conv_{[a,b]}g\bigg)(u_1),
\end{align*}
\end{enumerate}
where the derivative in the endpoints of the intervals are in the sense of right/left derivative. 
\end{proposition}

%

\begin{proof}
See Proposition 2.10 of \cite{bia_mod_13}.
\end{proof}

\begin{proposition}
\label{P_diff_vel_proporzionale_canc}
Let $g$ be a $C^{1,1}$ function, let $a < \bar u < b$. Then
\begin{equation*}
\bigg(\frac{d}{du}\conv_{[a,\bar u]} g\bigg)(\bar u-) - \bigg(\frac{d}{du} \conv_{[a,b]}g\bigg)(\bar u) 
\leq \Lip(g') (b - \bar u).
\end{equation*}
\end{proposition}

\begin{proof} See Proposition 2.15 of \cite{bia_mod_13}.
\end{proof}

\begin{proposition}
\label{P_estim_diff_conv}
Let $g,h : \R \longrightarrow \R$ be $C^1$ functions. Let $a,b \in \R$, $a<b$. Then it holds
\begin{equation*}
\bigg\| \frac{d}{du} \conv_{[a,b]}g  - \frac{d}{du} \conv_{[a,b]}h \bigg \|_{\infty} \leq 
\bigg\| \frac{dg}{du}  - \frac{dh}{du} \bigg \|_{\infty}, \quad
\bigg\| \frac{d \conv_{[a,b]}g}{d\tau} - \frac{d \conv_{[a,b]} h}{d \tau}\bigg\|_1 \leq \bigg\| \frac{dg}{d\tau} - \frac{dh}{d\tau}\bigg\|_1.
\end{equation*}
\end{proposition}
\begin{proof}
For the first estimate see Proposition 2.12 of \cite{bia_mod_14}, while for the second one see Lemma 3.1 of \cite{bia_03}.
\end{proof}

\subsection{Slopes of secant lines}

We now state two results related to the slope of the secant line of a function $g$ between two given points $a \leq b$. Their proofs are easy exercises. Using the language of Hyperbolic Conservation Laws, we will call this slope the \emph{Rankine-Hugoniot speed given by the map $g$ to the interval $[a,b]$}.

\begin{proposition}
\label{P_ri_lip}
Let $g: \R \to \R$ be a $C^{1,1}$ function and let $a \in \R$. Then the map 
\begin{equation*}
x \mapsto 
\begin{cases}
\dfrac{g(x) - g(a)}{x-a}, & \text{if } x \neq 0, \\
g'(a), & \text{if } x = 0
\end{cases}
\end{equation*}
is Lipschitz on $\R$, with Lipschitz constant equal to $\Lip(g')$.
\end{proposition}

\begin{proposition}
\label{P_diff_ri}
Let $g: \R \to \R$ be a $C^{1,1}$ function, let $[a,b] \subseteq \R$, $\bar u \in [a,b]$ such that ${\displaystyle \conv_{[a,b]} g}(\bar u) = g(\bar u)$. Then for any $u \in [a,b]$,
\begin{itemize}
\item if $u \in [a,\bar u]$, then 
\begin{equation*}
\sigmarh(g, [u,\bar u]) \leq \sigmarh(g, [u,b]);
\end{equation*}
\item if $u \in [\bar u, b]$, then
\begin{equation*}
\sigmarh(g, [\bar u, u]) \geq \sigmarh(g, [a,u]).
\end{equation*}
\end{itemize}
%
%
%
\end{proposition}

%
%
%

\end{document}